%% file: draft.tex
\def\editmode{0}

\def\bibfilenames{bibman_refs}

\def\spsformat{0}
\def\singlenarrowcol{0} 

\if\singlenarrowcol1

\documentclass[onecolumn]{IEEEtran}
\usepackage[pass]{geometry}
\newgeometry{textwidth=252pt, textheight=672pt}

\IEEEoverridecommandlockouts
\else

\if\spsformat1

\documentclass{article}
\usepackage{spconf}

\usepackage{algorithm}
\usepackage{algpseudocode}

\else

 \documentclass[journal]{IEEEtran}

\fi
\fi

\input{include_v6.tex}
\usepackage{xcolor}
\usepackage{algorithmic}
\usepackage[ruled,vlined,resetcount,noend,linesnumbered]{algorithm2e}
\usepackage{balance}

\usepackage{array}  

\usepackage{pdfpages}

\usepackage{hyperref} 

\newenvironment{changes}{\color{black}}{}
\newenvironment{changesv2}{\color{black}}{}
\newenvironment{changesv3}{\color{black}}{}
\newcommand{\arevtwo}[1]{\textcolor{black}{#1}} 
\newcommand{\arevthree}[1]{\textcolor{black}{#1}} 

\def\journal{1}

\usepackage{environ}

\if\journal0    
    \NewEnviron{journalonly}{%
    }
    \NewEnviron{conferenceonly}{%
    \BODY
    }
\else
    \if\journal1
        \NewEnviron{journalonly}{%
        \BODY
        }   
        \NewEnviron{conferenceonly}{%
        }  
    \else
        \if\journal2
        \newenvironment{journalonly}{\color{darkyellow}}{}
        
        \fi
    \fi
\fi

\def\extended{1} 

\if\extended0    
    \NewEnviron{extendedonly}{%
    }
    \NewEnviron{nonextendedonly}{%
    \BODY
    }
\else
    \if\extended1
        \NewEnviron{extendedonly}{%
        \BODY
        }   
        \NewEnviron{nonextendedonly}{%
        }  
    \else
        \if\extended2
        \newenvironment{extendedonly}{\color{violet}}{}
        \newenvironment{nonextendedonly}{\color{teal}}{}
        \fi
    \fi
\fi

\def\intermediatesteps{0}
\if\intermediatesteps1
    \NewEnviron{intermed}{%
        \color{purple}
        \BODY
    }
    \newcommand{\jumpline}{\\}
    \newcommand{\alignchar}{&}
    
\else
    \NewEnviron{intermed}{%
    }
    \newcommand{\jumpline}{}
    \newcommand{\alignchar}{}
    
\fi

\renewcommand{\arev}[1]{{#1}} 

\begin{document}

\title{Path Planning for Aerial Relays \\ via Probabilistic Roadmaps
}

\if\spsformat1
    \name{Pham Q. Viet as Daniel Romero}

    Author(s) Name(s)\thanks{Thanks to XYZ agency for funding.}}
    \address{Author Affiliation(s)}
\else
    \author{
        Pham Q. Viet and Daniel Romero
        \\
        Dept. of Information and Communication Technology,
        University of Agder, Grimstad, Norway.\\
        Email:\{viet.q.pham,daniel.romero\}@uia.no.
    }
\fi

\maketitle

\begin{abstract}
    \begin{changesv2}
        Autonomous uncrewed aerial vehicles (UAVs) can be utilized as aerial relays
        to serve users far from  terrestrial infrastructure. Unfortunately, existing algorithms for aerial relay path planning cannot  accommodate
        general flight constraints or channel models. This is  required in practice due to connectivity constraints,
        the presence of obstacles (e.g.
        buildings), and regulations.  This paper proposes a framework  that overcomes
        these  limitations by spatially discretizing the flight region. To cope with
        the resulting exponential growth in complexity, the framework  adopts
        a probabilistic roadmap approach, where a shortest path is found through a
        graph of randomly generated states. To attain high optimality with
        affordable complexity, the probability distribution used to generate these
        states is designed based  on heuristic path planners with theoretical
        guarantees. The  algorithms derived in this framework not only overcome the main limitations of existing schemes but also entail smaller computational complexity.
        Extensive  theoretical and numerical results corroborate the merits of the proposed approach.
    \end{changesv2}

\end{abstract}

\newcommand{\journalold}[1]{\begin{journalonly}#1\end{journalonly}} 
\input{notation.tex}

\newcommand\blfootnote[1]{%
    \begingroup
    \renewcommand\thefootnote{}\footnote{#1}%
    \addtocounter{footnote}{-1}%
    \endgroup
}

\begin{keywords}
    Aerial relays, path planning, probabilistic roadmaps, aerial
    communications.    \blfootnote{    This work has been funded by the IKTPLUSS grant 311994 of the Research
        Council of Norway.}

\end{keywords}

\section{Introduction}
\label{sec:intro}
\begin{changes}
    \input{intro.tex}
\end{changes}

\section{Related Work}
\label{sec:related_work}
\begin{changesv2}
    \begin{bullets}%
        \blt[intro]The usage of UAVs in communications has attracted extensive research efforts. The problems considered in the literature can be classified into three main categories:
        \begin{bullets}%
            \blt ABS placement,
            \blt path planning for data dissemination, and
            \blt relay path planning.
        \end{bullets}
        %
        \blt[aerial base station placement]In \emph{ABS placement}~\cite{viet2022introduction}, UAVs with onboard base stations hover at static locations to serve a typically large number of users; see~\cite{romero2024placement,chen2021relay} and the references therein. The problem here is to determine the 3D locations of the ABSs.


        \blt[path planning for data dissemination]In \emph{path planning for data dissemination}, the UAVs receive data at a certain location, fly to another location, and then transmit the data; see e.g.~\cite{asano2021delay} and the references therein.
        \blt[relay path planning] In \emph{relay path planning}, the UAVs establish  links between terminals, possibly in multiple hops.  The problem is to plan their paths.
        \begin{bullets}%
            \blt[number of relays]Many works focus on using a single UAV~\cite{wang2018power,zeng2016relaying,jiang2019trajectory,sun2021trajectory,chen2020trajectory,zhang2017trajectory} whereas others can accommodate multiple UAVs~\cite{zhang2022cooperative,zhang2018multi,liu2021relaying,ghazzai2018dual,lee2022trajectory,yanmaz2022positioning,yanmaz2023joint,yanmaz2024dynamic}.
            \blt[channel]Among these works, most consider free-space propagation~\cite{wang2018power,zeng2016relaying,jiang2019trajectory,sun2021trajectory,chen2020trajectory,zhang2022cooperative,zhang2018multi,liu2021relaying,yanmaz2022positioning,yanmaz2023joint,yanmaz2024dynamic}, but fading with arbitrary path loss exponents~\cite{zhang2017trajectory} and empirical air-to-ground channel models~\cite{ghazzai2018dual,lee2022trajectory} have also been considered. 
            \blt[approaches]Existing schemes adopt one of the following three approaches:
            \begin{bullets}%
                \blt[list of approaches]%
                \begin{bullets}%
                    \blt[non-linear opt](i) non-linear optimization over continuous variables that represent the spatial coordinates of all UAVs~\cite{wang2018power,zeng2016relaying,jiang2019trajectory,sun2021trajectory,chen2020trajectory,zhang2022cooperative,zhang2018multi,liu2021relaying,zhang2017trajectory};
                    \blt[mixed integer LP](ii) mixed integer linear programming (MILP) to determine the order in which a set of locations are visited following straight lines~\cite{ghazzai2018dual,lee2022trajectory};
                    \blt[steiner tree]and (iii) Steiner tree problem heuristics, which approximately minimize the number of required relays per time slot~\cite{yanmaz2022positioning,yanmaz2023joint,yanmaz2024dynamic}.
                \end{bullets}%
                \blt[limitations]Table~\ref{tab:literature} summarizes the limitations of these schemes.

            \end{bullets}%

        \end{bullets}%

        \blt[other works with uavs + PR] It is worth noting that there have
        been  works where PR has been applied to UAV path planning (see
        e.g. \cite{arista2021indoor} and references therein) but, to the best of our knowledge, never for communications.
        \blt[\ra single uav] Furthermore, algorithms with spatial discretization have been considered for   UAV communications (e.g. to plan a path through
        coverage areas~\cite{yang2019maps}) but never for relay path planning.%
    \end{bullets}%
\end{changesv2}%

\section{The Path Planning Problem}
\label{sec:problem}




\cmt{model}
\begin{bullets}%
    \blt[``stage'']Consider a spatial region $\region\subset\rfield^3$ and let $\oobregion\subset\region$ denote the set of points of $\region$ above the ground and outside any building or obstacle. For simplicity, it is assumed that
    \begin{changesv2}
        $[\locx,\locy,\locz]\transpose\in \oobregion$
        whenever $[\locx,\locy,\locz']\transpose\in \oobregion$ for some $\locz'<\locz$, which essentially means that the buildings or obstacles contain no holes or parts that stand out.
    \end{changesv2}
    %

    \blt[``characters'']
    \begin{bullets}%
        \blt[UAVs]A total of $\numuav$ aerial relays are deployed to establish a link between
        \blt[BS]a base station (BS) at location $\locbs\define[\locxbs,\locybs,\loczbs]\transpose\in\region$ and
        \blt[UE]a collection of $\numues$ users (UEs).
        \begin{changesv2}
            To simplify the exposition, it will be initially assumed that $\numues=1$; the case $\numues>1$ is addressed  in Sec.~\ref{sec:multi-ues}. The UE trajectory is denoted by the function $\locue(\cdot)$, defined by $t \mapsto \locue(t)\define [\locxue(t),\locyue(t),\loczue(t)]\transpose \in \region$, $t\geq 0$. Note that $\locue(\cdot)$ refers to a function whereas $\locue(t)$ is the vector that results from evaluating function $\locue(\cdot)$ at $t$.
        \end{changesv2}
    \end{bullets}
\end{bullets}%

\begin{bullets}%
    \blt[feasible region \ra min/max height, buildings...]
    Let $\flyregion\subset\oobregion$ be the set of spatial locations where the UAVs can fly. This is typically  determined by regulations (e.g. the minimum and maximum allowed altitudes,
    no-fly zones, and so on) and  other operational constraints.
    \blt[UAV position]The position of the $\induav$-th UAV at
    time $t$ is represented as $\loc_\induav(t)\in\flyregion$
    \blt[conf points] and the positions of all UAVs at time
    $t$ are collected into the $3 \times \numuav$ matrix
    $\confpt(t) \define [\loc_1(t),\ldots, \loc_\numuav(t)]$,
    referred to as the \emph{configuration
        point} (CP) at time $t$~\cite{kavraki1996roadmaps}. The set of all matrices
    whose columns are in $\flyregion$ is the
    so-called \emph{configuration space} (Q-space) and will
    be denoted as $\confspace$.
    \blt[take-off location]\begin{changesv2}The UAVs collectively follow a
        trajectory $\trajectory$, which is a function
        of the form
        $t \mapsto \confpt(t)\in\confspace$, $t\geq 0$.
    \end{changesv2}
    The  take-off locations of the UAVs are collected in matrix
    $\confpt_0 \define \confpt(0)$ and the maximum speed is $\maxuavspeed$.
\end{bullets}%

\subsection{Communication Model}
\label{sec:communication_model}
\begin{bullets}    %
    \blt[communication]
    \begin{bullets}%
        \blt[downlink] The targeted link must convey information
        in both ways but, to simplify the
        exposition, the focus here will be on the downlink.
        \blt[relay model]There, the signal transmitted by the BS is first
        decoded and retransmitted by UAV-1. The signal transmitted by UAV-1
        is decoded and retransmitted by UAV-2 and so on, until the UE
        receives the signal retransmitted by UAV-$\numuav$.
        \begin{changesv2}

            \blt[assumption]  For the considered schemes,  this communication can be
            implemented in many ways. However, the following assumption must be
            (at least approximately) satisfied:

            \emph{(as)~~the capacity of the channel between  two terminals does not depend on the locations of the other terminals.}

            Formally, if $\loc_{\induav}$ denotes the location of UAV-$\induav$ at the time of transmission, (as) requires that the capacity $\capacity(\loc_{\induav-1},\loc_\induav)$ between UAV-$(\induav-1)$ and UAV-$\induav$ does not depend on $\loc_{\induav'}$ for $\induav'\neq\induav, \induav-1$. Equivalently, the interference that UAV-$\induav$ receives from UAV-$\induav'$ should be much smaller than the signal it receives from UAV-$(\induav-1)$.
            \blt[interference from other transmitters]Note, nonetheless, that interference from any other transmitter in the environment can be readily included in  $\capacity(\loc_{\induav-1},\loc_\induav)$ without violating (as).
            
            \blt[approaches] The degree to which (as) holds depends on the allocation of communication resources (e.g. space, \arevthree{time}, frequency, or code) to the UAVs \footnote{\label{fnote:interference}\arevthree{Inter-hop interference can also be mitigated via directional antennas or beamforming, power control, and interference cancellation at the receivers.}}.
            \begin{bullets}%
                \blt[all orthogonal]If, for example, $\numres=\numuav+1$ orthogonal resources are available, each UAV and the BS can use  a different one and, as a result, (as) holds exactly. This case is not unrealistic in post-disaster scenarios where the terrestrial infrastructure is not operational and, therefore, available bandwidth abounds.
                \begin{changesv3}%
                    Even if $\numres<\numuav+1$, (as) can approximately hold if only nearby UAVs use orthogonal resources.
                \end{changesv3}%
                \blt[$\numres$ resources]More generally, when $\numres$ orthogonal resources are available, then UAV-$\induav$ uses the same resource as UAV-$(\induav+i \numres)$, $i=\pm1,\pm2,\ldots$. Thus, UAV-$\induav$ receives interference from UAV-$(\induav+i \numres-1)$, $i=\pm 1,\pm 2,\ldots$. These UAVs will generally be far away from UAV-$\induav$ for moderate values of $\numres$ and, therefore, the assumption will approximately hold.
            \end{bullets}
        \end{changesv2}


        \blt[rates]
        \begin{bullets}%
            \blt[min uav rate]Besides the data to be relayed
            towards the UE, each UAV consumes a rate
            $\minuavrate$ for command and control.
            \blt[rate UAV-i]This means that the useful rate between
            the BS and  UAV-$\induav$ for a generic CP
            $\confpt \define [\loc_1,\ldots, \loc_\numuav]$ can be recursively obtained~as
            \begin{equation}
                \label{eq:recrate}
                \rate_{\induav}(\confpt) =
                \max(0,\operatormin\left(\rate_{\induav-1}(\confpt)-\minuavrate,\capacity(\loc_{\induav-1},
                        \loc_{\induav})\right)).
            \end{equation}
            %
            \blt[UE]Similarly, the achievable rate of the UE when it is at $\locue$ is
            \begin{equation}%
                \uerate(\confpt,\locue) = \max(0,\operatormin\left(\rate_{\numuav}(\confpt) -\minuavrate,\capacity(\loc_{\numuav}, \locue)\right)).
            \end{equation}%
            The second argument in $\uerate$ will be omitted when it is clear from the context.
            \blt[connectivity requirement]
            \begin{bullets}%
                \blt[UAVs] Throughout the trajectory, the UAVs must have connectivity with the BS, meaning that $\rate_{\induav}(\confpt(t))\geq \minuavrate~\forall \induav,t$.%
            \end{bullets}%
        \end{bullets}%
    \end{bullets}%

    \subsection{Capacity Maps}
    \label{sec:capmap}
    \begin{changesv2}%
        \begin{bullets}%
            \blt[overview] A
            \emph{capacity map} is a function $\capacity$ that maps a pair of locations $\loc$ and $\loc'$ to
            the capacity  $\capacity(\loc,\loc')$ of the channel between them.
            %
            \blt[capacity(gain)]This capacity can be expressed in terms of the channel gain $\gain(\loc,\loc')$ between $\loc$ and $\loc'$ as
            \begin{align}
                \capacity(\loc,\loc') = \capacitysnr(\gain(\loc,\loc')) \define \bandwidth \log_2(1+\gain(\loc,\loc')/\noisepower),
            \end{align}
            where
            $\bandwidth$ is the bandwidth and $\noisepower$ captures both interference and noise power.

            \blt[channel-gain radio maps]There are many approaches to obtain function $\gain$.
            \begin{bullets}%
                \blt[ray-tracing]One can, for example, rely on 3D terrain/city models together with ray-tracing algorithms or other simulation software.
                \blt[tomographic channel-gain model]If such models are not available, one may construct a channel-gain radio map using measurements~\cite{romero2022cartography}.
                \begin{bullets}%
                    \blt[description]To this end, the most common approach builds upon the so-called \emph{tomographic model} \cite{patwari2008nesh,romero2018blind}, which prescribes that
                    \begin{align}                                        \gain(\loc,\loc')=
                        \gainpl(\dist(\loc,\loc'))\cdot
                        10^{\gaindbabs(\loc,\loc')/10},
                    \end{align}
                    where $\dist(\loc,\loc')\define\|\loc-\loc'\|$,  $\gainpl(\dist)$ is the \emph{path loss} of a link with distance $\dist$, and $\gaindbabs(\loc,\loc')$ is the gain due to absorption. Specifically,
                    \begin{bullets}
                        \blt[PL]\begin{equation}
                            \gainpl\left(\dist\right) \define \txpower\txgain\rxgain\left(\frac{\wavelength}{4\pi\dist}\right)^{\pathlossexp},
                        \end{equation}
                        where $\txpower$, $\txgain$, $\rxgain$, $\wavelength$, and $\pathlossexp$ are respectively the transmit power, transmit gain, receive gain, wavelength, and pathloss exponent.
                        \blt[abs]On the other hand,        $\gaindbabs(\loc,\loc')$ is given by the line integral~\cite{patwari2008nesh}
                        \begin{align}
                            \gaindbabs(\loc,\loc') \define
                            -\frac{1}{\sqrt{\dist(\loc,\loc')}}\int_{\loc}^{\loc'} \slf(\locaux) d\locaux,
                        \end{align}
                        where $\slf$ is the so-called \emph{spatial loss field} (SLF), which quantifies absorption at each spatial location. Clearly, $\slf(\locaux)$ is 0 if there is no obstacle at $\locaux$.

                    \end{bullets}
                \end{bullets}
                \blt[tmia]
                \begin{bullets}
                    \blt[def]The proposed algorithm can accommodate  \emph{arbitrary capacity maps}, not necessarily based on ray-tracing or the tomographic model. However,
                    some theoretical results will rely on a limit case of tomographic maps referred to as the \emph{tomographic map with infinite absorption} (TMIA), where $\slf(\locaux)\in\{0,\infty\}~\forall \locaux$.
                    Clearly, in this case,
                    \begin{equation}
                        \capacity(\loc,\loc') = \begin{cases}
                            \capacitydist(\dist(\loc,\loc')), & \text{if there is LOS between $\loc$ and $\loc'$} \\
                            0                                 & \text{otherwise},
                        \end{cases}
                        \label{eq:tmia}
                    \end{equation}
                    where $\capacitydist(\cdot)\define \capacitysnr(\gainpl(\cdot))$.
                    \blt[intuition]This corresponds to an environment where the obstacles (e.g. buildings) are opaque to  radio waves, an accurate assumption in high frequencies; e.g. in mmWave bands~\cite{romero2024placement}.
                    %
                    %
                    \blt[relation to Chen] A map of this kind has been used in the related literature; see e.g. \cite{chen2017map}.
                    \blt[convenient theoretically]It is convenient for theoretical derivations since it essentially constitutes a worst-case map.%
                \end{bullets}%
            \end{bullets}%
            \begin{changesv3}
            \begin{myremark}\label{rem:fading}
                The actual channel between two spatial locations can be worse than that predicted by the adopted capacity map. It may be thus convenient to include a safety margin in the computations to ensure connectivity. 

            \end{myremark}
            \end{changesv3}
        \end{bullets}%
    \end{changesv2}%
\end{bullets}%

\subsection{Problem Formulation}
\label{sec:problemformulation}
\cmt{problem formulation}
\begin{bullets}%
    \blt[Overview]\arev{This paper addresses the problem of designing the trajectory of the UAVs so that they can  serve the UE(s)}.
    \blt[Given] Given
    \begin{bullets}%
        \blt[BS loc]$\locbs$,
        \blt[initial UAV loc]$\confpt_0\in\confspace$,
        \blt[UE loc]$\locue(\cdot)$,
        \blt[num uavs]$\numuav$,
        \blt[capacity function ] $\capacity$, 
        \blt[min uav rate]$\minuavrate$, and
        \blt[max uav speed] $\maxuavspeed$,
    \end{bullets}%
    \blt[Requested]
    \begin{bullets}%
        \blt[trajectory that minimizes time to connect]the problem is to solve
        \begin{salign}[eq:trajproblem]
            \minimize_{\trajectory}~~&\funobj(\trajectory)\\
            \st~~&\confpt(t)\in \confspace~\forall t,~~\confpt(0)=\confpt_0\\
            &\rate_{\induav}(\confpt(t))\geq \minuavrate~\forall \induav,t\\
            &\| \dot \loc_\induav(t)\|\leq \maxuavspeed~\forall \induav,t,\label{eq:maxspeedconstr}
        \end{salign}
        where   the objective function $\funobj$
        \arevtwo{ is discussed next. Observe that the optimization variable $\trajectory$ is a  function of the continuous-time variable $t$.} %
        \blt[separation]Note also that \eqref{eq:trajproblem} does not
        enforce a minimum distance between UAVs. Such a
        constraint is omitted for simplicity but can be  accommodated in the proposed scheme.

        \arevtwo{The following objective functions will be considered:}
        \begin{itemize}%
            \item\textbf{Connection time.}
            \begin{bullets}%
                \blt[motivation]In many situations, it is desirable to establish connectivity between the UE and the BS as soon as possible. This is the case  when  time-sensitive information must be delivered in a short time, e.g. to  notify a user of a tsunami, earthquake, or military attack.
                \blt[objective] The goal is, therefore, to minimize the \emph{connection time}
                \begin{equation}
                    \funobj(\trajectory)=\ttc(\trajectory)\define\inf \{t~:~\uerate(\confpt(t),\locue(t))\geq
                    \minuerate\},
                    \label{eq:timeconnectobj}
                \end{equation}
                where  $\minuerate$ is the target rate.
                Note that, consistent with the standard convention for the
                infimum, $\ttc(\trajectory)=\infty$ if $\uerate(\confpt(t),\locue(t))<
                    \minuerate$  for all $t$.
                \blt[movement]
                \begin{changesv2}
                    Note also that  $\funobj(\trajectory)$ is not meaningful if the UE loses connectivity after the connection is established, which can happen if the UE moves. This renders this objective immaterial unless the UE is static.
                \end{changesv2}
            \end{bullets}%
            \item\textbf{Outage time.}
            \begin{bullets}%
                \blt[motivation] A natural objective when the UE is not static  is the time during which it has no connectivity~\cite{khuwaja2020effect}.
                \blt[objective]This motivates minimizing the \emph{outage time}
                \begin{align}
                    \funobj(\trajectory) = \int_0^{\horizon}
                    \funindicator[\uerate(\confpt(\timeInstant),\locue(\timeInstant))<\minuerate] d\timeInstant
                    \label{eq:outageobj}
                    ,
                \end{align}%
                where $\horizon$ is the time horizon \arev{and $\funindicator[.]$ was defined in Sec.~\ref{sec:intro}}.
            \end{bullets}%
            \item\textbf{Transferred data.}
            \begin{bullets}%
                \blt[motivation] \arev{In some applications, data may be relatively delay
                    tolerant. Thus, instead of minimizing outage time, one may be
                    interested in maximizing the total amount of data received by
                    the UE within a given time horizon  $\horizon$. This gives rise
                    to the objective function}
                \blt[objective]
                \begin{align}
                    \funobj(\trajectory) = -\int_0^{\horizon}
                    \uerate(\confpt(t),\locue(t)) dt
                    \label{eq:cumrateobj}
                    ,
                \end{align} %
                where  the minus sign is due to the fact that \eqref{eq:trajproblem} is a minimization problem.
            \end{bullets}%
        \end{itemize}%
        \blt[dependence on objective]%
        \begin{changesv2}%
            Clearly, the choice of  objective determines the optimal trajectory. An analysis of this influence is presented in Appendix~\ref{sec:influenceobj}.%
        \end{changesv2}%
        %
        %

    \end{bullets}%
\end{bullets}%

\section{Path Planning via Probabilistic Roadmaps}
\label{sec:introroadmap}
\cmt{overview}
\begin{bullets}%
    \blt[difficulty]Since Problem \eqref{eq:trajproblem} involves 
    optimization with respect to a trajectory, which comprises infinitely many CPs, the exact
    solution cannot generally be found by numerical means.
\end{bullets}%
\cmt{simplifications}%
\begin{bullets}%
    %
    \begin{changesv2}\blt[discretization]For this reason, both space and time will be discretized.
    \end{changesv2}

    \begin{bullets}%
        \blt[ fly grid]Specifically, the flight region
        $\flyregion$ is discretized into a regular 3D \emph{flight grid}
        $\grid\subset\flyregion\subset \rfield^3$, whose points are separated
        along the x, y, and z axes respectively by $\spacingx$, $\spacingy$,
        and $\spacingz$; see Fig.~\ref{fig:environment}. To simplify some expressions, it will be assumed that the take-off locations of the UAVs are in $\grid$.
        \blt[q-space]This spatial discretization also induces a grid $\confspacegrid$
        in the Q-space, which e.g. for $\numuav=2$ is given by $\confspacegrid\define
            \grid \times \grid$.

        \blt[trajectory]Regarding the time-domain discretization, the
        trajectory $\trajectory$ will be designed by first finding a CP
        sequence $\trajectorywps\define
            \{\confpt[0],\ldots,\confpt[\numwp-1]\}$ through the grid $\confspacegrid$.
        \blt[combined path]This sequence will be referred to as \emph{combined path}, whereas the  waypoint sequence $\loc_\induav[\indwp]$ that each individual UAV must follow will be referred to as a \emph{path}.
        \blt[feasible]Given $\trajectorywps$, the trajectory $\trajectory$ is recovered by interpolating the waypoints in $\trajectorywps$. A path  $\pathuavs$ will be said to be \emph{feasible} iff
        the associated trajectory $\trajectory$ is feasible.
    \end{bullets}
\end{bullets}

\cmt{algo overview \ra  probabilistic roadmaps}Having introduced the
discretization, the next step is to discuss how to obtain a feasible
waypoint sequence that attains a satisfactory objective value.
\begin{bullets}%
    \blt[motivate probabilistic roadmaps \ra high dimensionality]
    \begin{bullets}%
        \blt[UAV path planning]Conventional algorithms for planning
        the path of a single (non-relay) UAV create a graph whose nodes are the
        points of $\grid$ and where an edge exists between two nodes
        if the associated points are adjacent on the grid. In a 3D
        regular grid like $\grid$, each point  has typically 26 adjacent
        points, which renders the application of shortest-path
        algorithms on such a graph viable.
        However, the grid $\confspacegrid$ has exponentially many
        more points than $\grid$. For example, if $\grid$ is a (small)
        $10 \times 10 \times 10$ grid, then $\confspacegrid$ has
        $10^{3\numuav}$ points. Besides, since each of them has generally $27^\numuav-1$
        adjacent points, solving~\eqref{eq:trajproblem} via
        shortest-path algorithms is prohibitive.
    \end{bullets}%


    %
    \blt[briefly describe probabilistic roadmaps]
    \begin{bullets}%
        \blt[overview]To bypass this kind of  difficulties in other contexts,
        the seminal paper~\cite{kavraki1996roadmaps} proposed the PR algorithm,
        which consists of 3 steps: Step 1:~a node set
        $\prnodeset\subset\confspace$ with a much smaller number of nodes than
        $\confspacegrid$ is randomly generated. Step 2: the edge set
        $\predgeset\subset \prnodeset\times \prnodeset$ is constructed by
        connecting the nodes corresponding to CPs $\confpt$ and $\confpt'$ if
        (i) they are  nearest neighbors and (ii) it is possible to transition
        directly from  $\confpt$ to $\confpt'$. Step 3: a shortest path  is
        found on the graph with node set $\prnodeset$ and edge set~$\predgeset$.

        \blt[limitations] Unfortunately, \arevtwo{even though the number of CPs in $\prnodeset$ is significantly smaller than in $\confspacegrid$, it will still be prohibitive for the application at hand. This is because
            plain vanilla PR  generates the CPs uniformly at random over the configuration space and, thus, the the number of CPs  necessary
            to find a feasible (let alone satisfactory) path grows exponentially with the dimensionality of the configuration space.}

    \end{bullets}%
    \blt[proposed sampling: feasible trajectory]%
    \begin{bullets}%
        \blt[motivation]The main idea of this paper is to  modify PR to
        counteract this limitation:
        %
        %
        \blt[def: feasible trajectory]\arevtwo{instead of generating the CPs according to a uniform distribution in Step 1, a  special probability distribution is developed to ensure that a feasible waypoint sequence can be found by drawing a nearly minimal number of CPs.
        }

    \end{bullets}%
\end{bullets}%

\section{Path Planning for a Static UE}
\label{sec:static}

\begin{bullets}%
    \blt[Overview]This section addresses the path planning problem for a single static UE. The case of a single moving UE is addressed in Sec.~\ref{sec:moving}. 
    \begin{changesv2}
    Multiple UEs will be considered in Sec.~\ref{sec:multi-ues}.    
    \end{changesv2}%
    
    \blt[static special case]%
    \begin{changesv2}%
        Although a static UE is a special case of a moving UE where $\locue(t)=\locue~\forall t$, it is useful to first address the static scenario because it simplifies the presentation and yields an algorithm that attains greater optimality.
    \end{changesv2}%
    \blt[objective f]\begin{changesv2}
        It is also convenient to focus on the connection time objective \eqref{eq:timeconnectobj}. Minimizing the outage time \eqref{eq:outageobj} is equivalent since the UE is static. In turn, maximizing the transferred data is deferred to  Sec.~\ref{sec:moving} because the approach there
        is more suitable to enforce the time horizon in \eqref{eq:cumrateobj} as it  involves uniform time discretization.
    \end{changesv2}%
\end{bullets}%

\subsection{Planning the Tentative Path}
\label{sec:staticUE-tentative}

\begin{bullets}%
    \blt[Overview]
    \begin{changesv2}As indicated in Sec.~\ref{sec:introroadmap}, the proposed framework adapts PR by mainly modifying the probability distribution for CP generation. As detailed next, this distribution relies on  a heuristic that produces a feasible path. 
        \end{changesv2}


    \blt[definitions]
    \begin{bullets}%
        \blt[valid path]A trajectory $\trajectory$ is said to be \emph{valid} if it is feasible and attains a finite connection time $\ttc(\trajectory)$. Equivalently, a combined path $\pathuavs \define\{
            \confpt[0],\ldots,$ $\confpt[\numwp-1]\}$ is valid if it  is feasible and $\exists \indwp~:~\uerate(\confpt[\indwp])\geq \minuerate$.
        \blt[optimal path] A feasible path $\pathuavs$ on a grid $\confspacegrid$ is said to be \emph{optimal} if it attains the lowest connection time among all feasible paths  on $\confspacegrid$.
    \end{bullets}%
    \blt[overview]The heuristic  proposed in this section will be seen to produce a \emph{valid} (and sometimes even \emph{optimal}) path under general conditions.

    \blt[$\numuav=2$]
    \arevtwo{
        Interestingly, $\numuav=2$ UAVs often suffice to guarantee the existence of a valid path. This is the case e.g. if $\maxbheight$ and $\dist$  are not
        too large relative to $\minuerate$ and $\minuavrate$, where  $\maxbheight$ is
            the height of the highest obstacle and  $\dist\define\sqrt{
                    (\locxue-\locxbs)^2+(\locyue-\locybs)^2 }$ is the horizontal
            distance between the BS and the UE:}
    \begin{bullets}
        \blt[prop]
        \begin{myproposition}
            \arevtwo{          
            \thlabel{prop:feasibleexists}Let $\capacity$ be a TMIA map. Suppose that $\loc_\induav(0)=\locbs~\forall \induav$ and that the UAVs can fly
            above $\maxbheight$.             
            If
            $\maxbheight < \min(\loczue +            \capacitydist\inv(\minuerate),\loczbs+\capacitydist\inv(\minuerate+2\minuavrate))$
            and $\dist\leq \capacitydist\inv(\minuerate+\minuavrate)$,  then there
            exists a valid path with $\numuav=2$.
            }
        \end{myproposition}
        \begin{proof}
            Let $\locz\define \min(\loczue +            \capacitydist\inv(\minuerate),\loczbs+\capacitydist\inv(\minuerate+2\minuavrate))$ and suppose that $\numuav=2$.
            It is easy to show that if UAV-1 flies to $\loc_1\define[\locxbs,\locybs,\locz]\transpose$ and UAV-2 flies
            first to $\loc_1$ and later to $\loc_2\define[\locxue,\locyue,\locz]\transpose$, the resulting trajectory  $\trajectory$ is feasible and $\ttc(\trajectory)< (\locz-\loczbs+\dist)/\maxuavspeed<+\infty$.
        \end{proof}

        \blt[remarks]\arevtwo{Note that this sufficient condition enjoys great generality because it relies on a worst-case map.}
        %
    \end{bullets}%
    \blt[K=2]\arevtwo{Therefore, the case $\numuav=2$ is of special relevance. Since it is also easier to understand, the rest of this section assumes $\numuav=2$;
        the extension to  $\numuav>2$ is addressed in Sec.~\ref{sec:multi-uavs}.}

    \blt[overview]\arevtwo{The proof of \thref{prop:feasibleexists} is constructive and, therefore, provides a valid path. However, to apply PR it is preferable to adopt the approach described next since it  yields a valid  path that attains a significantly smaller objective; cf. Sec.~\ref{sec:experiments}. This approach  first generates the
        path for UAV-2. Then, a path is found for UAV-1 to serve UAV-2
        all the way. If this is not possible, the path of UAV-2 is
        \emph{lifted} until UAV-2 can be served.}

    \blt[generalities]Before delving into the details of the procedure, some notation and terminology needs to be introduced.
    \begin{bullets}%
        \blt[notation]Let 
        $\rateset(\loc,\rate)\define \{\loc'\in\grid:$ $\capacity(\loc,\loc')\geq \rate\}$
        \blt[candidate grid points UAV 1] and note that, for a path to be valid, it is  necessary (but
        not sufficient) that UAV-1 is in $\rateset(\locbs,2\minuavrate)$ throughout the path and in $\rateset(\locbs,$ $2\minuavrate + \minuerate)$ at the moment of establishing connectivity with the UE.
        \blt[notation]Similarly, let $\rateset(\loc,\rate, \rate')\define
            \{\loc''\in \grid | \exists \loc' \in
            \rateset(\loc,\rate):\capacity(\loc',\loc'')\geq \rate' \}$,
        \blt[candidate grid points UAV 2]and note that, for a path to be valid, it is necessary (not
        sufficient) that UAV-2 is in
        $\nodesetut\define \rateset(\locbs,2\minuavrate,\minuavrate)$ throughout the
        path and in $\destsetut\define \rateset(\locbs,2\minuavrate + \minuerate,
            \minuavrate + \minuerate)\cap \rateset(\locue,\minuerate)$ at
        the moment of establishing connectivity with the UE.

        \begin{changes}%

\blt[candidates]To facilitate the exposition, the set $\nodesetut$ will  be referred to as the set of  \emph{candidate} locations for UAV-2 since the  $\minuavrate$ requirement is satisfied if $\loc_2\in\nodesetut$ and UAV-1 is at a suitable location. 
            \blt[destinations]Similarly, $\destsetut$ will be referred to as the set of \emph{destinations} for UAV-2 since the UE rate will be above $\minuerate$ if  $\loc_2\in\destsetut$ and UAV-1 is at a suitable location.



            %

        \end{changes}%
    \end{bullets}%
\end{bullets}%

\subsubsection{Path for UAV-2}
\label{sec:static:pathuav2}

\begin{bullets}%
    \blt[path UAV-2]
    \begin{bullets}%
        \blt[goal]The idea is to start by first planning the
        path of UAV-2 by finding the shortest path (e.g. via
        Dijkstra's algorithm) from the given $\loc_2[0]=\loc_2(0)$
        \blt[graph]
        \begin{bullets}%
            \blt[destination]to the nearest
            point in $\destsetut\subset\nodesetut$ %
            \blt[nodes]through a graph $\graphuavtwo$ with node set $\nodesetut$.
            \blt[edges]In this graph, two nodes $\loc$ and $\loc'$ are connected
            if and only if $(\loc,\loc')\in\adjset$, where $\adjset$ denotes the set of all pairs of points in $\grid$ that are adjacent. 
            \blt[weights]\arevtwo{Since the objective is to minimize the connection time,} the weight of an edge $(\loc,\loc')$ can be set to
            $\|\loc-\loc'\|$ since this distance is proportional to the time it takes for
            UAV-2 to travel from $\loc$ to $\loc'$ at full speed~$\maxuavspeed$.
        \end{bullets}%

        \blt[waypoints] This procedure produces a path
        $\{\loc_2[0],\loc_2[1],\ldots, \loc_2[\numwp_0-1]\}$, where $\numwp_0$ is the length of the shortest path.
        \blt[summary]\arev{The algorithm is summarized as Algorithm~\ref{algo:trajtwo}.}
    \end{bullets}%
\end{bullets}%

\subsubsection{Path for UAV-1}
\label{sec:static:pathuav1}


\begin{bullets}%
    %
    \blt[get path UAV 1]If there exists a  path  for UAV-1 through $\grid$  that provides a sufficient rate to UAV-2 at all the waypoints $\loc_2[0],\loc_2[1],$ $\ldots, \loc_2[\numwp_0-1]$, the combined path will not only be valid but also potentially optimal. As seen later, this will often be the case, but not always.
    \begin{bullets}%
        \blt[waypoints] Formally, for the combined path to be feasible, the position of UAV-1 must satisfy $\loc_1 \in \nodesetuo[\indwp]\define\rateset(\locbs,2\minuavrate) \cap\rateset(\loc_2[\indwp],\minuavrate)$ when UAV-2 is at $\loc_2[\indwp]$. Besides, for the path to be valid,
        \blt[destination set]it is required that $\loc_1 \in \destsetuo\define\rateset(\locbs,2\minuavrate + \minuerate)\cap \rateset(\loc_2[\numwp_0-1],\minuavrate+ \minuerate)\subset \nodesetuo[\numwp_0-1]$ once UAV-2 reaches $ \loc_2[\numwp_0-1]$.
        \begin{changes}%
            \blt[notion of destinations] Along  the lines of the terminology introduced earlier, $\nodesetuo[\indwp]$ will be referred to as the set of  \emph{candidate} locations for UAV-1  at time step $\indwp$ and $\loc_1\in\destsetuo$ as the set of  \emph{destinations} of UAV-1.
        \end{changes}%

        \blt[extended graph] Since the set of candidate positions depends on $\indwp$, the path must be planned through an \emph{extended graph}.
        \begin{bullets}%
            \blt[enodes]Upon letting the set of extended nodes at time $\indwp$ be $\enodesetuo[\indwp] \define
                \{(\indwp,\loc)~|~\loc \in \nodesetuo[\indwp]\}$, the node set of the extended graph is $\enodesetuo \define \cup_{\indwp}\enodesetuo[\indwp]$.
            \blt[first approach]Initially, one can think of finding a path $(0,\loc_1[0]),(1,\loc_1[1]),\ldots,
                (\numwp_0-1,\loc_1[\numwp_0-1])$ such that
            \begin{bullets}%
                \blt $(\indwp,\loc_1[\indwp])\in
                    \enodesetuo[\indwp]~\forall \indwp$,
                \blt  $\loc_1[\numwp_0-1]\in\destsetuo$,
                \blt  and $(\loc_1[\indwp],\loc_1[\indwp+1])\in\adjset~\forall \indwp$.
            \end{bullets}%
            If this is possible, then the combined path $\{\confpt[\indwp]=[\loc_1[\indwp],\loc_2[\indwp]]$, $\indwp=0,\ldots, \numwp_0-1\}$, is, as indicated earlier, potentially optimal. For the cases where it is not possible, two techniques are presented: \emph{waiting} and \emph{lifting}.

            \blt[waiting approach]\textbf{Waiting.}
            \begin{bullets}%
                \blt[motivation]The aforementioned potentially optimal combined path  can  be found when UAV-1 can maintain the connectivity of UAV-2 just by moving to adjacent locations on $\grid$. However, this may not be the case: sometimes UAV-1 may need to perform multiple steps through adjacent locations on $\grid$ to fly around obstacles in order to guarantee the connectivity of UAV-2; see Fig.~\ref{fig:waitingreq}.
                \blt[waiting]In other words, UAV-2 may need to wait at a certain waypoint until UAV-1 adopts a suitable location.
                \blt[UAV-1 path]To allow for this possibility, the form of the  path of UAV-1 is generalized to be
                $(\indwp_0,\loc_1[0]),(\indwp_1,\loc_1[1]),\ldots, (\indwp_{\numwpwithwaits-1},$ $\loc_1[\numwpwithwaits-1])$ for some $\numwpwithwaits$, where
                \begin{bullets}%
                    \blt $(\indwp_\indaux,\loc_1[\indaux])\in \enodesetuo[{\indwp_\indaux}]$,
                    \blt $\indwp_0=0$,
                    \blt  $\loc_1[\numwpwithwaits-1]\in\destsetuo$,
                    \blt $ \indwp_{\indaux-1} \leq \indwp_{\indaux} \leq  \indwp_{\indaux-1} + 1$,
                    \blt  and $(\loc_1[\indaux],\loc_1[\indaux+1])\in \adjset$ for all $\indaux$.
                \end{bullets}%
                In words, the index $\indwp_\indaux$ need not increase monotonically, it suffices that it  does not decrease.
                \blt[UAV-2 path] The corresponding sequence of
                waypoints for UAV-2 will be $\loc_2[\indwp_0], \loc_2[\indwp_1],\ldots,\loc_2[\indwp_{\numwpwithwaits-1}]$. This means that UAV-2 \emph{waits} at $\loc_2[\indwp_\indaux]$ whenever $\indwp_{\indaux}=\indwp_{\indaux+1}$.%
                \begin{figure}%
                    \centering
                    \includegraphics[width=1\linewidth]{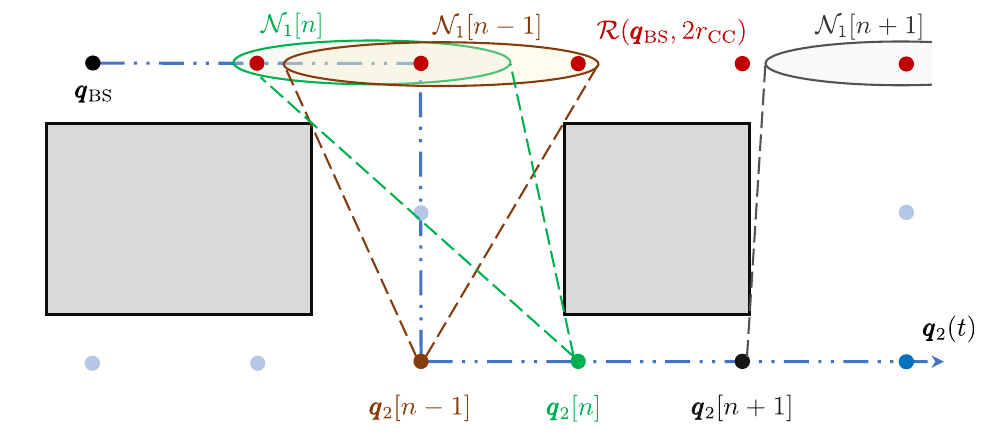}
                    \captionof{figure}{Top view of an example case where no path through adjacent points exists that allows UAV-1 to serve UAV-2 throughout the path of the latter. At some point, UAV-2 may need to wait so that UAV-1 can gain altitude. Grey boxes represent buildings and dots are grid points.   }
                    \label{fig:waitingreq}
                \end{figure}%

                \blt[weights]\arevtwo{
                To enable waiting, the extended graph needs to be modified so that     nodes
                $(\indwp,\loc)$ and $(\indwp',\loc')$ are connected iff $(\loc, \loc')\in \adjset$ and $\indwp\leq \indwp'\leq \indwp+1$. To minimize the time UAV-2 waits, the weight of an edge between $(\indwp,\loc)$ and $(\indwp',\loc')$ is $\|\loc-\loc'\|$.}
            \end{bullets}%
        \end{bullets}%

        \blt[lifting path]\textbf{Lifting.} In certain cases, a  path for UAV-1 may not  be found even with the waiting technique. To remedy this, one can \emph{lift} the  path of UAV-2 to expand the set of candidate locations of UAV-1.
        \begin{bullets}%
            \blt[height of lowest grid higher buildings]To this end, let $\higherbheight$ be the height of the lowest level in $\grid$ that is higher than all obstacles.
            \blt[lift a point]Also, for $\loc\in\grid$, let
            \begin{equation}
                \label{eq:liftfun}
                \liftfun(\loc)\define\begin{cases}
                    \loc+[0,0,\spacingz]\transpose & \text{if } [0,0,1]\loc + \spacingz\leq\higherbheight, \\
                    \loc                           & \text{otherwise},
                \end{cases}
            \end{equation}
            where $\spacingz$ is the spacing along the z-axis between two consecutive levels in $\grid$ and the inner product $[0,0,1]\loc$ returns the z-component (altitude) of $\loc$.
            Additionally, it is convenient to recursively let $\liftfun^{(\indaux)}(\loc)\define\liftfun(\liftfun^{(\indaux-1)}(\loc))$, where $\liftfun^{(1)}(\loc)\define \liftfun(\loc)$.

            \blt[lift a path]
            \arevtwo{
            To  extend  operator $\liftfun$  to paths, 
            let $\pathsingle_2\define\{\loc_2[0],$ $\loc_2[1],\ldots, \loc_2[\numwp_0-1]\}$ be the  path of UAV-2 provided by Algorithm \ref{algo:trajtwo}. Start by defining $\liftfun^{(\indlift)}(\pathsingle_2)=\pathsingle_2$ when $\indlift=0$.
            For the case $\indlift>0$, let $\maxnumliftup\define\min\{\indlift:\liftfun^{(\indlift)}(\loc_2[0])=\liftfun^{(\indlift+1)}(\loc_2[0])\}$ and $\indlift^\downarrow_\text{max}\define\min\{\indlift:\liftfun^{(\indlift)}(\loc_2[\numwp_0-1])=\liftfun^{(\indlift+1)}(\loc_2[\numwp_0-1])\}$ respectively denote the maximum number of times that the initial and final points of $\pathsingle_2$ can be lifted.
            The operator $\liftfun^{(\indlift)}(\pathsingle_2)$, $\indlift>0$, returns the path that results from concatenating the following paths:
            }
            \begin{bullets}%
                \blt[3 subpaths]
                \begin{enumerate}
                    \item an ascent path $\{\loc_2[0],\liftfun^{(1)}(\loc_2[0]),\ldots,\liftfun^{(\indlift^\uparrow-1)}(\loc_2[0])\}$, where $\indlift^{\uparrow}\define\min(\indlift,\maxnumliftup)$,
                    \item the shortest path $\{\loc_2^{(\indlift)}[0],\ldots,\loc_2^{(\indlift)}[\numwp_\indlift-1]\}$ from $\loc_2^{(\indlift)}[0]$ $\define\liftfun^{(\indlift^\uparrow)}$ $(\loc_2[0])$ to $\loc_2^{(\indlift)}[\numwp_\indlift-1]\define\liftfun^{(\indlift^\downarrow)}(\loc_2[\numwp_0-1])$ in the graph of Sec.~\ref{sec:static:pathuav2}, where $\indlift^\downarrow\define\min(\indlift,\indlift^\downarrow_\text{max})$, and
                    \item the descent path $\{\liftfun^{(\indlift^\downarrow-1)}(\loc_2[\numwp_0-1]),\liftfun^{(\indlift^\downarrow-2)}(\loc_2[\numwp_0-1]),$ $\ldots,\loc_2[\numwp_0-1]\}$.
                \end{enumerate}
            \end{bullets}%
        \end{bullets}%

        Lifting the path of UAV-2 expands the set of candidate locations for UAV-1. This motivates iteratively lifting the path of UAV-2 until a suitable path for UAV-1 can be found. Such a procedure is summarized in Algorithm~\ref{algo:trajone}.

        \subsubsection{Theoretical Guarantees}
        \label{sec:static:theoretical}

        The procedure described above and summarized in Algorithm~\ref{algo:trajone} is guaranteed to  eventually succeed under broad conditions:

        \begin{changes}%
            \begin{mytheorem}
                \thlabel{thm:guarantee}
                Let $\grid$ be a sufficiently dense regular grid and let $\confpt[0]=[\locbs,\locbs]$. Suppose that  $\capacity$ is  a TMIA map and 
                \begin{align}
                    \higherbheight\leq&\sqrt{[\capacitydist\inv(2\minuavrate)]^2-[\capacitydist\inv(2\minuavrate+\minuerate)]^2}\\
                    \minuerate\geq&\bandwidth\log_2\left(\frac{1+2^{\pathlossexp}\snrmincc}{1+\snrmincc}\right),
                    \label{lemma1:hypothesis}
                \end{align}
                where $\snrmincc\define2^{\minuavrate/\bandwidth}-1$. If
                a valid path for \eqref{eq:trajproblem}
                exists through waypoints in $\confspacegrid$, then the
                tentative path $\trajectorywpsvalid\define\{[\loc_1[\indaux],\exloc_2[\indwp_\indaux]],~\indaux=0,\ldots,\numwpwithwaits-1\}$ obtained from Algorithm~\ref{algo:trajone} is valid.
            \end{mytheorem}
            \begin{IEEEproof}
                See Appendix~\ref{sec:proof:guarantee}.
            \end{IEEEproof}

        \end{changes}%
        Thus, if $\minuerate$ and  $\minuavrate$ are not too large  relative to the size of the region, the approach in this section results in a valid combined path whenever such a path exists.
    \end{bullets}%
\end{bullets}%

\subsection{Probabilistic Roadmaps with Feasible Initialization}
\label{sec:static-qspace}
This section adapts PR to solve \eqref{eq:trajproblem} by relying on the tentative path produced by Algorithm~\ref{algo:trajone}.

\subsubsection{Construction of the Node Set}
\label{sec:static:nodeset}
\begin{bullets}%
    \blt[overview]As described earlier, the first step in PR is to
    randomly generate a set $\prnodeset$ of CPs. The sampled
    distribution drastically impacts the optimality of the resulting combined path and the computational
    burden of the algorithm. In the work at hand, $\prnodeset$ will
    comprise all the CPs of the path
    $\trajectorywpsvalid=\{\confpt[0],\ldots,\confpt[\numwpwithwaits-1]\}$
    from Sec.~\ref{sec:staticUE-tentative} together with $\numconfpt$ additional
    CPs drawn at random around the CPs in~$\trajectorywpsvalid$.%

    \blt[num samples at each conf pt]Specifically, for each
    $\confpt=[\loc_1,\loc_2] \in \trajectorywpsvalid $, the proposed
    sampling strategy generates $\lfloor
        \numconfpt/\numwpwithwaits\rfloor$ configuration points
    $\confptrnd=[\locrnd_1,\locrnd_2]$ as follows.
    \blt[sample around a location]
    \begin{bullets}%
        \blt First, generate $\locrnd_1$ by drawing a point of
        $\rateset(\locbs,2\minuavrate)-\{\loc_1\}$ with
        probability proportional to $1/\|\locrnd_1-\loc_1\|$.
        \blt Next, independently of $\locrnd_1$, generate
        $\locrnd_2$ by drawing a point of
        $\rateset(\locbs,2\minuavrate,\minuavrate)-\{\loc_2\}$ with
        probability proportional to $1/\|\locrnd_2-\loc_2\|$.
        \blt  If $\locrnd_2 \notin
            \rateset(\locrnd_1,\minuavrate)$, another pair
        $(\locrnd_1,\locrnd_2)$ is generated until $\locrnd_2
            \in \rateset(\locrnd_1,\minuavrate)$, which will
        eventually happen since $\locrnd_2 \in
            \rateset(\locbs,2\minuavrate,\minuavrate)$.
    \end{bullets}%
    %

    \blt[summary] This procedure, which is based on the distance to CPs in
    $ \trajectorywpsvalid $, ensures that many of the sampled CPs lie close to
    the tentative path while others may be farther away, thereby increasing the chances
    for finding nearly optimal paths.
\end{bullets}%

\subsubsection{Construction of the Edge Set}
\label{sec:static:edgeset}

\begin{bullets}%
    \blt[overview] The next step is to construct the edge set of a nearest neighbor graph whose  node set was generated in Sec.~\ref{sec:static:nodeset}.
    \blt[connectivity computation]To obtain trajectories with a lower connection time, it is convenient not to require that each UAV moves through adjacent points in $\grid$: what matters is that the UAVs can move from one CP to another (i) without losing connectivity and (ii) without abandoning $\flyregion$. Thus, an edge $(\confpt,\confpt')$ is added to the edge set if $\confpt$ and $\confpt'$ are nearest neighbors and conditions (i) and (ii) are satisfied. To numerically check (i) and (ii), one can verify that they hold for a finite set of points in the line segment between $\confpt$ and~$\confpt'$.

    \blt[distance]\arevtwo{Since the objective is the connection time}, the weight of an edge between
    $\confpt=[\loc_1,\loc_2]$ and $\confpt'=[\loc'_1,\loc'_2]$ will be given by  the time that the UAVs require to move from $\confpt$ to
    $\confpt'$. Given the speed constraint
    \eqref{eq:maxspeedconstr}, this time is
    determined by the UAV that  traverses the longest distance and, therefore,  equals
    $\max_\induav\|\loc_\induav-\loc'_\induav \|/\maxuavspeed$.
\end{bullets}%

\subsubsection{Path Planning in Q-Space}
\label{sec:static:planning}

\begin{bullets}%
    %
    \blt[shortest path]
    \begin{bullets}%
        \blt[algo]After the nearest-neighbor graph is constructed,
        a shortest path is sought from
        \blt[start]$\confpt_0$
        \blt[destination] to any of
        the CPs $[\loc_1,\loc_2]$ that satisfy $\loc_1 \in
            \rateset(\locbs,2\minuavrate+\minuerate)$ \arev{and $\loc_2 \in
                \rateset(\loc_1,\minuavrate +\minuerate)\cap
                \rateset(\locue,\minuerate)$}.
    \end{bullets}%
    \blt[performance] The resulting CP sequence $\trajectorywpspr$ will never
    have a greater objective than the tentative path $\trajectorywpsvalid$
    provided that all consecutive
    CPs in $\trajectorywpsvalid$ are connected in the PR graph. This
    holds so long as the number of neighbors in Sec.~\ref{sec:static:edgeset} is not too low.%
\end{bullets}%

\subsection{From the Waypoint Sequence to the Trajectory}
\label{sec:fromwptotrajstaticue}
\begin{bullets}
    \blt[given]Given the waypoint sequence
    $\trajectorywpspr=\{\confpt[0],\ldots,\confpt[\prnumwp-1]\}$ \arev{obtained in the previous step},
    it remains only to obtain the trajectory
    $\trajectory$.
    \blt[wp times]To this end, it is necessary  to determine the time
    at which the UAVs arrive at each of the waypoints
    $\confpt[\indwp]$. As indicated in Sec.~\ref{sec:static:planning}, the
    time it takes to arrive at
    $\confpt[\indwp]=[\loc_1[\indwp],\loc_2[\indwp]]$ from
    $\confpt[\indwp-1]=[\loc_1[\indwp-1],\loc_2[\indwp-1]]$ is
    $\max_\induav\|\loc_\induav[\indwp]-\loc_\induav[\indwp-1]
        \|/\maxuavspeed$. Let $t_\indwp$ represent the time at which all the
    UAVs arrive at $\confpt[\indwp]$ and let $t_0=0$. In this case, it clearly holds that
    \begin{align}
        \label{eq:timewp}
        t_\indwp = t_{\indwp-1} + \frac{\max_\induav\|\loc_\induav[\indwp]-\loc_\induav[\indwp-1]
            \|}{\maxuavspeed}.
    \end{align}
    \blt[interpolation] \arev{This provides $\confpt(t_\indwp)=\confpt[\indwp]$} for
    $\indwp=0,\ldots,\prnumwp-1$. For $t\geq
        t_{\prnumwp-1}$, one can simply set
    $\confpt(t)=\confpt(t_{\prnumwp-1})$.  The CPs
    $\confpt(t)$ for other values of $t$ will be determined by the flight controller, which  may be provided  just a sequence of
    waypoints along with their times. For simulation, one can use linear interpolation to resample $\confpt(\cdot)$ at uniform intervals.

    \begin{changesv2}
         Note that, if no lifting steps are used and UAV-2 travels at full speed all the time, the trajectory obtained  by applying \eqref{eq:timewp} directly on the tentative path is optimal  among all trajectories through adjacent gridpoints in $\grid$. This is because no other such a trajectory can attain UE connectivity in a shorter time. The trajectory returned by PRFI will thereby be optimal up to the suboptimality introduced by the spatial discretization.                 
    \end{changesv2}

    \blt[summary] The scheme is summarized as
    Algorithm~\ref{algo:prfi-static} and will be referred to as \emph{PR with
        feasible initialization} (PRFI) for static~UE. \begin{changes}
        Generalizations to $\numuav>2$ and more than one UE are respectively presented in Secs.~\ref{sec:multi-uavs} and \ref{sec:multi-ues}. The complexity analysis can be found in Sec.~\ref{sec:complexity}.
    \end{changes}
\end{bullets}%

\input{algos_static.tex}

\begin{changesv3}    
    \begin{myremark}\label{rem:synchronize}
        To ensure that the UAVs properly follow the planned trajectories, their on-board localization and synchronization systems should be sufficiently accurate. If, e.g. due to wind or unexpected obstacles, a UAV cannot reach a planned waypoint on time, it should inform the rest so they wait and  connectivity is maintained. Replanning may be necessary.
    \end{myremark}
    \begin{myremark}\label{rem:dynamicenv}
        For deployments in dynamic environments, it is convenient to (i) estimate  the capacity map frequently with measurements collected by the UAVs themselves and to (ii) replan the trajectories whenever the capacity map estimate is updated. This may be crucial to guarantee connectivity since obstacles may move or new obstacles appear.
    \end{myremark}
\end{changesv3}

\section{Path Planning for Moving UE}
\label{sec:moving}

\cmt{summary}
\begin{bullets}%
    \blt[overview]
    \begin{bullets}%
        \blt[intro] This section solves \eqref{eq:trajproblem} when the UE moves.
        \blt[cannot min. connect time]\arevtwo{Since minimizing the connection time is not meaningful in this scenario (cf.  Sec.~\ref{sec:problemformulation}), the focus will be on optimizing  the outage time  \eqref{eq:outageobj} and the amount of transferred data \eqref{eq:cumrateobj}.}%
    \end{bullets}%

    \blt[algorithm] Along the lines of Sec.~\ref{sec:static}, the algorithm here is referred to as \emph{PRFI for moving UE} and also   adapts PR to find a path in Q-space given a tentative path.
    \blt[trajectory sampling]
    \begin{bullets}%
        \blt[diff wrt static]However, the algorithm developed in this section significantly differs from the one in Sec.~\ref{sec:static} due to the different temporal dynamics of the problems they address: whereas in Sec.~\ref{sec:static} the UAVs must move at maximum speed to reach the destination as fast as possible, in the case of a moving UE, the time at which the UAVs must arrive at each waypoint is dictated by the trajectory of the UE.
        \blt[regular sampling]For this reason, the paths of the UE and UAVs will be obtained by sampling their trajectories at a regular interval $\sampint$.
        \blt[UE path]Specifically, the path of the UE will be represented as $\trajectoryue:=\{\locue[0],\locue[1],\ldots,\locue[\numwpue-1]\}$, where $\locue[\indwp]=\locue(\indwp\sampint)$, $\indwp=0,\ldots,\numwpue-1$.
        \blt[samp int vs uav speed]Similarly, the paths of the UAVs will be planned in such a way that each one is at a point of $\grid$ at every sampling instant. This requires that $\sampint$ is small enough so that the UAVs can move from one grid point to any adjacent one in this time.
    \end{bullets}%
\end{bullets}%

\subsection{Planning the Tentative Path}
\label{sec:moving-tentative}
\begin{bullets}%
    \blt[generalities]
    \begin{bullets}%
        \blt[$\numuav=2$]
        \arevtwo{
            As in Sec.~\ref{sec:static}, the explanation will assume $\numuav=2$  UAVs; the case $\numuav>2$ is addressed in Sec.~\ref{sec:multi-uavs}.}
        
        \blt[metric] Recall that the closer the tentative path to the optimal combined path, the greater the optimality of the combined path returned by PR. Therefore, it is desirable that the tentative path approximately minimizes the adopted metric.
        \begin{bullets}%
            \blt[min outage] In the case of the  outage time, this can be readily accomplished by planning the path of both UAVs separately along the lines of Sec.~\ref{sec:static}.
            \blt[cum rate] However, when the metric is the one in \eqref{eq:cumrateobj}, such an approach
            is not viable  because the UE rate is a function of the
            positions of both UAVs. As noted in Sec.~\ref{sec:introroadmap},
            planning such a path jointly would be computationally prohibitive. Thus, with this metric, the tentative path will
            still be planned to minimize the outage time. PR will then
            optimize the path in Q-space to maximize the metric in \eqref{eq:cumrateobj}.
        \end{bullets}%
    \end{bullets}%
\end{bullets}%
\subsubsection{Path for UAV-2}
\begin{bullets}%
    \blt[path for UAV2]
    \begin{bullets}%
        %
        \blt[goal]Given that the goal is to minimize the outage time, one would ideally like to \emph{impose} that UAV-2 remains in the
        set of locations where it can provide $\minuerate$ to the UE for a
        suitable location of UAV-1. Since this set generally changes over time and,
        therefore, such an approach need not be feasible, a reasonable
        alternative is to \textit{encourage} UAV-2 to stay in these sets of locations
        \begin{bullets}%
            \blt[the need for extended graph] by planning a path through  an
            extended graph with properly weighted edges.
        \end{bullets}%

        \blt[extended graph] To construct such a graph, note that
        \begin{bullets}%
            \blt[extended nodes]
            \begin{bullets}%
                \blt[candidate] the set of {candidate locations} of UAV-2 is $\nodesetut\define\rateset(\locbs,2\minuavrate,\minuavrate)$ and does not depend on the location of the UE. With $\enodesetut[\indwp]\define\{(\indwp,\loc)|\loc\in\nodesetut\}$ denoting the set of extended nodes at time step $\indwp$, the node set of the extended graph is $\enodesetut\define\cup_\indwp\enodesetut[\indwp]$.
                \blt[destinations]In contrast, the set of destinations  $\destsetut[\indwp]\define\rateset(\locbs,2\minuavrate+\minuerate,\minuavrate+\minuerate)\cap\rateset(\locue[\indwp],\minuerate)$, which comprises those locations where UAV-2 can provide $\minuerate$ to the UE, does generally change over time as it depends on $\locue[\indwp]$. The corresponding set of extended nodes at time step $\indwp$ is given by $\edestsetut[\indwp]\define\{(\indwp,\loc)|\loc\in\destsetut[\indwp]\}\subset\enodesetut[\indwp]\subset\enodesetut$.
            \end{bullets}

            \blt[edges]
            \begin{bullets}%
                \blt[def]In this graph, nodes $(\indwp,\loc)$ and $(\indwp',\loc')$ are connected by an edge iff $\indwp'= \indwp+1$ and $(\loc',\loc)\in \adjset$.
                \blt[paths]In this way, a path for UAV-2 is a sequence of extended nodes $(0,\loc_2[0]),(1,\loc_2[1]),\ldots,$ $(\numwpue-1,\loc_2[\numwpue-1])$ where
                \begin{bullets}%
                    \blt $(\indwp,\loc_2[\indwp])\in\enodesetut[\indwp]$
                    \blt and $(\loc_2[\indwp],\loc_2[\indwp+1]) \in \adjset\forall \indwp$.
                \end{bullets}%
            \end{bullets}%
            \blt[weights - connectivity] The weight  of an edge $((\indwp,\loc), (\indwp^\prime,\loc^\prime))$, which captures  the cost of traveling from $(\indwp,\loc)$ to $(\indwp^\prime,\loc^\prime)$,
            is given by
            \begin{align}
                \label{eq:costmovingueuavt}
                \weightfun((\indwp,\loc), & (\indwp',\loc')
                )=                                                                                                     \begin{cases}
                                                                                                                           0              & \text{ if } \loc = \loc', \loc'\in\destsetut[\indwp']    \\
                                                                                                                           1              & \text{ if } \loc \neq \loc', \loc'\in\destsetut[\indwp'] \\
                                                                                                                           \weightpenalty & \text{ if } \loc^\prime\notin\destsetut[\indwp'],
                                                                                                                       \end{cases}
            \end{align}
            where $\weightpenalty$ is a large positive number that encourages UAV-2 to stay in $\edestsetut[\indwp]$ at time step $\indwp$. Observe that, even when UAV-2 remains in these sets, the cost is greater if it moves.
            %
        \end{bullets}%

        \blt[start and end enodes] A shortest path algorithm is used to find the path of UAV-2 from the extended node corresponding  to the take-off location to any extended node in $\enodesetut[\numwpue-1]$. The procedure to find the path of UAV-2 is summarized as Algorithm~\ref{algo:pathtwo-movingue}.
    \end{bullets}%
\end{bullets}%

\subsubsection{Path for UAV-1}
\label{sec:moving:tentative-uav1}
\begin{bullets}%
    \blt[effect on cost]With the cost in \eqref{eq:costmovingueuavt}, the number of time slots where UAV-2 is in a location that can provide $\minuerate$ to the UE for a suitable location of UAV-1 is maximized.
    By suitable location it is meant that UAV-1 can provide $\minuavrate+\minuerate$ to UAV-2.
    Unfortunately, since the set of suitable locations for UAV-1 changes with $\indwp$, it may not be possible for UAV-1 to be in a suitable location all the time. By finding a  path for UAV-1 so that it stays within these sets as much as possible, the combined path will approximately minimize the outage time.

    \blt[planning path for UAV-1]
    \begin{bullets}%
        %
        \blt[extended graph]To this end, the  path must be  planned through an extended graph.
        \begin{bullets}%
            \blt[extended nodes]%
            \begin{bullets}%
                \blt[given path uav2]Let $\pathsingle_2=\{\loc_2[0],\loc_2[1],\ldots,$ $\loc_2[\numwpue-1]\}$ be the  path of UAV-2 returned by Algorithm~\ref{algo:pathtwo-movingue}.
                \blt[candidate]The set of candidate locations of UAV-1 at time step $\indwp$ is  $\nodesetuo[\indwp]\define\rateset(\locbs,2\minuavrate)\cap\rateset(\loc_2[\indwp],\minuavrate)$ and the associated set of extended nodes is $\enodesetuo[\indwp]\define\{(\indwp,\loc)|\loc\in\nodesetuo[\indwp]\}$. The node  set of the  extended graph  is therefore $\enodesetuo\define\cup_\indwp\enodesetuo[\indwp]$.
                \blt[destinations] On the other hand, the set of destinations of UAV-1 at time step $\indwp$ is given by  $\destsetuo[\indwp]\define\rateset(\locbs,2\minuavrate+\minuerate)\cap\rateset(\loc_2[\indwp],\minuavrate+\minuerate)\subset\nodesetuo[\indwp]$.

            \end{bullets}%
            \blt[edges]
            \begin{bullets}%
                \blt[no waiting]As opposed to  Sec.~\ref{sec:staticUE-tentative}, UAV-2 cannot wait for UAV-1 since that would introduce an offset between a part of $\pathsingle_2$ and $\trajectoryue$.
                \blt[def]Nodes $(\indwp,\loc)$ and $(\indwp',\loc')$ are therefore connected by an edge iff $\indwp'= \indwp+1$ and $(\loc',\loc)\in \adjset$.
                \blt[example path] This means that the path of UAV-1  in the extended graph will have the form $(0,\loc_1[0]),(1,\loc_1[1]),\ldots,(\numwpue-1,\loc_1[\numwpue-1])$, where
                \begin{bullets}%
                    \blt[] $(\indwp,\loc_1[\indwp])\in\enodesetuo[\indwp]$
                    \blt[] and $(\loc_1[\indwp],\loc_1[\indwp+1]) \in \adjset~\forall \indwp$.
                \end{bullets}%
            \end{bullets}%
            \blt[weights - connectivity] Similarly to  UAV-2, \arev{the weight of the edge from $(\indwp,\loc)$ to $(\indwp^\prime,\loc^\prime)$ is} given by \eqref{eq:costmovingueuavt} with $\destsetuo[\indwp']$ in place of $\destsetut[\indwp']$.
        \end{bullets}%

        \blt[end and start] The goal is, therefore, to  find a path  from the extended node corresponding to the take-off location to any node in $\enodesetuo[\numwpue-1]$. Among such feasible paths, a shortest path algorithm picks the one  that results in the lowest accumulated cost, hence the lowest outage time if $\weightpenalty$ is sufficiently large.

        \blt[lifting]\textbf{Lifting.} As in Sec.~\ref{sec:static},   given $\pathsingle_2$ there may be 
        no path for UAV-1 such that the combined path is feasible. Similarly to Sec.~\ref{sec:static:pathuav1}, one can remedy this by lifting $\pathsingle_2$ since this generally expands the sets of candidate locations of UAV-1.
        \begin{changes}%
            \begin{bullets}%
                \blt[diff. in num wps]However, the lifting operator to be used here differs from the one in Sec.~\ref{sec:static:pathuav1} since it must preserve the path length. 

                \blt[lifting operator]
                \begin{bullets}%
                    \blt[tentative path uav2]Consider an arbitrary path $\pathsingle=\{\loc[0],\loc[1],\ldots,\loc[\numwp-1]\}$.
                    \blt[lift and add] Let $\addfun(\pathsingle)\define\{\loc[0],\liftfun(\loc[0]),\liftfun(\loc[1]),\ldots,$ $\liftfun(\loc[\numwp-1])\}$, where $\liftfun$ is  defined in \eqref{eq:liftfun}, be an operator that lifts each point and appends the first one at the beginning. Observe that the length of $\addfun(\pathsingle)$ equals the length of $\pathsingle$ plus 1.  Also, let $\addfun^{(1)}(\pathsingle)\define\addfun(\pathsingle)$ and $\addfun^{(\indlift)}(\pathsingle)\define\addfun(\addfun^{(\indlift-1)}(\pathsingle))$.
                    \blt[remove 1st pt]On the other hand, let $\trimfun(\pathsingle)\define$ $\{\loc[0],\loc[1],\ldots,$ $\loc[\indfirstrepeatedwp-1],\loc[\indfirstrepeatedwp+1],\ldots,$ $\loc[\numwp-1]\}$, where $\indfirstrepeatedwp$ is the smallest $\indwp$ such that $\loc[{\indwp}]=\loc[{\indwp}+1]$ and $\indfirstrepeatedwp=\numwp-1$ if  $\loc[{\indwp}]\neq\loc[{\indwp}+1]$ for all $\indwp$. Observe that the length of $\trimfun(\pathsingle)$ equals the length of $\pathsingle$ minus 1.
                    Also, $\trimfun^{(1)}(\pathsingle)\define\trimfun(\pathsingle)$ and $\trimfun^{(\indaux)}(\pathsingle)\define\trimfun(\trimfun^{(\indaux-1)}(\pathsingle))$.
                    \blt[new lifting operator]Finally,
                    let $\lifttrimfun^{(0)}(\pathsingle)= \pathsingle$ and let
                      $\lifttrimfun^{(\indlift)}(\pathsingle)\define\trimfun^{(\indlift)}(\addfun^{(\indlift)}(\pathsingle))$ for $\indlift>0$. Note that (i) the length of $\lifttrimfun^{(\indlift)}(\pathsingle)$ equals the length of $\pathsingle$, and (ii), if $\pathsingle$ is a path where each pair of consecutive waypoints are adjacent in $\graph$, the same holds for $\lifttrimfun^{(\indlift)}(\pathsingle)$.
                \end{bullets}%

                \blt[iterative lifiting]
                \begin{bullets}%
                    \blt[overall procedure]As in Sec.~\ref{sec:static:pathuav1}, the lifting operator is  iteratively applied to $\pathsingle_2$ until a path $\{\loc_1[0],\loc_1[1],\ldots,\loc_1[\numwpue-1]\}$ for UAV-1 is found.
                    %
                    %
                    %
                    \blt[combined path]
                    With $\indlift$  the number of required lifting steps,
                    the tentative path is then $\trajectorywpsfeas=\{\confpt[0],\ldots,\confpt[\numwpue-1]\}$, where $\confpt[\indwp]=[\loc_1[\indwp], \exloc_2[\indwp]]$
                    for $\{\exloc_2[0],\ldots,\exloc_2[\numwpue-1]\}\define\lifttrimfun^{(\indlift)}(\pathsingle_2)$.
                    %
                \end{bullets}%
                \blt[summary]\arev{This procedure  is summarized as Algorithm~\ref{algo:pathone-movingue}.}
            \end{bullets}%
        \end{changes}%
    \end{bullets}%

\end{bullets}%
\subsubsection{Theoretical Guarantees}
\begin{bullets}%
    \blt[theoretical guarantees]As in Sec.~\ref{sec:static:theoretical}, it is possible to guarantee that the above iterative lifting procedure will eventually produce a feasible path.

    \begin{bullets}%
        \blt[flight level at h max] Let $\higherbheight$ be the height of the lowest grid level that is higher than all  obstacles and let $\gridmaxliftlevel\define\{\loc\in\grid:[0,0,1]\loc=\higherbheight\}$ be the grid level of height $\higherbheight$. Let $\cylinder(\cylinderradius) \define \{[\locx,\locy,\locz]\transpose\in\rfield^3: (\locx -\locxbs)^2 + (\locy-\locybs)^2\leq\cylinderradius^2\}$  be a cylinder of radius $\cylinderradius$ centered at $\locbs$ and let $\cylinderradiusmin$ be the smallest $\cylinderradius$ such that $\grid \subset \cylinder(\cylinderradius)$ or, equivalently, the maximum horizontal distance from the BS to any point in $\grid$.
        \blt[sufficient cond]
        \begin{mytheorem}
            \thlabel{prop:feasible-movingue}
            Suppose that $\confpt_0=[\locbs,\locbs]$ and
            let $\pathsingle_2$ be the  path of UAV-2 returned by Algorithm~\ref{algo:pathtwo-movingue}. If $\higherbheight\leq\capacitydist\inv(2\minuavrate)$ and
            $\cylinderradiusmin\leq\capacitydist\inv(\minuavrate)$,
            then Algorithm~\ref{algo:pathone-movingue} will provide a feasible combined path.
        \end{mytheorem}
        \blt[proof]
        \begin{proof}
            See Appendix~\ref{sec:proof:feasible-movingue}.
        \end{proof}

        \blt[outage time]It is also easy to see that,
        for a  sufficiently large $\weightpenalty$,
        the tentative path is not only feasible but also optimal in terms of outage time if no lifting steps are required and UAV-1 remains at destination points throughout the entire path.
        \blt[data transfer]When it comes to total transferred data, the tentative path will not generally be optimal, but can reasonably be expected to be similar to  the optimal path in many cases.
    \end{bullets}%
\end{bullets}%

\subsection{Probabilistic Roadmaps with Feasible Initialization}
\label{sec:moving-qspace}
The next step is to find a combined path  around the tentative path that approximately  optimizes the considered metric. Since the set of candidate CPs changes over time, an extended graph needs to be adopted. To operate on this graph, the PR algorithm in Sec.~\arev{\ref{sec:introroadmap}} will be generalized.

\subsubsection{Construction of the Node Set}
\begin{bullets}%
    \blt[sampling conf. pts] In addition to
    \begin{bullets}%
        \blt[on tentative path]the CPs in the tentative path $\trajectorywpsfeas=\{\confpt[0],\ldots,\confpt[\numwpue-1]\}$, the algorithm  draws $\numconfpt\geq\numwpue$ additional CPs. In particular,
        %
        \blt[num samples at each cp]   $\lfloor \numconfpt /\numwpue \rfloor$ CPs are drawn
        around each $\confpt[\indwp]$
        as in Sec.~\arev{\ref{sec:static:nodeset}}.
    \end{bullets}%
    \blt[extended nodes]With $\prnodeset[\indwp]$ representing the set that contains $\confpt[\indwp]$ and the CPs drawn around $\confpt[\indwp]$, the set of extended nodes at time step $\indwp$ is given by  $\enodeset[\indwp]\define\{(\indwp,\confpt)|\confpt\in\prnodeset[\indwp]\}$.
\end{bullets}%

\subsubsection{Construction of the Edge Set}
\label{sec:movingue:predgeset}
\begin{bullets}%
    \blt[edges] Two extended nodes $(\indwp,\confpt)$ and $(\indwp',\confpt')$ are connected by an edge iff $\indwp'= \indwp+1$, $(\loc_1, \loc_1')\in\adjset$, and $(\loc_2, \loc_2')\in\adjset$, where $\confpt=[\loc_1,\loc_2]$ and $\confpt'=[\loc_1',\loc_2']$.
    \blt[weights] The edge weights  depend on the metric to be optimized.
    \begin{bullets}%
        \blt[min outage] To minimize the outage time, \arevtwo{one needs to minimize the number of time slots in which $\uerate(\confpt[\indwp],\locue[\indwp])$ is below $\minuerate$ (this follows by discretizing \eqref{eq:outageobj}  as $\funobj(\trajectory) \approx\sampint\sum_0^{\numwp-1} \funindicator[\uerate(\confpt[\indwp],\locue[\indwp])<\minuerate]$). As a result, one can set
        }
        \begin{align}
            \label{eq:weightsoutage}
            \weightfun((\indwp, \confpt),(\indwp', \confpt'))=                                     \\
            \nonumber
             & \begin{cases}
                   0              & \text{ if } \confpt = \confpt', \uerate(\confpt')\geq\minuerate    \\
                   1              & \text{ if } \confpt \neq \confpt', \uerate(\confpt')\geq\minuerate \\
                   \weightpenalty & \text{ if } \uerate(\confpt')<\minuerate,
               \end{cases}
        \end{align}
        where $\weightpenalty$ is again a large positive number.
        \blt[transferred data] To maximize the transferred data,
        observe that the integral in \eqref{eq:cumrateobj} can be discretized as
        \begin{align}
            \funobj(\trajectory) \approx -\sampint\sum_{\indwp=0}^{\numwpue-1}
            \uerate(\confpt[\indwp],\locue[\indwp])
            \label{eq:cumrateobjdisc}
            .
        \end{align}
        Since a shortest path algorithm  minimizes the sum of the weigths of the edges in a path, one can therefore set $\weightfun((\indwp, \confpt),(\indwp', \confpt'))= -\uerate(\confpt', \locue[\indwp'])$
        \footnote{\begin{changesv2}Indeed, other additive (cummulative) metrics (objectives) can also be optimized by the same approach.\end{changesv2}}.
    \end{bullets}%
\end{bullets}%
\subsubsection{Path Planning in Q-Space}
\begin{bullets}%
    \blt[shortest path] A shortest-path algorithm is then used to find the path  $\trajectorywpspr=\{\confpt[0],\confpt[1],\ldots,\confpt[\numwpue-1]\}$ of the UAVs in the extended graph from \arev{$\confpt_0$} to an extended node in $\enodeset[\numwpue-1]$ that results in the lowest accumulated cost.
    \blt[diff PR]This differs from standard PR, which finds a shortest path through a nearest-neighbor graph.%
\end{bullets}%

\subsection{From the Waypoint Sequence to the Trajectory}
\label{sec:movingwptotraj}

\begin{bullets}%
    \blt[sampling] As described at the beginning of Sec.~\ref{sec:moving},
    the produced path for the UAVs is sampled at regular intervals $\sampint$.
    \blt[given] Thus, given  $\trajectorywpspr$,
    \blt[trajectory] the final trajectory  $\trajectory$  satisfies $\confpt(\indwp\sampint)=\confpt[\indwp],\indwp=0,\ldots,\numwpue-1$. As indicated earlier, the position of the UAVs at  intermediate time instants is determined by the flight controller.
    \begin{changesv2}
        \begin{bullets}%
            \blt[comparison] Observe that
            \begin{bullets}%
                \blt[moving] PRFI for moving UEs returns a path in which the UAVs fly to  adjacent grid points at each  time step, albeit not necessarily at maximum speed.
                \blt[static] Meanwhile, PRFI for static UEs returns a path in which the UAVs fly at their maximum speed from $\confpt[\indwp]$ to $\confpt[\indwp+1]$, but, $\loc_{\induav}[\indwp]$ is not necessarily adjacent to $\loc_{\induav}[\indwp+1]$ on the grid. Therefore, although PRFI for moving UEs can be used to serve static UEs, PRFI for static UEs generally has greater~optimality.
            \end{bullets}%
        \end{bullets}%
    \end{changesv2}

    \blt[summary] The complete procedure is summarized as Algorithm~\ref{algo:prfi-moving};
    \begin{changes}
        see  Secs.~\ref{sec:multi-uavs}, \ref{sec:multi-ues} and~\ref{sec:complexity} for generalizations and complexity analysis.
    \end{changes}
    \input{algos_moving.tex}

    \begin{changesv2}
        \blt[setting $\sampint$] Since $\sampint$ affects the optimality of the obtained trajectory, it is important to set it as low as allowed by the available computational resources.
        \begin{bullets}%
            \blt[complexity]To this end, note that the complexity of the algorithm can be expressed as a function $\complexityfun(\numwpue,\numgridpt)$. It can be shown that $\complexityfun$ is
            $\complexity(\numwpue\numgridpt\log(\numwpue\numgridpt))$ (cf. Sec.~\ref{sec:complexity}) in order notation, but the exact expression,
            which depends on the implementation, is required here.
            \blt[in terms of $\sampint$]The idea is to express $\numwpue$ and $\numgridpt$ in terms  of $\sampint$ and then solve for $\sampint$.
            \begin{bullets}%
                \blt Specifically, start by noting that $\numwpue=\horizon/\sampint$.
                \blt On the other hand,  note  that  $\sampint$ should equal  the time a UAV needs to fly to the farthest adjacent grid point at full speed. If, for simplicity, 
                $\oobregion$ is a cube of side length $\arealength$  and $\spacingscalar\define\spacingx=\spacingy=\spacingz$, then  $\sampint=\sqrt{3}\spacingscalar/\maxuavspeed =\sqrt{3}\arealength/(\maxuavspeed\sqrt[3]{\numgridpt})$. This implies that $\numgridpt=[\sqrt{3}\arealength/(\sampint\maxuavspeed)]^3$.
            \end{bullets}%
            \blt As a result, the complexity can be written as
            $\complexityfun(\horizon/\sampint,[\sqrt{3}\arealength/(\sampint\maxuavspeed)]^3)$. Equating this expression  to the available resources and solving for $\sampint$ yields  the desired $\sampint$  (and also $\spacingscalar$ as $\spacingscalar=\maxuavspeed\sampint/\sqrt{3}$).
        \end{bullets}%
    \end{changesv2}
\end{bullets}%

\section{Extensions}
\label{sec:extensions}

\subsection{Extension to more than two UAVs}
\label{sec:multi-uavs}

\begin{changesv2}%
    This section extends Algorithms~\ref{algo:prfi-static} and \ref{algo:prfi-moving} to the case of more than two UAVs.
    \begin{bullets}%
        \blt[notation] To this end, let $\rateset(\loc,\rate_1,\rate_2,\ldots,\rate_\induav)\define \{\loc'\in\grid|\exists\bar{\loc}\in\rateset(\loc,\rate_1,\rate_2,\ldots,\rate_{\induav-1}):\capacity(\bar{\loc},\loc')\geq \rate_\induav\}$, $\induav>2$.
        \blt[prfi for static ues]
        \begin{bullets}%
            \blt[algorithm] PRFI for static UEs is extended to $\numuav$ UAVs, $\numuav>2$, as follows.
            \begin{bullets}%
                \blt[K]The tentative path of UAV-$\numuav$ is planned by Algorithm~\ref{algo:trajtwo} after replacing $\nodeset_{2}$ and $\destsetut$ in Steps \ref{step:static-ue:uav2-candidates} and \ref{step:static-ue:uav2-destinations} with $\nodeset_{\numuav}\define\rateset(\locbs,\numuav\minuavrate,(\numuav-1)\minuavrate,\ldots,2\minuavrate,\minuavrate)$ and $\destset_{\numuav}\define\rateset(\locbs,\numuav\minuavrate+\minuerate,(\numuav-1)\minuavrate+\minuerate,\ldots,2\minuavrate+\minuerate,\minuavrate+\minuerate)\cap\rateset(\locue,\minuerate)$.
                \blt[k]The tentative path of UAV-$\induav$, $\induav=\numuav-1, K-2,\ldots,1$, is planned by Algorithm~\ref{algo:trajone} after replacing $\nodesetuo[\indwp]$ and $\destsetuo[\indwp]$ in Steps~\ref{step:static-ue:uav1-candidates} and \ref{step:static-ue:uav1-destinations} with $\nodeset_{k}[\indwp]=\rateset(\locbs,\numuav\minuavrate,(\numuav-1)\minuavrate,\ldots,(\numuav-\induav+1)\minuavrate)\cap\rateset(\exloc_{k+1}[\indwp],(\numuav-\induav)\minuavrate)$ and $\destset_{k}[\indwp]=\rateset(\locbs,\numuav\minuavrate+\minuerate,(\numuav-1)\minuavrate+\minuerate,\ldots,(\numuav-\induav+1)\minuavrate+\minuerate)\cap\rateset(\exloc_{k+1}[\indwp],(\numuav-\induav)\minuavrate+\minuerate)$. In Step \ref{step:static-ue:lift}, the path of UAV-$\induav'$ needs to be lifted for all $\induav'>\induav$. The remaining steps can be readily extended.
            \end{bullets}%
            By planning the tentative path for the $\numuav$ UAVs in this way, Algorithm~\ref{algo:prfi-static} is  extended to find the approximately optimal path.
            %
            %
            \begin{extendedonly}
                \blt[observe] In the extension above, Algorithm~\ref{algo:trajone} adopts the waiting approach (described in Sec.~\ref{sec:static:pathuav1}) to plan the tentative path for UAV-$\induav$. This approach may require UAV-$(\induav+1)$ to wait at some time steps. This then requires UAV-$(\induav+2)$ to wait for UAV-$(\induav+1)$, UAV-$(\induav+3)$ to wait for UAV-$(\induav+2)$, and so on.
            \end{extendedonly}
        \end{bullets}

        \blt[prfi for moving ues] Similarly, PRFI for moving UEs can be extended to $\numuav>2$ by successively planning the tentative path for UAV-$\numuav$, then UAV-$(\numuav-1)$ and so on.
        \begin{bullets}%
            \blt[algorithm] Particularly,
            \begin{bullets}
                \blt[K]the tentative path of UAV-$\numuav$ is planned by Algorithm~\ref{algo:pathtwo-movingue} after replacing $\nodeset_{2}[\indwp]$ and $\destsetut[\indwp]$ in Steps \ref{step:moving-ue:uav2-candidates} and \ref{step:moving-ue:uav2-destinations} with $\nodeset_{\numuav}[\indwp]\define\rateset(\locbs,\numuav\minuavrate,(\numuav-1)\minuavrate,\ldots,2\minuavrate,\minuavrate)$ and $\destset_{\numuav}[\indwp]\define\rateset(\locbs,\numuav\minuavrate+\minuerate,(\numuav-1)\minuavrate+\minuerate,\ldots,2\minuavrate+\minuerate,\minuavrate+\minuerate)\cap\rateset(\locue[\indwp],\minuerate)$.
                \blt[k]The tentative path of UAV-$k$, $k=K-1, K-2,\ldots,1$, is planned by Algorithm~\ref{algo:pathone-movingue} after replacing $\nodesetuo[\indwp]$ and $\destsetuo[\indwp]$ in Steps~\ref{step:moving-ue:uav1-candidates} and \ref{step:moving-ue:uav1-destinations} with $\nodeset_{k}[\indwp]=\rateset(\locbs,\numuav\minuavrate,(\numuav-1)\minuavrate,\ldots,(\numuav-\induav+1)\minuavrate)\cap\rateset(\exloc_{k+1}[\indwp],(\numuav-\induav)\minuavrate)$ and $\destset_{k}[\indwp]=\rateset(\locbs,\numuav\minuavrate+\minuerate,(\numuav-1)\minuavrate+\minuerate,\ldots,(\numuav-\induav+1)\minuavrate+\minuerate)\cap\rateset(\exloc_{k+1}[\indwp],(\numuav-\induav)\minuavrate+\minuerate)$. In Step \ref{step:moving-ue:lift}, the path of UAV-$\induav'$ needs to be lifted for $\induav'>\induav$. The remaining steps can be readily extended.
            \end{bullets}%
            By planning the tentative paths of the $\numuav$ UAVs, Algorithm~\ref{algo:prfi-moving} is extended  to find the approximately optimal path.
            %
        \end{bullets}%
        %
    \end{bullets}%
\end{changesv2}

\subsection{Extension to multiple UEs}
\label{sec:multi-ues}
\begin{changesv2}
    \begin{bullets}%
        \blt[intro] This section extends Algorithms~\ref{algo:prfi-static} and \ref{algo:prfi-moving} to the case of multiple UEs. Extending the algorithms in Sec.~\ref{sec:multi-uavs} follows along the same lines.
        \blt[notation]
        \begin{bullets}
            \blt[num ues] Let $\numues$ and $\{1,2,\ldots,\numues\}$ respectively denote the number and the set of indices of the UEs.
        \end{bullets}%
        
        \blt[prfi for static ues] In the case of static UEs,
        \begin{bullets}%
            \blt[loc ues] the location of the $\indue$-th UE is denoted by $\locue^{(\indue)}$.
            \blt[rate] Given a configuration point $\confpt$, the achievable rate of the $\indue$-th UE is $\uerate^{(\indue)}(\confpt)\define\uerate(\confpt,\locue^{(\indue)})$.
            \blt[rate] For simplicity, assume that each UE requires the same rate $\minuerate$.
            \blt[objective] Since it may not be possible to serve all UEs at the same time, a reasonable objective would be to serve as many UEs as possible.
            \begin{bullets}%
                \blt[grid pts to serve] Formally, the index set $\setueinds$ of served UEs can be taken to be the largest subset of $\{1,2,\ldots,\numues\}$ such that the destination set $\setgridptsmultiues\define\cap_{\indue\in\setueinds}\rateset(\locue^{(\indue)},\minuerate)$ is not empty.
            \end{bullets}%
            \blt[algorithm] Then, Algorithms~\ref{algo:trajtwo} and \ref{algo:trajone} are modified as follows:
            \begin{bullets}
                \blt[input] $\{\locue^{(\indue)}\}_{\indue\in\setgridptsmultiues}$ will be one of the inputs.
                \blt[uav 2] $\destsetut$ in Step~\ref{step:static-ue:uav2-destinations} in Algorithm~\ref{algo:trajtwo} is replaced with $\destsetut=\rateset(\locbs,2\minuavrate+|\setueinds|\minuerate,\minuavrate+|\setueinds|\minuerate)\cap\setgridptsmultiues$.
                \blt[uav 1] $\destsetuo$ in Step~\ref{step:static-ue:uav1-destinations} of Algorithm~\ref{algo:trajone} is then replaced with $\destset_1=\rateset(\locbs, 2\minuavrate + |\setueinds|\minuerate)\cap \rateset(\exloc_2[\numwp_\indlift-1],\minuavrate + |\setueinds|\minuerate)$.
            \end{bullets}
            Finally, Algorithm~\ref{algo:prfi-static} is used to find the approximately optimal path by replacing Step~\ref{step:static-shortest-path} with
            $\trajectorywpspr:=$ shortest\_path($\confpt_0$, $\{\confpt \in \prnodeset:\uerate^{(\indue)}(\confpt)\geq\minuerate,\forall\indue\in\setueinds\}$).
        \end{bullets}%

        \blt[prfi for moving ues] In the case of moving UEs,
        \begin{bullets}%
            \blt[trajectory]  assume without loss of generality that all UEs follow paths of length  $\numwpue$. The trajectory of the $\indue$-th UE can then be denoted by $\trajectoryue^{(\indue)}\define\{\locue^{(\indue)}[0],\locue^{(\indue)}[1],\ldots,\locue^{(\indue)}[\numwpue-1]\}$, where $\locue^{(\indue)}[\indwp]$ is the location of the $\indue$-th UE at time step $\indwp$.
            \blt[objective] Since it may not be possible to provide connectivity to all UEs at every time step, a reasonable objective would be to serve as many UEs as possible.
            \begin{bullets}%
                \blt[grid pts to serve]Formally, the index set $\setueinds[\indwp]$ of served UEs at time step $\indwp$ can be taken to be the largest subset of $\{1,2,\ldots,\numues\}$ such that $\setgridptsmultiues[\indwp]\define\cap_{\indue\in\setueinds[\indwp]}\rateset(\locue^{(\indue)}[\indwp],\minuerate)\neq\emptyset$.
            \end{bullets}%
            \blt[algorithm] Then, Algorithms~\ref{algo:pathtwo-movingue} and \ref{algo:pathone-movingue} are modified as follows.
            \begin{bullets}%
                \blt[input] $\{\trajectoryue^{(\indue)}\}_{\indue=1}^{\numues}$ will be one of the inputs.
                \blt[uav 2] $\destsetut[\indwp]$ in Step~\ref{step:moving-ue:uav2-destinations} of Algorithm~\ref{algo:pathtwo-movingue} is replaced with $\destsetut[\indwp]=\rateset(\locbs,2\minuavrate+|\setueinds[\indwp]|\minuerate,\minuavrate+|\setueinds[\indwp]|\minuerate)\cap\setgridptsmultiues[\indwp]$.
                \blt[uav 1] $\destsetuo[\indwp]$ in Step~\ref{step:moving-ue:uav1-destinations} of Algorithm~\ref{algo:pathone-movingue} is  replaced with $\destset_1[\indwp]=\rateset(\locbs, 2\minuavrate + |\setueinds[\indwp]|\minuerate)\cap \rateset(\exloc_2[\indwp],\minuavrate + |\setueinds[\indwp]|\minuerate)$.
            \end{bullets}%
            These modifications extend Algorithm~\ref{algo:prfi-moving} to the case $\numues>1$.
        \end{bullets}%
    \end{bullets}%
\end{changesv2}%

\subsection{Complexity analysis}
\label{sec:complexity}
\begin{changesv2}
    \input{complexity.tex}
\end{changesv2}

\section{Numerical Experiments}
\label{sec:experiments}
\cmt{overview} This section presents numerical results that validate and assess the performance of the proposed 
algorithms. The developed simulator, the simulation code, and some videos  can be found
at \url{https://github.com/uiano/pr_for_relay_path_planning}.

\subsection{Simulation Setup}
Unless stated otherwise, the experiments adopt the parameters in Table \ref{tab:simparams}.
\begin{bullets}%
    \blt[simulation setup]
    \begin{bullets}%
        \blt[channel]
        \arevtwo{
            Also,  $\capacity$ is obtained with the tomographic
            channel model (cf. Sec.~\ref{sec:capmap}) with an
            absorption of 1 dB/m inside the buildings and 0 dB/m outside.
        }

        \begin{table}[!t]
            \begin{changes}
                \begin{center}
                    \caption{Simulation parameters unless otherwise stated.}
                    \label{tab:simparams}
                    \begin{tabular}{ |>{\centering}p{1cm}|p{4.4cm}|>{\centering\arraybackslash}p{2cm}| }
                        \hline
                        \textbf{Notation}     & \textbf{Physical meaning}                                     & \multirow{1}{*}{\textbf{Simulation value}} \\
                        \hline
                        $\region\in\rfield^3$ & Considered region [m$\times$m$\times$m]                   & $500\times500\times100$   \\
                        \hline
                        $\grid$               & Flight grid $[\numgridptx\times\numgridpty\times\numgridptz]$ & $12\times12\times 8$      \\
                        \hline
                                              & Minimum flight height                                         & 12.5 m                    \\
                        \hline
                        $\heighttop$          & Maximum flight height                                         & 87.5 m                    \\
                        \hline
                        $\locbs$              & Location of the BS                                            & $[20,470,0]\transpose$    \\
                        \hline
                                              & Number of buildings                                           & 25                        \\
                        \hline
                        $\numuav$             & Number of UAVs                                                & 2                         \\
                        \hline
                        $\maxuavspeed$        & Maximum UAV speed                                     & 7 m/s                     \\
                        \hline
                        $\minuavrate$         & Minimum UAV rate                                      & 200 kbps                  \\
                        \hline
                                              & Carrier frequency                                             & 6 GHz                     \\
                        \hline
                        $\bandwidth$          & Bandwidth                                                     & 20 MHz                    \\
                        \hline
                        $\txpower$            & Transmit power                                                & 17 dBm                    \\
                        \hline
                        $\txgain$, $\rxgain$  & Tx./Rx. Antenna gain                                          & 12 dBi                    \\
                        \hline
                        $\noisepower$         & Noise power                                                   & -97 dBm                   \\
                        \hline
                        $\numconfpt$          & No. of configuration points in PRFI                        & 2000                      \\
                        \hline
                                              & No. of neighbors in PRFI                                   & 100                       \\
                        \hline
                                              & No. of Monte Carlo (MC) realizations                       & 400                       \\
                        \hline
                    \end{tabular}
                \end{center}
            \end{changes}
        \end{table}
    \end{bullets}%
\end{bullets}%

\subsection{Static UE}
\label{sec:sim:static}
\begin{bullets}%
    \blt[Overview]This section studies the performance of the proposed PRFI algorithm for static UEs.
    \blt[simulation setup]
    \begin{bullets}%
        \blt[building height] Throughout this section, all buildings have a height of 40 m.
        \blt[ue loc]
        \begin{bullets}%
            \blt[dist bs ue] For generating $\locue$ at each Monte Carlo (MC) realization, the distance $\distbsue =\|\locue-\locbs\|$ is first generated uniformly at random in the interval $[\distbsuemin,\distbsuemax]$, where, unless otherwise stated, $\distbsuemin=50$ m and $\distbsuemax=650$ m.
            \blt[position]Subsequently, $\locue$ is drawn uniformly at random among the points that (i) are outside the buildings, (ii) are at a distance $\distbsue$ from $\locbs$, and (iii) satisfy $\capacity(\locbs,\locue)\leq\minuerate$.
        \end{bullets}%
    \end{bullets}%

    \begin{changesv2}
        \blt[benchmarks] As benchmarks, this section adapts five state-of-the-art algorithms. The optimization problems that they rely on can be solved for the setup at hand, resulting in the following trajectories:
        \begin{itemize}%
            \item Zeng et al. \cite{zeng2016relaying}:
                  a UAV takes off vertically at the BS until reaching
                  $\ptbsabove\define[\locxbs,\locybs,
                          \heighttop]\transpose$. It then flies horizontally to
                  $\ptmid\define(\ptbsabove+\ptueabove)/2$, where $\ptueabove\define[\locxue,\locyue,\heighttop]\transpose$. This would coincide with the trajectory obtained via \cite{zhang2022cooperative} and \cite{liu2021relaying}.
            \item Ghazzai et al. \cite{ghazzai2018dual}: a UAV takes off at the BS to
                  $\ptbsabove$ and flies horizontally to the grid point
                  that maximizes the UE rate predicted by the channel model proposed in
                  \cite{alhourani2014lap}.
            \item Lee et al. \cite{lee2022trajectory}: two UAVs lift off at the BS
                  to $\ptbsabove$. Then, UAV-1 flies horizontally to $\ptmid$. Meanwhile, UAV-2 flies horizontally to $\ptueabove$.
            \item Zhang et al. \cite{zhang2018multi}:  two UAVs lift off at the BS
                  to $\ptbsabove$. Then, UAV-1 stays at $\ptbsabove$. Meanwhile,  UAV-2 flies horizontally to the grid point that maximizes the UE rate.
            \item Yanmaz et al. \cite{yanmaz2022positioning}: two UAVs lift off at the BS to $\ptbsabove$. Then, UAV-1 stays at $\ptbsabove$. Meanwhile,  UAV-2 flies horizontally to $\ptueabove$.
                  %
        \end{itemize}%
    \end{changesv2}
    \arevtwo{These algorithms are compared with PRFI (Tentative) and PRFI, which respectively correspond to Algorithms~\ref{algo:trajone} and~\ref{algo:prfi-static}}.

    \blt[experiments]
    \begin{bullets}%
        \blt[rate vs time]
        \begin{bullets}%
            \blt[figure] Fig.~\ref{fig:static-meanuerate}  compares the mean instantaneous rate $\expected[\uerate(\confpt(t))]$  of the considered algorithms.
            %
            \begin{figure}%
                \centering
                \includegraphics[width=\linewidth]{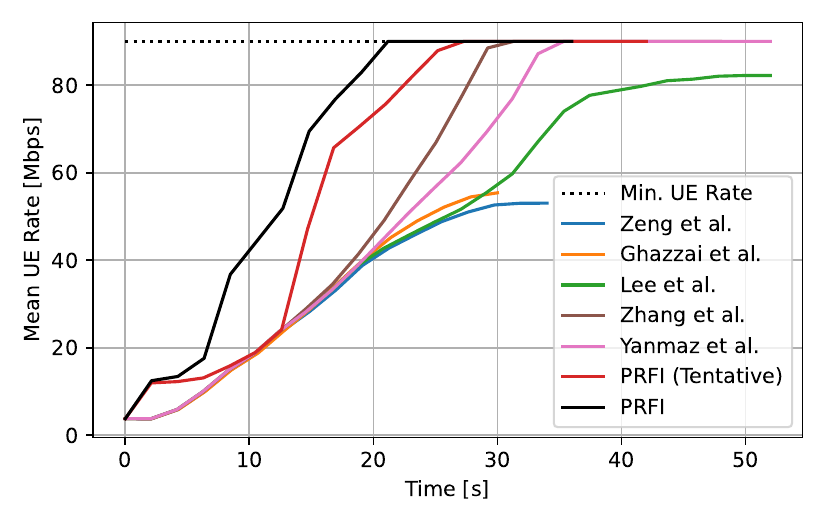}
                \caption{Expected UE rate $\expected[\uerate(\confpt(\timeInstant))]$ vs. $\timeInstant$. The proposed algorithm is the first to attain the target rate $\minuerate$ ($\minuerate = 90$ Mbps, $[\distbsuemin,\distbsuemax]=[150,250]$ m).}
                \label{fig:static-meanuerate}
            \end{figure}%
            \blt[observations]%
            \begin{bullets}%
                \blt[init rate]Since all UAVs start from the BS, the initial rate is the same for all algorithms.
                %
                \blt[obtain min rate]%
                \begin{bullets}%
                    \blt[benchmarks] Although \arevtwo{Zeng et al., Ghazzai et al., and Lee et al.} do not reach the target rate, \arevtwo{Zhang et al. and Yanmaz et al.} do succeed. That \arevtwo{Yanmaz et al.} meets the rate is guaranteed by  \thref{prop:feasibleexists}.
                    %
                    \blt[tentative faster than benchmark 3] PRFI (Tentative), which corresponds to the trajectory produced by Algorithm~\ref{algo:trajone},
                    is already faster than \arevtwo{Zhang et al. and Yanmaz et al.}, which  corroborates the efficacy of the proposed initialization.
                    \blt[faster than tentative] PRFI, which returns the result of applying PR to the tentative path, is \arevtwo{significantly} faster than PRFI (Tentative) \arevtwo{and all benchmarks}. This validates the adoption  of  PR.
                \end{bullets}%
                %
                %
                %
            \end{bullets}%
        \end{bullets}%

        \blt[connection time vs]
        \begin{bullets}
            \blt[min ue rate]
            \begin{bullets}%
                \blt[description] The second experiment studies the influence of $\minuerate$ on the expectation of the connection
                time, which is the cost in~\eqref{eq:timeconnectobj}. To
                this end, Fig.~\ref{fig:static-vs-minrate} plots
                the mean connection time
                $\meanconn\define\expected[\ttc(\trajectory)~\big|~\ttc(\trajectory)<\infty]$ and
                the probability of failure $\probfail\define\prob[\ttc(\trajectory)=\infty]$ vs. $\minuerate$. \arevtwo{In other words, $\meanconn$ is the average of the connection time in the successful MC realizations ($\exists t: \uerate(\confpt(t))\geq \minuerate$) whereas $\probfail$ quantifies the fraction of unsuccessful MC realizations.}
                %
                \begin{figure}
                    \centering
                    \includegraphics[width=\linewidth]{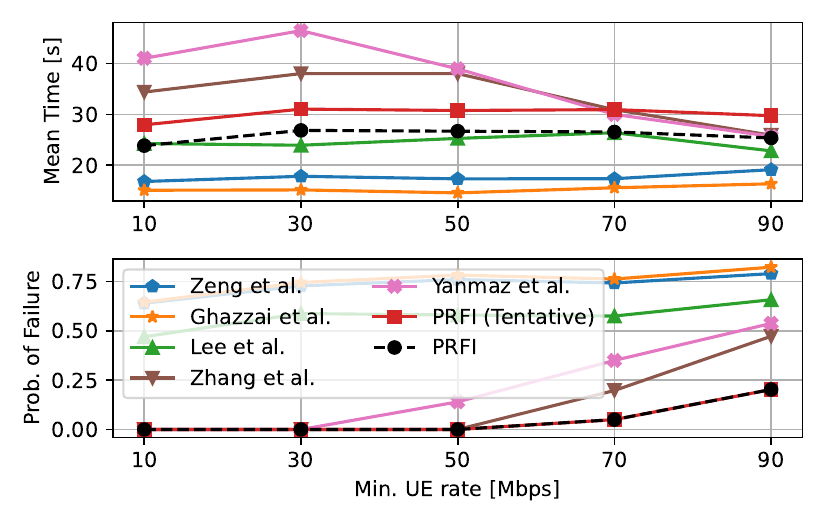}
                    \caption{$\meanconn$ and $\probfail$ vs. $\minuerate$. \arevtwo{Some benchmarks achieve a smaller mean connection time because they only succeed in the easiest MC realizations. }
                    }
                    \label{fig:static-vs-minrate}
                \end{figure}
                \blt[metrics]

                \blt[observation]
                \begin{bullets}%
                    \blt[zeng, ghazzai, lee] \arevtwo{Observe that Zeng et al., Ghazzai et al., and Lee et al.} have a lower $\meanconn$ than PRFI. However, looking at their $\probfail$ reveals that this is because they only succeed in a small fraction of the MC realizations.
                    %
                    %
                    \blt[zhang and yanmaz]\arevtwo{Zhang et al. and Yanmaz et al. are} outperformed by PRFI in both metrics.
                    %
                    \blt[faster than tentative]\arevtwo{Note again that  PRFI is considerably faster than PRFI (Tentative), which again corroborates the efficacy of the proposed approach.}
                \end{bullets}%
            \end{bullets}%

            \blt[distance]
            \begin{bullets}%
                \blt[description] The next experiment studies the influence of the distance $\distbsue=\|\locbs-\locue\|$ on $\meanconn$ and $\probfail$.
                \begin{bullets}%
                    \blt[figure] To this end, Fig.~\ref{fig:static-vs-dist} plots
                    $\meanconn$ and $\probfail$ vs. $\distbsue$.
                    %
                    \begin{figure}
                        \centering
                        \includegraphics[width=\linewidth]{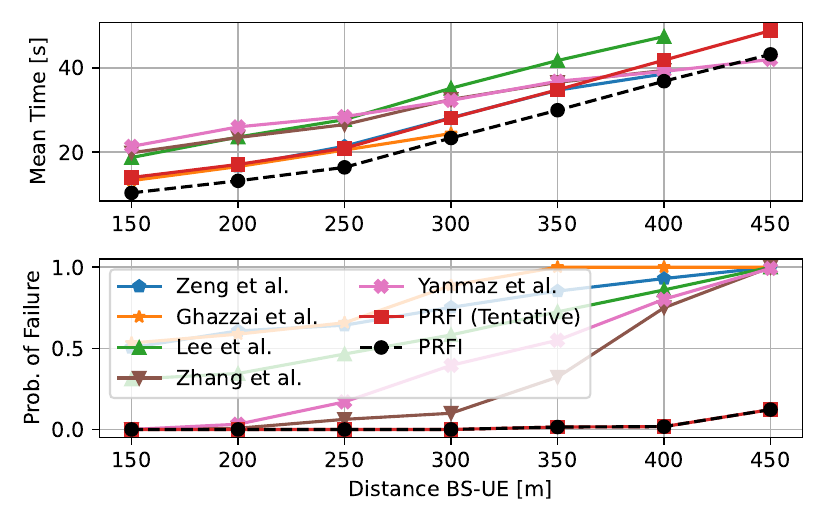}
                        \caption{$\meanconn$ and $\probfail$ vs. mean $\|\locue-\locbs\|$ ($\minuerate$ = 90 Mbps).}
                        \label{fig:static-vs-dist}
                    \end{figure}
                    \blt[generate dist bs ue] \arev{
                        In this figure, for a given value $\dist$ on the x-axis, $\distbsuemin=\dist-20$~m and $\distbsuemax=\dist+20$~m.}
                \end{bullets}
                \blt[observation]
                \begin{bullets}%
                    \blt[trends]
                    \begin{bullets}%
                        \blt[Overall increasing]\arevtwo{As expected, both $\meanconn$ and $\probfail$ increase with $\distbsue$.}
                        %
                    \end{bullets}%
                    %
                    \blt[prfi lowest]\arevtwo{
                        PRFI is seen to yield the best performance by far even for large values of~$\distbsue$.}
                \end{bullets}%
            \end{bullets}%
        \end{bullets}

        \begin{changesv2}

            \blt[optimality]
            \begin{bullets}%
                \blt[recall]The next experiment quantifies how often the PRFI path is guaranteed to be optimal.
                \blt[metric]The metric adopted to this end is the \emph{probability of guaranteed optimality}, which is the fraction of  MC realizations in which the optimality conditions   in Sec.~\ref{sec:fromwptotrajstaticue} hold.
                \blt[description]Most of the times, this probability is nearly 1. To reduce this metric to more interesting values, the problem is made more challenging by
                \begin{bullets}%
                    \blt[bs ue locs] placing the BS and  UE at opposite corners and
                    \blt[disable] disabling a randomly selected set of flight grid points.
                \end{bullets}%
                \blt[figure]Fig.~\ref{fig:static-optimality} plots the probability of guaranteed optimality and infeasibility vs. $\minuavrate$ for different fractions of disabled grid points. The probability of infeasibiligy reflects the fraction of MC realizations in which no feasible trajectory exists.
                %
                \begin{figure}%
                    \centering
                    \includegraphics[width=\linewidth]{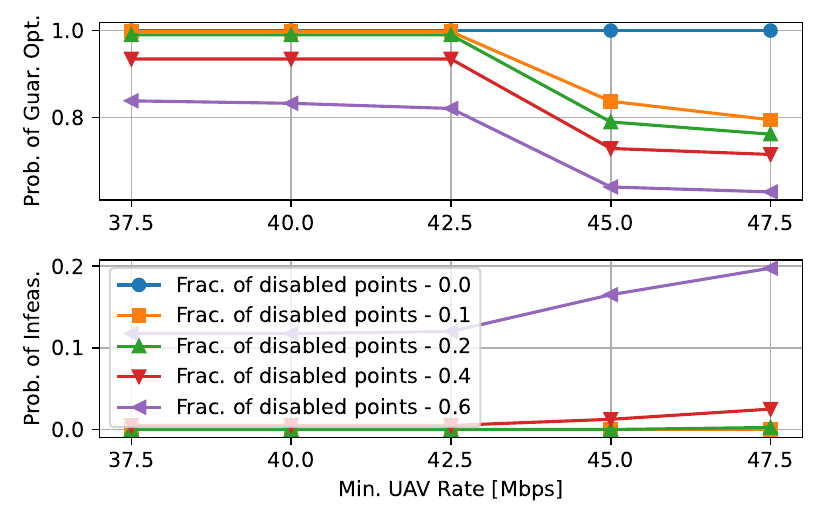}
                    \caption{
                        \arevtwo{Probability of guaranteed optimality and infeasibility vs. $\minuavrate$ (The height of each building is uniformly distributed between 20~m and 40~m, $\numgridptz=4$, $\minuerate=10$ Mbps)}.}
                    \label{fig:static-optimality}
                \end{figure}%
                \blt[observation]It is observed that PRFI yields an optimal path unless the fraction of disabled points and $\minuavrate$ are \emph{simultaneously} sufficiently large. It is also seen that, when an optimal path is not returned, it is because no feasible path even exists. 
            \end{bullets}%
        \end{changesv2}
    \end{bullets}%
\end{bullets}%

\subsection{Moving UE}
\begin{bullets}%
    \blt[Overview]This section studies the performance of  PRFI for moving UEs; cf. Sec.~\ref{sec:moving}.
    \blt[objective]Due to space restrictions, the focus will be on  maximizing the  transferred data. The UE
    \blt[simulation setup]%
    \begin{bullets}%
        \blt[ue path] follows a random trajectory
        of 300 seconds at a speed of 2 m/s whose  initial position is generated as the UE location in Sec.~\ref{sec:sim:static}.
        \blt[building height] Throughout, the height of each building  at each MC realization is uniformly distributed between 20~m and 75~m.

        \begin{extendedonly}
            An example of such a trajectory is shown in Fig. \ref{fig:trajectory-ue}.
            \begin{figure}
                \centering
                \includegraphics[width=.9\linewidth, trim={35cm 0cm 35cm 9cm}, clip]{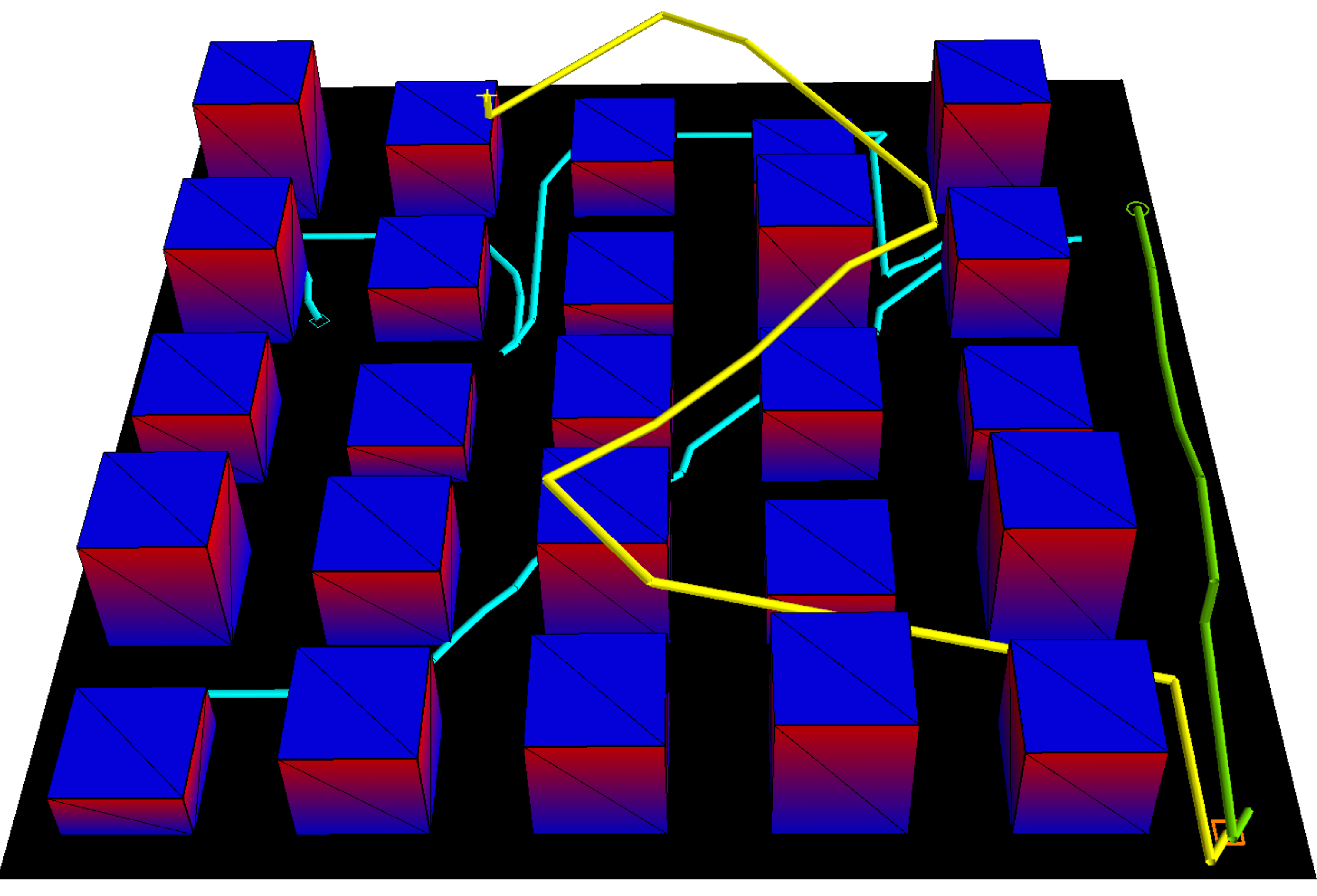}
                \captionof{figure}{An example of the UAV trajectories provided by Algorithm~\ref{algo:prfi-moving}.}
                \label{fig:trajectory-ue}
            \end{figure}
        \end{extendedonly}
    \end{bullets}%

    \begin{changesv2}
        \blt[benchmarks] For the setup at hand, the benchmarks in the previous section are adapted so that
        \begin{bullets}
            \blt[overall] the UAVs  fly on $\grid$, since this simplifies the computation of the transferred data.
            \begin{bullets}
                \blt[adjacent] At every time step, each UAV either stays at its current grid point or flies to an adjacent one.
                \blt[take off]In all algorithms, the UAVs  take off vertically at the BS until   $\ptbsabove$.
            \end{bullets}
            \blt[description] The rest of the trajectory is as follows:
            \begin{itemize}
                \item Zeng et al. \cite{zeng2016relaying}: at  time step $\indtimestep$, the UAV flies towards the adjacent grid point that is closest to $\ptmid[\indtimestep+1]\define(\ptbsabove+\ptueabove[\indtimestep+1])/2$, where $\ptueabove[\indtimestep+1]\define[\locxue[\indtimestep+1],\locyue[\indtimestep+1],\heighttop]\transpose$.
                \item Ghazzai et al. \cite{ghazzai2018dual}: at each time step, the UAV flies to the adjacent grid point that maximizes the UE rate predicted by  the channel model in \cite{alhourani2014lap}.
                \item Lee et al. \cite{lee2022trajectory}: at time step $\indtimestep$, UAV-1 flies to the adjacent grid point that is closest to $\ptmid[\indtimestep+1]$ and UAV-2 flies to the adjacent grid point that is closest to $\ptueabove[\indtimestep+1]$.
                \item Zhang et al. \cite{zhang2018multi}:  UAV-1 remains at $\ptbsabove$ and, at each time step, UAV-2 flies to the adjacent grid point that maximizes the UE rate in the next time step.
                \item Yanmaz et al. \cite{yanmaz2022positioning}: at time step $\indtimestep$, UAV-1 remains at $\ptbsabove$ and UAV-2 flies to the adjacent grid point that is closest to $\ptueabove[\indtimestep+1]$.
                \item Benchmark 6: cf. Appendix~\ref{sec:segment-benchmark}. The parameters of this benchmark are $\numlocstoreplan=15$  and $\numknownuelocs=17$.
            \end{itemize}
            These algorithms are compared with PRFI (Tentative) and PRFI, which respectively correspond to Algorithms~\ref{algo:pathone-movingue} and~\ref{algo:prfi-moving}.%
        \end{bullets}%
    \end{changesv2}

    \blt[experiments]
    \begin{bullets}%
        \blt[rate vs time]
        \begin{bullets}%
            \blt[figure] Fig.~\ref{fig:rate-vs-time-90} plots the MC average of the  UE  rate vs. $\timeInstant$.
            \begin{figure}%
                \centering
                \includegraphics[width=0.9\linewidth]{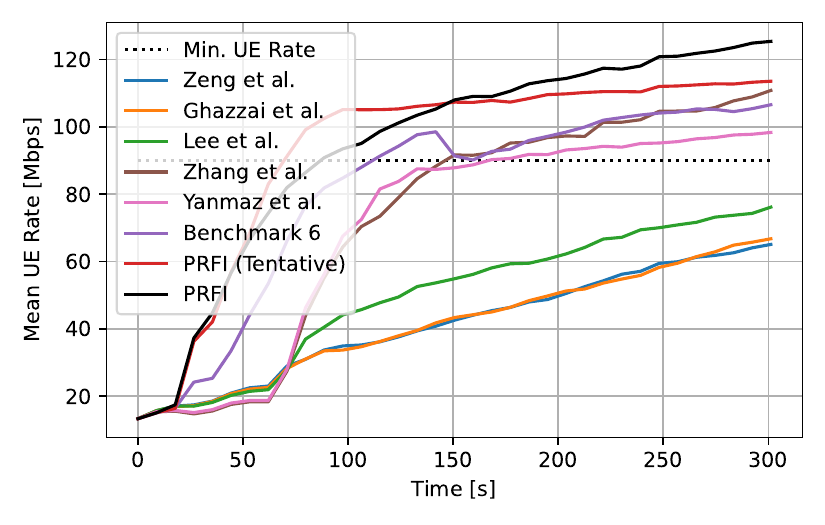}
                \caption{Mean UE rate vs. $\timeInstant$  ($\minuerate=90$ Mbps). }
                \label{fig:rate-vs-time-90}
            \end{figure}%
            \begin{changesv2}
                \blt[observation]%
                \begin{bullets}%
                    %
                    \blt[fastest minUErate] As expected,
                    \begin{bullets}%
                        \blt[tentative fastest] PRFI (Tentative) attains $\minuerate$ before PRFI but the latter yields a larger rate in the long term.
                        \blt[prfi]This is because PRFI (Tentative) aims at minimizing outage time whereas PRFI pursues the maximization of the total transfered data.%
                    \end{bullets}%
                    %
                    %
                \end{bullets}%
            \end{changesv2}
        \end{bullets}%

        \blt[transfered data vs time]
        \begin{bullets}%
            \blt[figure]Fig.~\ref{fig:data-vs-time} plots the total transferred data  vs. $\timeInstant$.
            \begin{figure}
                \centering
                \includegraphics[width=0.9\linewidth]{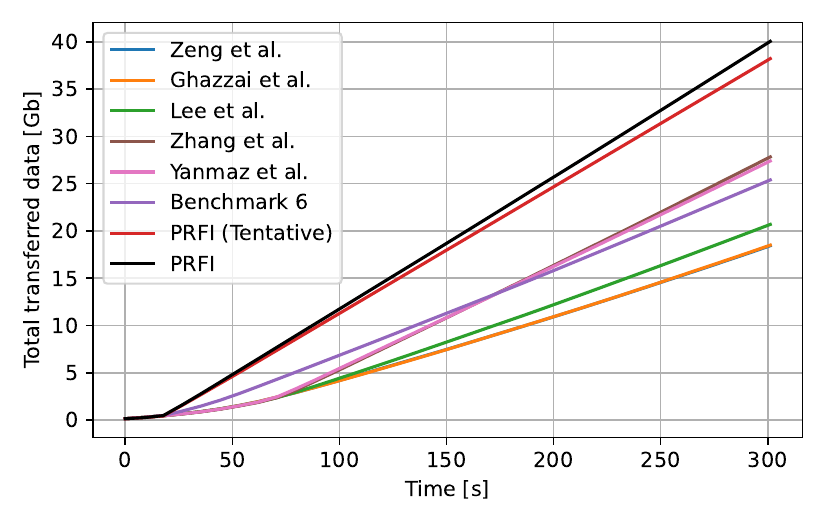}
                \caption{Total transferred data vs. $\timeInstant$ ($\distbsuemin=130$~m, $\distbsuemax=170$~m, $\minuerate=60$ Mbps).
                }
                \label{fig:data-vs-time}
            \end{figure}
            \begin{changesv2}
                \blt[observation]%
                \begin{bullets}%
                    \blt[data transferred]PRFI  offers the greatest slope, which implies that the  margin by which it outperforms its competitors increases as time progresses.
                    \blt[tentative]The difference between PRFI and PRFI (Tentative) is not very large, which suggests that the proposed initialization is nearly optimal and, therefore, one may consider bypassing Algorithm~\ref{algo:prfi-moving} to reduce computational cost.%
                \end{bullets}%
            \end{changesv2}
        \end{bullets}%

        \blt[prfi only]
        \begin{bullets}%
            \blt[description]To investigate how to set the parameters of PRFI,
            \begin{bullets}
                \blt[figure] Fig.~\ref{fig:vs-min-ue-rate-prfi-avg-rate} plots the average UE rate vs. $\minuerate$ for different numbers $\numconfpt$ of drawn CPs in the PR step.
                \begin{figure}[t]
                    \centering
                    \includegraphics[width=0.9\linewidth]{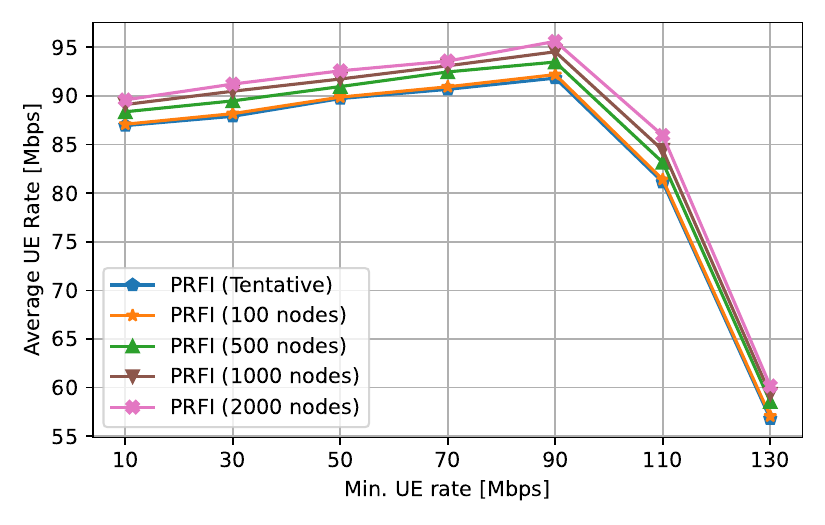}
                    \caption{Average UE rate vs. $\minuerate$ (800 MC realizations). }
                    \label{fig:vs-min-ue-rate-prfi-avg-rate}
                \end{figure}
                \blt[building absorption] To make differences between parameter values more conspicuous, an infinite  building absorption is adopted. %
            \end{bullets}%
            \blt[observation]
            \begin{bullets}%
                \blt[number of CPs]
                \begin{bullets}%
                    \blt[tradeoff av. rate-outage]As expected,  the greater $\numconfpt$, the higher the average UE rate,
                    %
                    \blt[diminishing returns]but  with diminishing returns: for example, the difference between $\numconfpt=100$ and $\numconfpt=1000$ is much more significant than between $\numconfpt=1000$ and $\numconfpt=2000$.%
                \end{bullets}%

                \blt[Influence of $\minuerate$]
                \arevtwo{Although $\minuerate$ does not affect the optimal path, it determines which suboptimal path PRFI returns.
                    To select it the best way, }
                \begin{bullets}%
                    \blt[low min rate]\arevtwo{observe from Fig.~\ref{fig:vs-min-ue-rate-prfi-avg-rate} that the average UE rate   increases slowly when $\minuerate$ is below a certain value}
                    \blt[high min rate]\arevtwo{and decreases quickly afterwards.} This suggests that it is preferable to select a reasonably small $\minuerate$ in practice.
                \end{bullets}%
            \end{bullets}%
        \end{bullets}%

        \blt[rate vs dist bs ue]
        \begin{bullets}%
            \blt[description] Fig.~\ref{fig:vs-dist-bs-ue} aims at analyzing  the influence of the initial $\distbsue$ on the mean UE rate and fraction of outage time.
            \blt[prob. outage]\arevtwo{The latter is defined as the fraction of time where $\uerate(\confpt(t),\locue(t))<\minuerate$.}
            \begin{bullets}%
                \blt[figure]
                \begin{figure}
                    \centering
                    \includegraphics[width=0.9\linewidth]{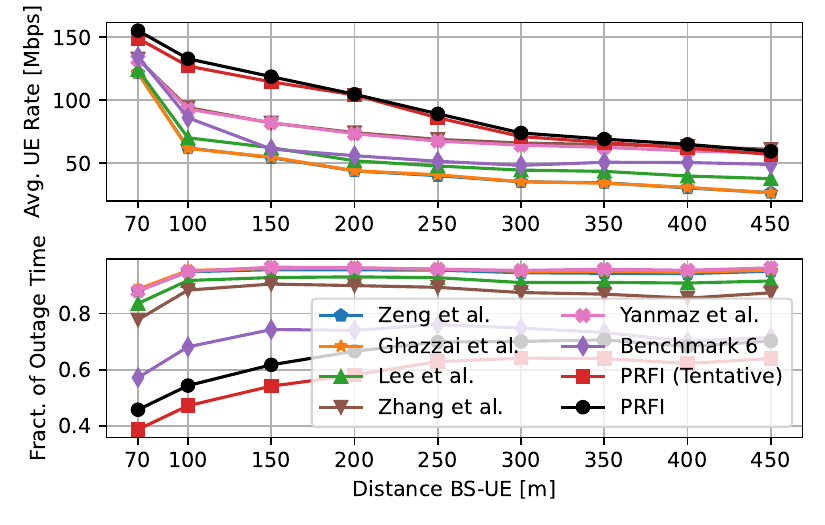}
                    \caption{Influence of the initial $\distbsue$ on performance ($800$ realizations, $\minuerate=110$~Mbps).}
                    \label{fig:vs-dist-bs-ue}
                \end{figure}%
                \blt[generate dist bs ue] For a given value $\dist$ on the x-axis, $\distbsuemin=\dist-20$~m and $\distbsuemax=\dist+20$~m.%
            \end{bullets}%
            \blt[observation]
            \begin{bullets}%
                %
                \blt[rel. benchmarks]~As expected, PRFI provides the highest average UE rate and lowest fraction of outage time.
                %
                \blt[rel. tentative]\arevtwo{Note, however, that PRFI (Tentative) attains a lower fraction of outage time than PRFI. This is because  the former precisely targets this objective; cf. Sec.~\ref{sec:moving-tentative}.
                }
                %
                %
                %
            \end{bullets}%
        \end{bullets}%

        \begin{extendedonly}
            \blt[vs building height, exp7052]
            \begin{bullets}%
                \blt[description] The final experiment studies the influence of the environment. To this end,
                \begin{bullets}%
                    \blt[figure] Fig.~\ref{fig:vs-building-height} depicts the average UE rate vs. the mean building height.
                    \blt[mean height]For each $h$ on the horizontal axis, the height of each building at each MC realization is uniformly distributed between $h - 20$~m and $h + 20$~m.
                    %
                    %
                    \begin{figure}
                        \centering
                        \includegraphics[width=.9\linewidth]{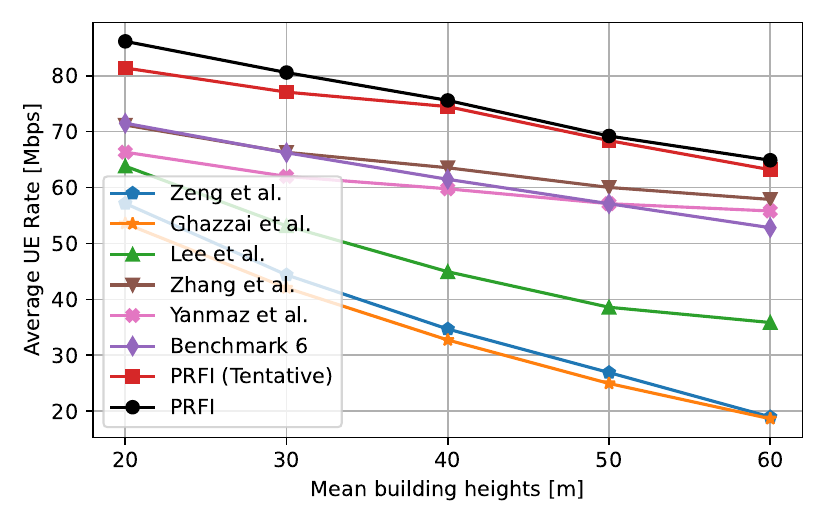}
                        \caption{Average UE rate vs. mean building height ($\minuerate=100$ Mbps).}
                        \label{fig:vs-building-height}
                    \end{figure}%
                \end{bullets}%
                \blt[observation]
                \begin{bullets}%
                    \blt[increase height] As expected, a greater \arevtwo{mean} height of the buildings results in a performance degradation.
                    \begin{bullets}%
                        \blt[explain] This is because a greater building size constrains the possible trajectories and impairs the propagation conditions by decreasing  channel gain, which limits  the locations where the UAVs can provide $\minuerate$ to the UE and the CPs where  the UAVs receive $\minuavrate$.
                    \end{bullets}%
                    \blt[outperform] Despite this fact, the proposed algorithm widely outperforms the  benchmarks.
                \end{bullets}
            \end{bullets}
        \end{extendedonly}
    \end{bullets}%
\end{bullets}%

\section{Conclusions}
\label{sec:conclusions}

\begin{bullets}%
    \blt[purpose]This paper developed a framework for path planning of multiple aerial relays that approximately optimizes communication metrics while accommodating arbitrary  constraints on the flight region.
    %
    %
    \begin{changes}%
        \blt[proposed approach]The idea is to build upon the celebrated PR algorithm, which finds a shortest path through a random graph of CPs. To cope with the need for a large number of CPs in plain PR,
        \begin{bullets}%
            \blt[pr modifiction] a modification was proposed in which the CPs are drawn around a tentative path.
        \end{bullets}%
        \blt[cases]This approach was applied to serve both static and moving users with any number of UAVs. To this end, heuristic rules leading to tentative paths with theoretical guarantees were proposed.
    \end{changes}%
    \blt[results] Numerical results demonstrate the merits of the proposed algorithms.
    \blt[future work] Future work will
    \begin{bullets}%
        \blt[more than 2] \arevtwo{investigate alternative sampling strategies for PR}
        \blt[data collection] \arevtwo{and approaches for data collection from  terrestrial wireless sensors using UAVs.}
    \end{bullets}%
\end{bullets}%

\printmybibliography
\appendices

\begin{changes}
    \section{Influence of the objective on the trajectory}\label{sec:influenceobj}
    \cmt{intro} This appendix analyzes how the selection of the objective function determines the optimal  trajectory.

\input{diff_obj.tex}
    
\end{changes}

\balance

\section{Proof of Theorem~\ref{thm:guarantee}}
\label{sec:proof:guarantee}

\begin{extendedonly}
    \begin{bullets}%
        \blt[given path] Let $\pathsingle_2\define\{\loc_2[0],\loc_2[1],\ldots,\loc_2[\numwp_0-1]\}$ be the path for UAV-2 returned by Algorithm~\ref{algo:trajtwo}.
        \blt[max num lift] Observe that, for a given $\higherbheight$, there is a maximum number of times that $\pathsingle_2$ can be lifted before the lifting operator returns the same path as its input, that is, $\liftfun^{(\indlift)}(\pathsingle_2) = \liftfun^{(\indlift+1)}(\pathsingle_2)$ for a sufficiently large $\indlift$. Let $\numlift$ denote the smallest of such values of $\indlift$, i.e., $
            \numlift\define \min\{\indlift\in\mathbb{N}:\liftfun^{(\indlift)}(\pathsingle_2)=\liftfun^{(\indlift+1)}(\pathsingle_2)\}$.

        \blt[UAV-2 path lifted to top] If Algorithm~\ref{algo:trajone} fails to provide a valid path, it necessarily fails to find a valid path at all iterations and, in particular, at the $\numlift$-th iteration. Therefore, to prove the theorem, it suffices to show that the algorithm succeeds if it reaches the $\numlift$-th iteration, which is the worst case.
        Equivalently, it has to be shown that Algorithm~\ref{algo:trajone} can find a path to $\mathbb{N} \times \destsetuo$  through the extended graph corresponding to $\liftfun^{(\numlift)}(\pathsingle_2)$. 
        \blt[combined path feasible] Since Algorithm~\ref{algo:trajone} is based on a shortest path algorithm,
        it will return a path to $\mathbb{N} \times \destsetuo$ if at least one such a path exists. 
        Therefore, to prove the theorem, it suffices to find \textit{any} path to $\mathbb{N} \times \destsetuo$. 

        \blt[proof structure]To establish the existence of such a path in the extended graph, the rest of the proof will design a path for UAV-1 and show that the resulting  combined path $\trajdouble\define\{\exconfpt[0],\ldots,\exconfpt[\numwp-1]\}$ is valid. With  $\exconfpt[\indwp]=[
            \exloc_1[\indwp],\exloc_2[\indwp]
            ]$, this means that 
            $\exloc_1[\numwp-1]\in\destsetuo$ and 
            the following conditions hold:
        \begin{itemize}
            \item[C1:] The rate between the BS and UAV-1 is at least $2\minuavrate$, i.e., $\capacity(\locbs,\exloc_1[\indwp])\geq 2\minuavrate$ for all $\indwp$, and
            \item[C2:] The rate from UAV-1 to UAV-2 is at least $\minuavrate$, i.e., $\capacity(\exloc_1[\indwp],$ $\exloc_2[\indwp])\geq \minuavrate$ for all $\indwp$.
        \end{itemize}%
        \blt[path separation]To simplify the exposition, the path $\liftfun^{(\numlift)}(\pathsingle_2)$ will be separated into the following subpaths:        
            \begin{salign}[eq:subpaths]
                \trajuavsecup   \define \{  & \loc_2[0],\liftfun^{(1)}(\loc_2[0]),\ldots,\liftfun^{(\numliftup)}(\loc_2[0])\},\label{eq:uav2path1}       \\
                \nonumber
                \trajuavsectop   \define \{ & \liftfun^{(\numliftup)}(\loc_2[0]),\loc_2^{(\numlift)}[1],\ldots,\loc_2^{(\numlift)}[\numwp_{\numlift}-2], \\&\liftfun^{(\numliftdown)}(\loc_2[\numwp_{0}-1])\},\label{eq:uav2path2} \\
                \trajuavsecdown  \define \{ & \liftfun^{(\numliftdown)}(\loc_2[\numwp_{0}-1]),\ldots,\loc_2[\numwp_0-1]\}\label{eq:uav2path3},
            \end{salign}        
        where 
        \begin{salign}                    
        \numliftup\define&   \min\{\indlift\in\mathbb{N}:\liftfun^{(\indlift)}(\loc_2[0])=\liftfun^{(\indlift+1)}(\loc_2[0])\}\\ \numliftdown\define&  \min\{\indlift\in\mathbb{N}:\liftfun^{(\indlift)}(\loc_2[\numwp_0-1])=\liftfun^{(\indlift+1)}(\loc_2[\numwp_0-1])\},
        \end{salign}
            and $\trajuavsectop$ is the shortest path from $\liftfun^{(\numliftup)}(\loc_2[0])$ to $\liftfun^{(\numliftdown)}(\loc_2[\numwp_0-1])$ in the graph of Sec.~\ref{sec:static:pathuav2}.
        For each subpath in \eqref{eq:subpaths}, a subpath will be designed for UAV-1. The resulting combined path will be   valid because the last subpath  ends at $\destsetuo$  by construction and each combined subpath will be shown to satisfy C1 and C2.         

        \begin{figure}
            \centering
            \includegraphics[width=\linewidth, trim={2.4cm 3.5cm 1.3cm 3.3cm},clip]{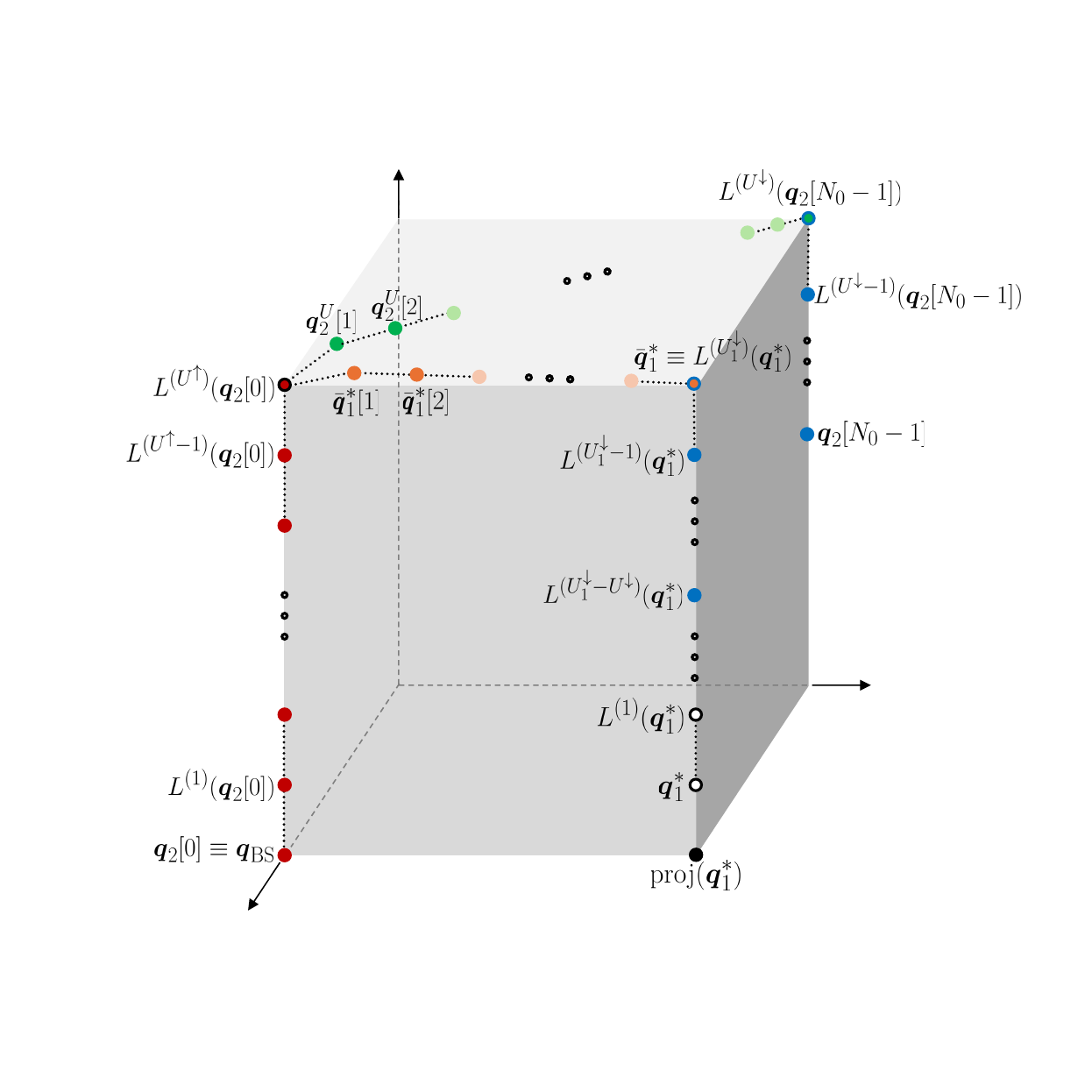}{}
            \caption{Illustration of the trajectory constructed in the proof of \thref{thm:guarantee}.}
            \label{fig:enter-label}
        \end{figure}
        \blt[part 1 - take off]\textbf{Take-off subpath ($\pathsingle^\uparrow$):}
        \begin{bullets}
            \blt[choose path] In the considered combined path, when UAV-2 follows $\trajuavsecup$, UAV-1 follows $\trajuavsecup$ as well. Recall that the minimum separation between the UAVs was disregarded in \eqref{eq:trajproblem} for simplicity. Thus, the combined path of UAV-1 and UAV-2 is given by
            \begin{align}
                \nonumber
                \trajdouble^{\uparrow}\define\Big\{ & \Big[\loc_2[0],                           \loc_2[0]\Big],                             \left[\liftfun^{(1)}(\loc_2[0]),           \liftfun^{(1)}(\loc_2[0])\right],           \ldots ,                                  & \\
                                                    & \left[\liftfun^{(\numliftup)}(\loc_2[0]),                                                                                                                             \liftfun^{(\numliftup)}(\loc_2[0])\right]\Big\}.
            \end{align}
            \blt[showing feasible]
            \begin{bullets}
                \blt[c1 - rate btw bs-uav1] C1: By hypothesis,
                \begin{align}
                    \higherbheight & \leq\sqrt{[\capacitydist\inv(2\minuavrate)]^2-[\capacitydist\inv(2\minuavrate+\minuerate)]^2}, \label{eq:heightassumption}
                \end{align}
                which implies that $ \higherbheight^2            \leq[\capacitydist\inv(2\minuavrate)]^2$ or, equivalently, $
                    \higherbheight$ $\leq\capacitydist\inv(2\minuavrate)$.
                Thus, since $\loc_2[0]=\locbs$, $\forall\exloc_{1}\in\trajuavsecup$, it follows that $\|\locbs-\exloc_1\|\leq\|\locbs-\liftfun^{(\numliftup)}(\loc_2[0])\|\leq\higherbheight  \leq\capacitydist\inv(2\minuavrate)
                $.
                Noting that $\capacity$ is a decreasing function of the distance, yields
                \begin{align}
                    \capacity(\locbs,\exloc_1) & \geq\capacity\left(\locbs,\liftfun^{(\numliftup)}(\loc_2[0])\right) \geq2\minuavrate\label{eq:rateabovebs},
                \end{align}
                which establishes C1.

                \blt[c2 - rate uav1-uav2] C2: trivial.
            \end{bullets}%
        \end{bullets}%

        \blt[part 2 - on the top layer] \textbf{Top subpath $(\trajuavsectop)$:} The combined path will be divided into two parts:

        \begin{bullets}%
            \blt[uav2 moves] \textit{Part 1:}
            \begin{bullets}
                \blt[choose path] UAV-1 stays at $\liftfun^{(\numliftup)}(\loc_2[0])$ while UAV-2 follows $\trajuavsectop$ in \eqref{eq:uav2path2}. Note that, as per $\trajuavsectop$, UAV-2 flies at constant height $\higherbheight$ from $\liftfun^{(\numliftup)}(\loc_2[0])$ to $\liftfun^{(\numliftdown)}(\loc_2[\numwp_{0}-1])$. The combined path is then
                \begin{align}
                    \nonumber
                     & \overset{{\rightarrow}}{\trajdouble_1}\define\Big\{\left[\liftfun^{(\numliftup)}(\loc_2[0]), \liftfun^{(\numliftup)}(\loc_2[0])\right],                                        \\
                    \nonumber
                     & \left[\liftfun^{(\numliftup)}(\loc_2[0]), \loc_2^{(\numlift)}[1]\right],   \ldots,  \left[\liftfun^{(\numliftup)}(\loc_2[0]), \loc_2^{(\numlift)}[\numwp_{\numlift}-2]\right], \\
                     & \left[\liftfun^{(\numliftup)}(\loc_2[0]),  \liftfun^{(\numliftdown)}(\loc_2[\numwp_{0}-1])\right]\Big\}.
                \end{align}

                \blt[showing feasible]
                \begin{bullets}
                    \blt[c1 - rate bs \ra uav1] C1: It follows from \eqref{eq:rateabovebs}.

                    \blt[c2 - rate uav1 \ra uav2] C2: To prove that $\capacity(\liftfun^{(\numliftup)}(\loc_2[0]),\exloc_2)\geq\minuavrate,\forall\exloc_2\in\trajuavsectop$,
                    \begin{bullets}
                        \blt[dist to above des uav2] it can be observed that $\trajuavsectop$ is the shortest path from $\liftfun^{(\numliftup)}(\loc_2[0])$ to $\liftfun^{(\numliftdown)}(\loc_2[\numwp_0-1])$, hence $\forall\exloc_2\in\trajuavsectop$,
                        \begin{align}
                            \|\liftfun^{(\numliftup)}(\loc_2[0])-\exloc_2\| \leq\|\liftfun^{(\numliftup)}(\loc_2[0])-\liftfun^{(\numliftdown)}(\loc_2[\numwp_0-1])\|.
                            \label{eq:capontopdes}
                        \end{align}
                        Hence, it suffices to prove that $\|\liftfun^{(\numliftup)}(\loc_2[0])-\liftfun^{(\numliftdown)}$ $(\loc_2[\numwp_0-1])\| \leq \capacitydist^{-1}(\minuavrate)$.
                        \blt[rates at des uav1] To this end, recall that  $\loc_2[\numwp_0-1]\in\destsetut$. Thus, when UAV-2 is at $\loc_2[\numwp_0-1]$, there exists a location $\locdesuavo=[q_{1x}^*,q_{1y}^*,q_{1z}^*]\transpose$ for UAV-1 such that
                        \begin{align}
                            \label{eq:capontopdesdf}
                            \locdesuavo  \in\rateset(\locbs,2\minuavrate+\minuerate)  \cap\rateset(\loc_2[\numwp_0-1],\minuavrate+\minuerate).
                        \end{align}
                        It follows that $ \capacity(\locbs,\locdesuavo)  \geq
                            2\minuavrate+\minuerate$ and $
                            \capacity(\loc_2[\numwp_0-1],\locdesuavo)\geq
                            \minuavrate+\minuerate, \label{eq:ratedesuav1}$
                        or, equivalently, that
                        \begin{salign}
                            \|\locbs-\locdesuavo\| &\leq\capacitydist\inv(2\minuavrate+\minuerate), \text{ and}
                            \label{eq:distdesuavone}
                            \\
                            \|\loc_2[\numwp_0-1]-\locdesuavo\| &\leq\capacitydist\inv(\minuavrate+\minuerate).
                            \label{eq:distdesuavonetwo}
                        \end{salign}

                        \blt[triangle inequality] Let
                        $\locdesuavotop=[q_{1x}^{*},q_{1y}^{*},\higherbheight]\transpose$
                        and consider three points at the same height
                        $\higherbheight$:
                        $\liftfun^{(\numliftup)}(\loc_2[0]),\liftfun^{(\numliftdown)}(\loc_2[\numwp_{0}-1])$,
                        and $\locdesuavotop$. From the triangle inequality, one
                        has
                        \begin{align}
                            \nonumber
                             & \|\liftfun^{(\numliftup)}(\loc_2[0])  - \liftfun^{(\numliftdown)}(\loc_2[\numwp_{0}-1])\|                                                                 \\
                             & \leq\|\liftfun^{(\numliftup)}(\loc_2[0]) - \locdesuavotop\| + \|\liftfun^{(\numliftdown)}(\loc_2[\numwp_{0}-1])-\locdesuavotop\|.\label{eq:ineqsingtosum}
                        \end{align}
                        Since $\liftfun^{(\numliftup)}(\loc_2[0])=\liftfun^{(\numliftup)}(\locbs)$ and $\locdesuavotop$ are at the same height, it follows that
                        \begin{equation}
                            \|\liftfun^{(\numliftup)}(\loc_2[0]) - \locdesuavotop\| \leq\|\locbs-\locdesuavo\|  \overset{\eqref{eq:distdesuavone}}{\leq} \capacitydist^{-1}(2\minuavrate+\minuerate).
                            \label{eq:distineq1}
                        \end{equation}
                        Similarly, since $\liftfun^{(\numliftdown)}(\loc_2[\numwp_{0}-1])$ and $\locdesuavotop$ are at the same height,
                        \begin{subequations}
                            \begin{align}
                                \|\liftfun^{(\numliftdown)}(\loc_2[\numwp_{0}-1]) -\locdesuavotop\| & \leq \|  \loc_2[\numwp_{0}-1]-\locdesuavo\|                                                                \\
                                                                                                    & \overset{\eqref{eq:distdesuavonetwo}}{\leq}\capacitydist^{-1}(\minuavrate+\minuerate).\label{eq:distineq2}
                            \end{align}
                        \end{subequations}
                        From \eqref{eq:distineq1} and  \eqref{eq:distineq2},
                        \begin{align}
                            \nonumber
                            \|\liftfun^{(\numliftup)}(\loc_2[0]) - \locdesuavotop\| + \|\liftfun^{(\numliftdown)}(\loc_2[\numwp_{0}-1])-\locdesuavotop\| \\
                            \leq\capacitydist^{-1}(2\minuavrate+\minuerate) +  \capacitydist^{-1}(\minuavrate+\minuerate).\label{eq:ineqsumdist}
                        \end{align}

                        \blt[lemma inequality]The following result provides an upper bound for the right-hand side:
                        \begin{mylemma}
                            \label{lem:capac}
                            If Eq.~\eqref{lemma1:hypothesis} holds,
                            then
                            \begin{equation}
                                \label{eq:capinvsumbound}
                                \capacitydist^{-1}(2\minuavrate+\minuerate) + \capacitydist^{-1}(\minuavrate+\minuerate)<\capacitydist^{-1}(\minuavrate).
                            \end{equation}
                        \end{mylemma}
                        \begin{IEEEproof}
                            See Appendix~\ref{sec:proof:capac}.
                        \end{IEEEproof}

                        \blt[rate uav1 \ra uav2] From \eqref{eq:capontopdes}, \eqref{eq:ineqsingtosum}, \eqref{eq:ineqsumdist} and Lemma  \ref{lem:capac}, it holds that, $\forall \exloc_2\in\trajuavsectop$,
                        \begin{subequations}
                            \begin{align}
                                \nonumber
                                 & \|\liftfun^{(\numliftup)}(\loc_2[0]) - \exloc_2\|\overset{\eqref{eq:capontopdes}}{\leq}\|\liftfun^{(\numliftup)}(\loc_2[0]) - \liftfun^{(\numliftdown)}(\loc_2[\numwp_{0}-1])\|                                 \\
                                \nonumber
                                 & \overset{\eqref{eq:ineqsingtosum}}{\leq}\|\liftfun^{(\numliftup)}(\loc_2[0]) - \locdesuavotop\| + \|\liftfun^{(\numliftdown)}(\loc_2[\numwp_{0}-1]) -\locdesuavotop\|                                           \\
                                \nonumber
                                 & \overset{\eqref{eq:ineqsumdist}}{\leq}\capacitydist^{-1}(2\minuavrate+\minuerate) +  \capacitydist^{-1}(\minuavrate+\minuerate) \overset{(\text{Lemma \ref{lem:capac}})}{\leq} \capacitydist^{-1}(\minuavrate).
                            \end{align}
                        \end{subequations}
                    \end{bullets}
                \end{bullets}
                \blt[conclusion] Therefore, $\capacity(\liftfun^{(\numliftup)}(\loc_2[0]),\exloc_2) \geq\minuavrate$, which establishes C2.
            \end{bullets}

            \blt[uav1 moves] \textit{Part 2:}
            \begin{bullets}%
                \blt[choose path] Next, UAV-1 follows a shortest path from $\liftfun^{(\numliftup)}$ $(\loc_2[0])$ to $\locdesuavotop$ in $\grid$ while UAV-2 stays at $\liftfun^{(\numliftdown)}(\loc_2[\numwp_{0}-1])$.
                \begin{bullets}%
                    \blt[adjacent grid points] Since the flight grid is dense enough, there exists a sequence of adjacent grid points $\bar{\pathsingle}\define\big\{\liftfun^{(\numliftup)}(\loc_2[0]),\locdesuavotop[1],\locdesuavotop[2],\ldots,\locdesuavotop[\bar{\numwp}-2],\locdesuavotop\big\}$ that are sufficiently close to the line segment from $\liftfun^{(\numliftup)}(\loc_2[0])$ to $\locdesuavotop$ so that the rate between the BS and UAV-1 on this path will be at least $\capacity(\loc_2[0],\locdesuavotop)$.
                    \blt[combined path] The combined path is then
                    \begin{align}
                        \nonumber
                         & \overset{{\rightarrow}}{\trajdouble_2}\define
                        \Big\{\left[\liftfun^{(\numliftup)}(\loc_2[0]), \liftfun^{(\numliftdown)}(\loc_2[\numwp_{0}-1])\right],          \\
                         & \left[\locdesuavotop[1],                      \liftfun^{(\numliftdown)}(\loc_2[\numwp_{0}-1])\right], \ldots, \\
                        \nonumber
                         & \left[\locdesuavotop[\bar{\numwp}-2] ,          \liftfun^{(\numliftdown)}(\loc_2[\numwp_{0}-1])\right],
                        \left[\locdesuavotop,                    \liftfun^{(\numliftdown)}(\loc_2[\numwp_{0}-1])\right]\Big\}.
                    \end{align}
                \end{bullets}

                \blt[showing feasible]
                \begin{bullets}%
                    \blt[c1 rate bs \ra uav1] C1:
                    \begin{bullets}%
                        \blt[distances] Since $\bar{\pathsingle}$ is the set of adjacent grid points approximately on the line segment between $\liftfun^{(\numliftup)}(\loc_2[0])$ and $\locdesuavotop$, it holds that, $\forall\exloc_1\in\bar{\pathsingle}$,
                        \begin{equation}
                            \|\liftfun^{(\numliftup)}(\loc_2[0]) -\exloc_1\|\leq\|\liftfun^{(\numliftup)}(\loc_2[0]) -\locdesuavotop\|.
                            \label{eq:distontop}
                        \end{equation}
                        Squaring both sides and adding $\|\locbs - \liftfun^{(\numliftup)}(\loc_2[0])\|^2$ yields
                        \begin{equation}
                            \begin{split}
                                &\|\locbs - \liftfun^{(\numliftup)}(\loc_2[0])\|^2  + \|\liftfun^{(\numliftup)}(\loc_2[0]) - \exloc_1\|^2\\
                                & \leq \|\locbs - \liftfun^{(\numliftup)}(\loc_2[0])\|^2 + \|\liftfun^{(\numliftup)}(\loc_2[0]) - \locdesuavotop\|^2.
                            \end{split}
                        \end{equation}
                        \blt[greater than rate at des] From Pythagoras' theorem, $\forall \exloc_1\in\bar{\pathsingle}, \|\locbs - \exloc_1\|^2 \leq \|\locbs  - \locdesuavotop\|^2$, which in turn implies that
                        \begin{align}
                            \capacity(\locbs,\exloc_1) \geq \capacity(\locbs, \locdesuavotop).\label{eq:capuav1top}
                        \end{align}
                        \blt[projection] Let $\text{proj}(\loc)\define[x,y,\loczbs]\transpose$ be the projection of $\loc=[x,y,z]\transpose$ on the horizontal plane containing the BS. Then, $\text{proj}(\locdesuavo)\equiv\text{proj}(\locdesuavotop)$. From Pythagoras' theorem,
                        \begin{subequations}
                            \begin{align}
                                 & \|\locbs-\locdesuavotop\|^2 \leq \higherbheight^2 + \|\locbs-\text{proj}(\locdesuavotop)\|^2                                                                                                \\
                                 & = \higherbheight^2 + \|\locbs-\text{proj}(\locdesuavo)\|^2 \leq \higherbheight^2 + \|\locbs-\locdesuavo\|^2                                                                                 \\
                                 & \overset{\eqref{eq:distdesuavone}}{\leq} \higherbheight^2 + [\capacitydist\inv(2\minuavrate+\minuerate)]^2 \overset{\eqref{eq:heightassumption}}{\leq} [\capacitydist\inv(2\minuavrate)]^2.
                            \end{align}
                        \end{subequations}
                        Hence,
                        \begin{align}
                            \capacity(\locbs,\locdesuavotop) \geq 2\minuavrate.
                            \label{eq:rateabovedesuav1}
                        \end{align}
                        \blt[conclusion] From \eqref{eq:capuav1top} and \eqref{eq:rateabovedesuav1}, $\forall\exloc_1\in\bar{\pathsingle},$
                        \begin{equation}
                            \capacity(\locbs,\exloc_1)\geq 2 \minuavrate,\label{eq:ratebtwbsuav1}
                        \end{equation}
                    \end{bullets}
                    which proves C1.

                    \blt[c2 rate uav1 \ra uav2] C2:
                    \begin{bullets}%
                        \blt[largest distances] Since $\bar{\pathsingle}$ comprises  grid points sufficiently close to the line segment between $\liftfun^{(\numliftup)}(\loc_2[0])$ and $\locdesuavotop$, $\forall\exloc_1\in\bar{\pathsingle}$,
                        \begin{equation}
                            \|\exloc_1 -\locdesuavotop\|\leq\|\liftfun^{(\numliftup)}(\loc_2[0]) -\locdesuavotop\|.\label{eq:largestdist}
                        \end{equation}
                        \blt[triangle inequality] From the triangle inequality, $\forall\exloc_1\in\bar{\pathsingle}$,
                        \begin{salign}
                            \nonumber
                            & \|\exloc_1-\liftfun^{(\numliftdown)}(\loc_2[\numwp_{0}-1])\| \\
                            & \leq \|\exloc_1 -\locdesuavotop\| + \|\locdesuavotop-\liftfun^{(\numliftdown)}(\loc_2[\numwp_{0}-1])\|\\
                            & \overset{\eqref{eq:largestdist}}{\leq} \|\liftfun^{(\numliftup)}(\loc_2[0])-\locdesuavotop\| + \|\locdesuavotop-\liftfun^{(\numliftdown)}(\loc_2[\numwp_{0}-1])\|.
                            \label{eq:ineqdistuav1move}
                        \end{salign}
                        \blt[rate uav1 \ra uav2] Then, from \eqref{eq:ineqdistuav1move}, \eqref{eq:ineqsumdist}, and Lemma \ref{lem:capac}, $\forall\exloc_1\in\bar{\pathsingle},$
                        \begin{salign}
                            \nonumber
                            & \|\exloc_1 - \liftfun^{(\numliftdown)}(\loc_2[\numwp_{0}-1])\|\\
                            \nonumber
                            & \overset{\eqref{eq:ineqdistuav1move}}{\leq} \|\liftfun^{(\numliftup)}(\loc_2[0])-\locdesuavotop\| + \|\locdesuavotop-\liftfun^{(\numliftdown)}(\loc_2[\numwp_{0}-1])\| \\
                            \nonumber
                            & \overset{\eqref{eq:ineqsumdist}}{\leq}\capacitydist^{-1}(2\minuavrate+\minuerate)+\capacitydist^{-1}(\minuavrate+\minuerate) \overset{\text{Lemma \ref{lem:capac}}}{\leq} \capacitydist^{-1}(\minuavrate).
                        \end{salign}
                        Therefore,
                        \begin{align}
                            \capacity(\exloc_1,\liftfun^{(\numliftdown)}(\loc_2[\numwp_{0}-1])) \geq \minuavrate,\label{eq:ratebtwuavstop}
                        \end{align}
                    \end{bullets}%
                \end{bullets}%
                which proves C2.
            \end{bullets}%
        \end{bullets}%

        \blt[part 3 - landing] \textbf{Landing subpath $(\trajuavsecdown)$:}
        \begin{bullets}
            \blt[remind] The landing path of UAV-2 is  $\trajuavsecdown = \{\liftfun^{(\numliftdown)}(\loc_2[\numwp_{0}-1]),\liftfun^{(\numliftdown-1)}(\loc_2[\numwp_{0}-1]),\ldots,\loc_2[\numwp_0-1]\}$.
            \blt[num lift des UAV-1] Let $\numlift_1^{\downarrow}\define\min\{\indlift\in\mathbb{N}:\liftfun^{(\indlift)}(\locdesuavo)=\liftfun^{(\indlift+1)}(\locdesuavo)\define\locdesuavotop\}$.

            \blt[des uav1 lower] If $\numliftdown\leq\numlift_1^{\downarrow}$, $\locdesuavo$ has a lower altitude than $\loc_2[\numwp_0-1]$ and $\liftfun^{(\numlift_1^{\downarrow}-\numliftdown)}(\locdesuavo)$ and $\loc_2[\numwp_0-1]$ are at the same height. The case when $\numliftdown>\numlift_1^{\downarrow}$ can be proven similarly.
            \begin{bullets}%
                \blt[choose path]
                \begin{bullets}%
                    \blt[both uavs go down] With $\numliftdown\leq\numlift_1^{\downarrow}$, the UAVs descend simultaneously  following the combined subpath
                    \begin{align}
                        \nonumber
                        \trajdouble_{1}^{\downarrow}\define\Big\{ & \left[\liftfun^{(\numlift_1^{\downarrow})}(\locdesuavo), \liftfun^{(\numliftdown)}(\loc_2[\numwp_{0}-1])\right],            \\
                        \nonumber
                                                                  & \left[\liftfun^{(\numlift_1^{\downarrow}-1)}(\locdesuavo),\liftfun^{(\numliftdown-1)}(\loc_2[\numwp_{0}-1])\right], \ldots, \\
                        \nonumber
                                                                  & \left[\liftfun^{(\numlift_1^{\downarrow}-\numliftdown+1)}(\locdesuavo),\liftfun^{(1)}(\loc_2[\numwp_{0}-1])\right],         \\
                                                                  & \left[\liftfun^{(\numlift_1^{\downarrow}-\numliftdown)}(\locdesuavo), \loc_2[\numwp_{0}-1]\right]\Big\}.
                    \end{align}
                    \blt[uav1 goes down]
                    After that UAV-1 continues its descent  while UAV-2 stays at $\loc_2[\numwp_0-1]$. The second combined subpath is then
                    \begin{align}
                        \nonumber
                        \trajdouble_{2}^{\downarrow}\define\Big\{ & \left[\liftfun^{(\numlift_1^{\downarrow}-\numliftdown)}(\locdesuavo), \loc_2[\numwp_{0}-1]\right],                      \\
                        \nonumber
                                                                  & \left[\liftfun^{(\numlift_1^{\downarrow}-\numliftdown-1)}(\locdesuavo), \loc_2[\numwp_{0}-1]\right], \ldots,            \\
                                                                  & \left[\liftfun^{(1)}(\locdesuavo), \loc_2[\numwp_{0}-1]\right], \left[\locdesuavo, \loc_2[\numwp_{0}-1] \right] \Big\}.
                    \end{align}
                \end{bullets}
                \blt[show feasible]
                \begin{bullets}
                    \blt[c1 rate bs \ra uav1] C1: Proving C1 amounts to showing that $\capacity(\locbs,   \liftfun^{(\indlift)}(\locdesuavo)) \geq 2\minuavrate$ for $\indlift=0,\ldots,\numlift_1^{\downarrow}$. It is easy to see that
                    \begin{align}
                        \nonumber
                        \capacity(\locbs,   \liftfun^{(\indlift)}(\locdesuavo)) & \geq \min\left[
                            \capacity(\locbs,   \liftfun^{(0)}(\locdesuavo))    ,
                            \capacity(\locbs,   \liftfun^{(\numlift_1^{\downarrow})}(\locdesuavo))
                        \right]                                                                   \\
                                                                                & =
                        \min \left[
                            \capacity(\locbs,   \locdesuavo)    ,
                            \capacity(\locbs,   \locdesuavotop)
                            \right].
                    \end{align}
                    From \eqref{eq:capontopdesdf}, it follows that $\capacity(\locbs,   \locdesuavo) \geq 2\minuavrate+\minuerate > 2\minuavrate$. On the other hand, from \eqref{eq:rateabovedesuav1}, it follows that $\capacity(\locbs,   \locdesuavotop) \geq 2\minuavrate$. Hence, C1 is proven for $\trajdouble_1^{\downarrow}$ and $\trajdouble_2^{\downarrow}$.

                    \blt[c2 rate uav1 \ra uav2]
                    \begin{bullets}
                        \blt[rates while uavs at same heights] C2: One has, $\forall\indlift=0,\ldots,\numliftdown$, $\|\liftfun^{(\numlift_1^{\downarrow}-\indlift)}(\locdesuavo)-\liftfun^{(\numliftdown-\indlift)}$ $(\loc_2[\numwp_0-1])\| \leq \|\locdesuavo-\loc_2[\numwp_0-1]\|$, then,
                        \begin{align}
                            \nonumber
                             & \capacity(\liftfun^{(\numlift_1^{\downarrow}-\indlift)}(\locdesuavo),\liftfun^{(\numliftdown-\indlift)}(\loc_2[\numwp_0-1])) \geq \capacity(\locdesuavo,\loc_2[\numwp_0-1]) \\
                             &
                            \overset{\eqref{eq:capontopdesdf}}{\geq}
                            \minuavrate + \minuerate > \minuavrate,\label{eq:ratebtwuavsgoup}
                        \end{align}
                        which proves C2 for $\trajdouble_1^{\downarrow}$.

                        \blt[rates uav1 goes down] In $\trajdouble_2^{\downarrow}$,  UAV-1 descends from $\liftfun^{(\numlift_1^{\downarrow}-\numliftdown)}(\locdesuavo)$ to $\locdesuavo$ while UAV-2 stays at $\loc_2[\numwp_0-1]$. Start by noting that $\liftfun^{(\indlift)}(\locdesuavo)$ is between $\locdesuavo$ and $\liftfun^{(\numlift_1^{\downarrow}-\numliftdown)}(\locdesuavo)$, $\forall \indlift=0,\ldots,\numlift^{\downarrow}_{1}-\numliftdown$, which means that
                        \begin{align}
                            \|\liftfun^{(\indlift)}(\locdesuavo)-\liftfun^{(\numlift_1^{\downarrow}-\numliftdown)}(\locdesuavo)\|^2 & \leq \|\locdesuavo-\liftfun^{(\numlift_1^{\downarrow}-\numliftdown)}(\locdesuavo)\|^2.\label{eq:linedownuav1}
                        \end{align}
                        From Pythagoras' theorem,
                        \begin{align}
                            \nonumber
                             & \|\liftfun^{(\indlift)}(\locdesuavo) - \loc_2[\numwp_0-1]\|^2 =  \|\liftfun^{(\indlift)}(\locdesuavo) -\liftfun^{(\numlift_1^{\downarrow}-\numliftdown)}(\locdesuavo)\|^2                                             \\
                            \nonumber
                             & ~~~+ \|\liftfun^{(\numlift_1^{\downarrow}-\numliftdown)}(\locdesuavo) - \loc_2[\numwp_0-1] \|^2                                                                                                                       \\
                            \nonumber
                             & \overset{\eqref{eq:linedownuav1}}{\leq} \|\locdesuavo -\liftfun^{(\numlift_1^{\downarrow}-\numliftdown)}(\locdesuavo)\|^2 + \|\liftfun^{(\numlift_1^{\downarrow}-\numliftdown)}(\locdesuavo) - \loc_2[\numwp_0-1]\|^2 \\
                             & = \|\locdesuavo - \loc_2[\numwp_0-1]\|^2.
                        \end{align}
                        It follows that $\capacity(\liftfun^{(\indlift)}(\locdesuavo),\loc_2[\numwp_0-1])\geq \capacity(\locdesuavo,\loc_2[\numwp_0-1]) \geq\minuavrate + \minuerate > \minuavrate,$
                        where the second inequality follows from \eqref{eq:capontopdesdf}. This proves C2 for $\trajdouble_2^{\downarrow}$.
                    \end{bullets}
                \end{bullets}
            \end{bullets}

            \blt[des uav2 lower] The case when $\numliftdown>\numlift^{\downarrow}_1$, i.e., $\locdesuavo$ has a higher altitude than $\loc_2[\numwp_0-1]$, can be proven similarly.

        \end{bullets}
    \end{bullets}
\end{extendedonly}

\begin{nonextendedonly}
    \cmt{intro} The following is a sketch of the proof of \thref{thm:guarantee}. The complete proof can be found in the extended version of this paper \cite{viet2026planningarxiv}.
    \begin{bullets}
        \blt[given path] Let $\pathsingle_2\define\{\loc_2[0],\loc_2[1],\ldots,\loc_2[\numwp_0-1]\}$ be the path for UAV-2 returned by Algorithm~\ref{algo:trajtwo}
        \begin{bullets}%
            \blt[max num lift up] and $\numlift\define \min\{\indlift\in\mathbb{N}:\liftfun^{(\indlift)}(\pathsingle_2)=\liftfun^{(\indlift+1)}(\pathsingle_2)\}$.
        \end{bullets}%
        \blt[UAV-2 path lifted to top] If Algorithm~\ref{algo:trajone} fails to provide a valid path, it necessarily fails to find a valid path at all iterations and, in particular, at the $\numlift$-th iteration. Therefore, to prove the theorem, it suffices to show that the algorithm succeeds if it reaches the $\numlift$-th iteration, which is the worst case.
        Equivalently, it has to be shown that Algorithm~\ref{algo:trajone} can find a path to $\mathbb{N} \times \destsetuo$  through the extended graph corresponding to $\liftfun^{(\numlift)}(\pathsingle_2)$. 
        \blt[combined path feasible] Since Algorithm~\ref{algo:trajone} is based on a shortest path algorithm,
        it will return a path to $\mathbb{N} \times \destsetuo$ if at least one such a path exists. 
        Therefore, to prove the theorem, it suffices to find \textit{any} path to $\mathbb{N} \times \destsetuo$. 
                \blt[proof structure]To establish the existence of such a path in the extended graph, the  proof designs a path for UAV-1 and shows that the resulting  combined path is valid.

    \end{bullets}
\end{nonextendedonly}

\begin{extendedonly}
    \section{Proof of Lemma \ref{lem:capac}}
    \label{sec:proof:capac}
    \begin{bullets}
        \blt[change variables] Let $\rateoverbw\define\minuavrate/\bandwidth>0$ and $\rateueoveruav\define\minuerate/\minuavrate$. Then $
            (2\minuavrate+\minuerate)/\bandwidth=(2 + \rateueoveruav)\minuavrate/\bandwidth=(\rateueoveruav+2)\rateoverbw$ and                $
            (\minuavrate+\minuerate)/\bandwidth = (1+\rateueoveruav)\minuavrate/\bandwidth=(\rateueoveruav+1)\rateoverbw$.
        \blt[change lhs] From \eqref{eq:tmia},
        \begin{changesv2}
            $\forall\rate>0$
        \end{changesv2}
        , it follows that
        \begin{align}
            \capacitydist^{-1}(\rate) & \define \left(\frac{\rxpowerfull}{2^{\rate/\bandwidth}-1}\right)^{1/\pathlossexp},\label{eq:capacityinverse}
        \end{align}
        where $\rxpowerfull\define\txpower\txgain\rxgain\wavelength^{\pathlossexp}/[\noisepower(4\pi)^{\pathlossexp}]$. As a result,
        \begin{align}
              & \capacitydist^{-1}(2\minuavrate+\minuerate) + \capacitydist^{-1}(\minuavrate+\minuerate)                                          \nonumber                                      \\
            \begin{intermed}
                =
            \end{intermed}
            \alignchar
            \begin{intermed}
                \left(\frac{\rxpowerfull}{2^{(2\minuavrate+\minuerate)/\bandwidth}-1}\right)^{1/2} + \left(\frac{\rxpowerfull}{2^{(\minuavrate+\minuerate)/\bandwidth}-1}\right)^{1/2}
            \end{intermed}
            \jumpline
            = & \left(\frac{\rxpowerfull}{2^{(\rateueoveruav+2)\rateoverbw}-1}\right)^{1/\pathlossexp} + \left(\frac{\rxpowerfull}{2^{(\rateueoveruav+1)\rateoverbw}-1}\right)^{1/\pathlossexp}.
            \label{eq:intransform}
        \end{align}
        \blt[1st transform] Since $\rateoverbw>0$, one has that $ (\rateueoveruav+2)\rateoverbw                                                           >(\rateueoveruav+1)\rateoverbw$, which in turn implies that
        \begin{subequations}
            \begin{align}
                \left(\frac{\rxpowerfull}{2^{(\rateueoveruav+2)\rateoverbw}-1}\right)^{1/\pathlossexp} & < \left(\frac{\rxpowerfull}{2^{\left(\rateueoveruav+1\right)\rateoverbw} - 1}\right)^{1/\pathlossexp}.
            \end{align}
        \end{subequations}
        It follows that
        \begin{align}
            \begin{aligned}
                 & \capacitydist^{-1}(2\minuavrate+\minuerate) + \capacitydist^{-1}(\minuavrate+\minuerate) \\& < 2\left(\frac{\rxpowerfull}{2^{\left(\rateueoveruav+1\right)\rateoverbw} - 1}\right)^{1/\pathlossexp}.
                \label{eq:ineqinverse}
            \end{aligned}
        \end{align}
        \blt[update] By hypothesis,
        \begin{salign}
            \minuerate&\geq\bandwidth\log_2\left(1+2^{\pathlossexp}\snrmincc\right)\\
            \nonumber
            &\hspace{2cm}-\bandwidth\log_2\left(1+\snrmincc\right)\\
            &=\bandwidth\log_2\left(2^{\pathlossexp}\left(2^{\minuavrate/\bandwidth}-1\right)+1\right)-\minuavrate\\
        \end{salign}
        Then,
        \begin{salign}
            \rateueoveruav=\frac{\minuerate}{\minuavrate} &\geq\frac{\bandwidth}{\minuavrate}\log_2\left(2^{\pathlossexp}\left(2^{\minuavrate/\bandwidth}-1\right)+1\right)-1\\
            &=\frac{1}{\rateoverbw}\log_2\left(2^{\pathlossexp}\left(2^{\rateoverbw}-1\right)+1\right)-1.
        \end{salign}
        Simple algebraic manipulations yield
        \begin{align}
            %
            %
            %
            \frac{2^{\pathlossexp}}{2^{\rateoverbw\left(\rateueoveruav+1\right)}-1} & \leq\frac{1}{2^{\rateoverbw}-1},
        \end{align}
        which implies that
        \begin{align}
            %
            2\left(\frac{\rxpowerfull}{2^{\left(\rateueoveruav+1\right)\rateoverbw}-1}\right)^{1/\pathlossexp} & \leq\left(\frac{\rxpowerfull}{2^{\rateoverbw}-1}\right)^{1/\pathlossexp}=\capacitydist^{-1}(\minuavrate).
            \label{eq:ineqabbrev}
        \end{align}
        %
        %
        %
        \blt[conclusion]
        Combining \eqref{eq:ineqinverse} and \eqref{eq:ineqabbrev} yields \eqref{eq:capinvsumbound}.
    \end{bullets}
\end{extendedonly}

\section{Proof of Theorem~\ref{prop:feasible-movingue}}
\label{sec:proof:feasible-movingue}
\begin{extendedonly}
    This proof follows a similar logic to the one in the proof of \thref{thm:guarantee}.
    \begin{bullets}%
        \blt[lift to top]Algorithm~\ref{algo:pathone-movingue}  fails iff there is no path for UAV-1 that results in a  feasible combined path at all iterations, in particular at the $\numlift$-th iteration, where $\numlift \define \min \{\indlift:\lifttrimfun^{(\indlift)}(\pathsingle_2)=\lifttrimfun^{(\indlift+1)}(\pathsingle_2)\}$.
        \blt[what to show] Therefore, it suffices to show that there exists a path for UAV-1 (not necessarily the one produced by Algorithm~\ref{algo:pathone-movingue}) that results in a feasible combined path when UAV-2 follows $\lifttrimfun^{(\numlift)}(\pathsingle_2)$.
        \blt[proof structure] To this end, a path will be designed for UAV-1 and the combined path $\pathuavs\define\{\exconfpt[0],\ldots,\exconfpt[\numwpue-1]\}$, where $\exconfpt[\indwp]\define[\exloc_1[\indwp], \exloc_2[\indwp]]$ and $\left\{
            \exloc_2[0],\ldots,\exloc_2[\numwpue-1]
            \right\}=\lifttrimfun^{(\numlift)}(\pathsingle_2)$
        will be shown to be feasible. This means that the following conditions hold:
        \begin{itemize}
            \item[C1:] The rate between the BS and UAV-1 is at least $2\minuavrate$, i.e., $\capacity(\locbs,\exloc_1[\indwp])\geq 2\minuavrate$ for all $\indwp$, and
            \item[C2:] The rate from UAV-1 to UAV-2 is at least $\minuavrate$, i.e., $\capacity(\exloc_1[\indwp],$ $\exloc_2[\indwp])\geq \minuavrate$ for all $\indwp$.
        \end{itemize}%
        \blt[path separation] To simplify the exposition, the path $\lifttrimfun^{(\numlift)}(\pathsingle_2)$ will be separated into the following subpaths:
        \begin{align}
            \trajuavsecup   \define \{  & \loc_2[0],\liftfun^{(1)}(\loc_2[0]),\ldots,\liftfun^{(\maxnumliftup)}(\loc_2[0])\},              \\
            \nonumber
            \trajuavsectop   \define \{ & \liftfun^{(\maxnumliftup)}(\loc_2[0]),\loctop_2[1],\ldots,\loctop_2[\numwpue-\maxnumliftup-1]\},
        \end{align}
        where $\pathsingle_2\define\{\loc_2[0],\loc_2[1],\ldots,\loc_2[\numwpue-1]\}$.
        For each subpath, a path will be designed for UAV-1 and the resulting combined path will be shown to satisfy C1 and C2. The designed path is illustrated in Fig.~\ref{fig:enter-label}.

        \blt[part 1 - take off] \textbf{Take-off subpaths ($\trajuavsecup$)}
        \begin{bullets}
            \blt[choose path] In the considered combined path, when UAV-2 follows $\trajuavsecup$, UAV-1 follows $\trajuavsecup$ as well. Recall that the minimum separation between the UAVs was disregarded in \eqref{eq:trajproblem} for simplicity. Thus, the combined path of UAV-1 and UAV-2 is given by
            \begin{align}
                \nonumber
                \trajdouble^{\uparrow}\define\Big\{ & \Big[\loc_2[0],                           \loc_2[0]\Big],                             \left[\liftfun^{(1)}(\loc_2[0]),           \liftfun^{(1)}(\loc_2[0])\right],           \ldots , & \\
                                                    & \left[\liftfun^{(\maxnumliftup)}(\loc_2[0]),\liftfun^{(\maxnumliftup)}(\loc_2[0])\right]\Big\}.
            \end{align}
            \blt[showing feasible]
            \begin{bullets}
                \blt[c1 - rate btw bs-uav1] C1: By hypothesis, $\higherbheight$ $\leq\capacitydist\inv(2\minuavrate)$.
                Thus, since $\loc_2[0]=\locbs$, $\forall\exloc_{1}\in\trajuavsecup$, it follows that $\|\locbs-\exloc_1\|\leq\|\locbs-\liftfun^{(\maxnumliftup)}($ $\loc_2[0])\|\leq\higherbheight  \leq\capacitydist\inv(2\minuavrate)
                $.
                Noting that $\capacity$ is a decreasing function of the distance yields
                \begin{salign}
                    \capacity(\locbs,\exloc_1) & \geq\capacity\left(\locbs,\liftfun^{(\maxnumliftup)}(\loc_2[0])\right) \geq2\minuavrate,
                \end{salign}
                which establishes C1.

                \blt[c2 - rate uav1-uav2] C2: trivial.
            \end{bullets}%
        \end{bullets}

        \blt[part 2 - op top] \textbf{Top subpath ($\trajuavsectop$)}:
        \begin{bullets}
            \blt[choose path]In the considered combined path, when UAV-2 follows $\trajuavsectop$, UAV-1 stays at $\liftfun^{(\maxnumliftup)}(\loc_2[0])$. This results in the following combined path
            \begin{align}
                \nonumber
                \trajdouble^{\uparrow}\define\Big\{ & \Big[\liftfun^{(\maxnumliftup)}(\loc_2[0]), \liftfun^{(\maxnumliftup)}(\loc_2[0])\Big], \left[\liftfun^{(\maxnumliftup)}(\loc_2[0]),\loctop_2[1]\right], & \\
                                                    & \ldots,\left[\liftfun^{(\maxnumliftup)}(\loc_2[0]),\loctop_2[\numwpue-\maxnumliftup-1]\right]\Big\}.
            \end{align}
            \blt[showing feasible]
            \begin{bullets}
                \blt[c1 - rate btw bs-uav1] C1: shown previously.

                \blt[c2 - rate uav1-uav2] C2: By definition of $\numlift$, $[0,0,1]\exloc_2=\higherbheight,\forall\exloc_2\in\trajuavsectop$. Also, $\loc_2[0]=\locbs$. This means that $\forall\exloc_2\in\trajuavsectop,\|\liftfun^{(\maxnumliftup)}(\loc_2[0])-\exloc_2\|\leq\cylinderradiusmin\leq\capacitydist\inv(\minuavrate)$. Then $\capacity(\liftfun^{(\maxnumliftup)}(\loc_2[0]),\exloc_2)\geq\minuavrate$, $\forall\exloc_2\in\trajuavsectop$.
            \end{bullets}%
        \end{bullets}
    \end{bullets}
\end{extendedonly}

\begin{nonextendedonly}
    The following is a sketch of the proof of \thref{prop:feasible-movingue}. The complete proof can be found in the extended version of this paper \cite{viet2026planningarxiv}. This proof follows a similar logic to the one in the proof of \thref{thm:guarantee}.
    \begin{bullets}%
        \blt[lift to top]Algorithm~\ref{algo:pathone-movingue} fails iff there is no path for UAV-1 that results in a  feasible combined path at all iterations, in particular at the $\numlift$-th iteration, where $\numlift \define \min \{\indlift:\lifttrimfun^{(\indlift)}(\pathsingle_2)=\lifttrimfun^{(\indlift+1)}(\pathsingle_2)\}$.
        \blt[what to show] Therefore, it suffices to show that there exists a path for UAV-1 (not necessarily the one produced by Algorithm~\ref{algo:pathone-movingue}) that results in a feasible combined path when UAV-2 follows $\lifttrimfun^{(\numlift)}(\pathsingle_2)$. The proof constructs such a path for UAV-1 and shows that the combined path is feasible.
        %

        %
    \end{bullets}
\end{nonextendedonly}

\section{Benchmark 6}
\label{sec:segment-benchmark}

\begin{extendedonly}
    \begin{bullets}
        \blt[overview] In this benchmark, for every $\numlocstoreplan$ time steps, the planner is given the next $\numknownuelocs$ locations of the user,
        $\numlocstoreplan \leq \numknownuelocs$.
        \blt[algorithm] The following steps will be iteratively implemented at time steps $\indwp\numlocstoreplan,\indwp=0,1,\ldots$.
        \begin{enumerate}
            \item[S1:] The planner uses the algorithms in
                Sec.~\ref{sec:staticUE-tentative} to plan a path for the UAVs to
                the nearest grid points where they can serve
                \begin{itemize}
                    \item All of the next $\numknownuelocs$ locations of the user.
                    \item If such grid points do not exist, the planner plans a
                          path to the nearest grid points where the UAVs can
                          serve the last $(\numknownuelocs-1)$ known locations
                          of the user, i.e., $\locue[\indwp\numlocstoreplan +
                                  \indaux], \indaux = 1,..., \numknownuelocs - 1$, and
                          so on.
                    \item In the most extreme case when the planner cannot find
                          grid points to simultaneously guarantee $\minuerate$
                          to multiple locations of the user, the planner plans a
                          path to the nearest grid point where the UAVs can
                          serve the last known location of the user, i.e.,
                          $\locue[\indwp\numlocstoreplan +
                                  \numknownuelocs - 1]$.
                \end{itemize}
            \item[S2:] If the length of the path obtained in Step~1 is less than
                $\numlocstoreplan$, the last configuration point of the path is
                repeated until the length of the path is $\numlocstoreplan$. If
                the length of the path obtained in Step~1 is greater than
                $\numlocstoreplan$, only the first $\numlocstoreplan$
                configuration points of the path are kept.
            \item[S3:] The last configuration point of the path obtained in
                Step~2 provides the start locations of the UAVs in the next
                iteration, i.e., at time step $(\indwp + 1)\numlocstoreplan$.
        \end{enumerate}
    \end{bullets}
\end{extendedonly}

\begin{nonextendedonly}
    \cmt{intro} The following is a short description of Benchmark 6. The complete explanation can be found in the extended version of this paper \cite{viet2026planningarxiv}.
    \begin{bullets}%
        \blt[overview] This benchmark replans the path  every $\numlocstoreplan$ time steps. When replanning, the next $\numknownuelocs$ locations of the UE are given, where $\numlocstoreplan \leq \numknownuelocs$.
        \blt[algorithm]Algorithm~\ref{algo:trajtwo} is modified by replacing $\destsetut$ with the intersection of the  destination sets that correspond to the last $\indknownuelocs$ of the given $\numknownuelocs$  UE locations.
        At time step $\indwp\numlocstoreplan$, $\indwp=0,1,\ldots$,
        this modified algorithm is run for $\indknownuelocs=\numknownuelocs, \numknownuelocs-1, \ldots, 1$ until a path for UAV-2 is found.
        The path for UAV-1 is obtained using Algorithm~\ref{algo:trajone}.


    \end{bullets}%
\end{nonextendedonly}


\end{document}

\subsection{Static UE - Globecom}
\begin{bullets}
    \blt[tested algorithms]
    \begin{bullets}
        \blt[why benchmarks] Due to the presence of buildings, no algorithm in
        the literature that \acom{we are} aware of can directly accommodate the
        considered simulation setup. Instead, three benchmarks will be
        considered:
        \blt[benchmarks]
        \begin{bullets}%
            \blt[b1] Benchmark 1 is a single UAV that takes off at
            $\locbs$ vertically to a height of \arev{$\higherbheight$} and then moves
            horizontally in straight line to the middle point
            between the BS and UE, i.e., to the point
            \arev{$([\locxbs,\locybs, \higherbheight]\transpose+[\locxue,\locyue,\higherbheight]\transpose)/2$}.
            \blt[b2] In Benchmark 2, two UAVs lift off at the BS
            to a height of $\higherbheight$ above it. Then UAV-1 moves to
            $(2/3)[\locxbs,\locybs,\higherbheight]\transpose+(1/3)[\locxue,\locyue,\higherbheight]\transpose$
            whereas UAV-2 moves to
            $(1/3)[\locxbs,$ $\locybs,\higherbheight]\transpose + (2/3)[\locxue,$ $\locyue,\higherbheight]\transpose$.
            \blt[b3] Benchmark 3 is similar to Benchmark 2, but
            UAV-1 remains  at
            $[\locxbs,\locybs,\higherbheight]\transpose$ after lift off whereas UAV-2 flies
            to $[\locxue,\locyue,\higherbheight]\transpose$. Under the
            hypotheses of \thref{prop:feasibleexists}, the
            resulting path is always feasible and yields a finite
            cost.
        \end{bullets}%
        All three benchmarks stop at the point of their trajectories
        where the UE rate is maximum. \ra \acom{stop when the rate between UAV-1 and the BS is less than $2\minuavrate$.}
        \blt[proposed] The proposed PRFI algorithm is implemented with a
        \acom{$12\times 12 \times 8$} grid $\grid$. The number of CPs is \acom{500} and the number of nearest neighbors is limited to~\acom{50}.
    \end{bullets}

    \blt[description of the experiments] \arev{Building absorption is set to 1 dB/m.}
    \begin{bullets}
        \blt[mean ue rate]Fig.~\ref{fig:meanuerate} plots the MC estimate
        of the expectation $\expected[\uerate(\confpt(t))]$ vs. $t$ for
        all compared algorithms when $\minuerate$ = 100 Mbps.  Since all
        UAVs start from the BS, the initial rate is the same for all
        algorithms. This rate is not 0 because in some MC realizations the
        BS and the UE may already be close to each other and possibly in
        LOS.  Observe that the proposed algorithm attains the target rate
        $\minuerate$ in just a few seconds. The rate continues increasing
        beyond $\minuerate$ due to two effects: first, due to the spatial
        discretization, there are no grid points $\confpt$ that exactly
        result in $\uerate(\confpt)=\minuerate$. Since the algorithm needs
        to find destinations with $\uerate(\confpt)\geq\minuerate$, this
        will generally result in destinations that exceed this
        minimum. Second, in those MC realizations where UAV-2 needs to
        turn a corner to reach the UE, the rate increases suddenly. This
        can be seen by plotting $\uerate(\confpt(t))$ for individual MC
        realizations. Note also that, in accordance with
        \thref{prop:feasibleexists}, Benchmark 3 eventually attains
        $\minuerate$. Finally, it is also observed that the rate of the
        benchmarks initially decreases. This is due to the take-off in
        those realizations where there are good propagation conditions
        directly between the BS and UE locations.

        \blt[mean time to connect] The second experiment studies the
        influence of $\minuerate$ onto the expectation of the connection
        time, which is the cost in Problem~\eqref{eq:trajproblem}.  To
        this end, Fig.~\ref{fig:timetoconnect} plots
        $\expected[\ttc(\trajectory)|\ttc(\trajectory)<\infty]$ as well as
        the probability of failure vs. $\minuerate$. The notation
        $\expected[\ttc(\trajectory)|\ttc(\trajectory)<\infty]$ indicates
        that only those MC realizations where the $\minuerate$ is attained
        are considered in the MC average. The probability of failure is
        defined as the fraction of MC realizations for which
        $\uerate(\confpt(t))<\minuerate~\forall t$.

        The proposed algorithm showcases a 0 probability of failure until
        $\minuerate\approx 160$ Mbps and, throughout this interval, its
        expected connection time is much smaller than for the
        benchmarks. Benchmark 3 also provides a null probability of
        failure until $\minuerate$ is too large, in which case the
        hypotheses of \thref{prop:feasibleexists} no longer hold.

        \begin{figure}
            \centering
            \includegraphics[width=1\linewidth]{figs/fig1n-mean-time.pdf}
            \captionof{figure}{Mean connection time  and probability of failure vs. the UE rate requirement $\minuerate$ averaged across 50 MC realizations.}
            \label{fig:timetoconnect}
        \end{figure}

        \blt[mean ue rate] Finally, Fig.~\ref{fig:vsdistance} plots
        $\expected[\ttc(\trajectory)|\ttc(\trajectory)<\infty]$ and the
        probability of failure vs. the distance $\|\locue-\locbs\|$. 50 MC
        realizations are generated for each considered value of
        $\|\locue-\locbs\|$ by first drawing $\locbs$ uniformly at random
        out of the buildings and then drawing $\locue$ on a circle of the
        appropriate radius and centered at $\locbs$. Consistent
        with Fig.~\ref{fig:timetoconnect}, PRFI yields a zero probability
        of failure throughout. At this $\minuerate$, Benchmark 1 yields a zero connection time but this is because it only succeeds
        to deliver this rate in the MC realizations where the BS and UE
        have already good direct communication conditions.

        \begin{figure}
            \centering
            \includegraphics[width=1\linewidth]{figs/fig3n-distance.pdf}
            \captionof{figure}{Mean connection time and probability of failure vs. the distance $\|\locue-\locbs\|$ averaged across 50 MC realizations ($\minuerate$ = 160 Mbps).}
            \label{fig:vsdistance}
        \end{figure}
    \end{bullets}
\end{bullets}%

\section{Figure - for moving UE}
\begin{bullets}

    \blt[exp7027]
    \begin{figure}
        \centering
        \begin{subfigure}[b]{0.49\textwidth}
            \centering
            \includegraphics[width=\textwidth]{figs/moving-ue/experiment_7027-0.pdf}
            \caption{}
        \end{subfigure}
        \hfill
        \begin{subfigure}[b]{0.49\textwidth}
            \centering
            \includegraphics[width=\textwidth]{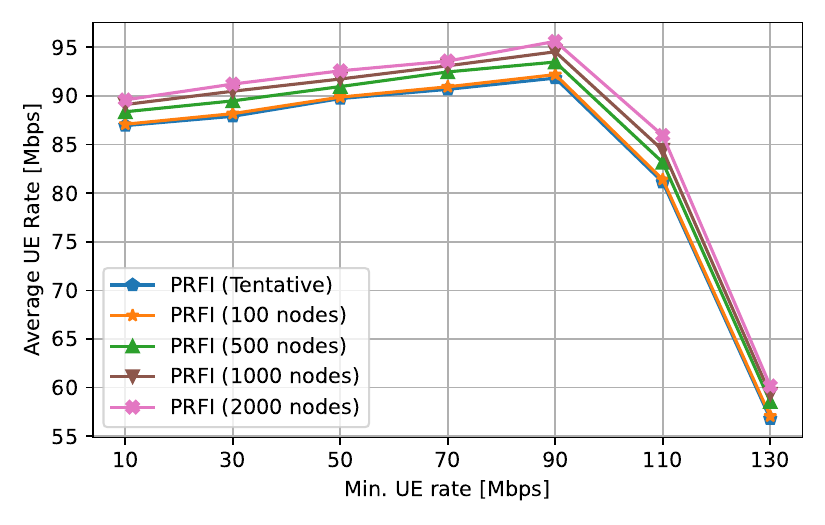}
            \caption{\acom{this}}
        \end{subfigure}
        \hfill
        \begin{subfigure}[b]{0.49\textwidth}
            \centering
            \includegraphics[width=\textwidth]{figs/moving-ue/experiment_7027-2.pdf}
            \caption{\acom{this}}
        \end{subfigure}
        \hfill
        \begin{subfigure}[b]{0.49\textwidth}
            \centering
            \includegraphics[width=\textwidth]{figs/moving-ue/experiment_7027-3.pdf}
            \caption{}
        \end{subfigure}
        \hfill
        \begin{subfigure}[b]{0.49\textwidth}
            \centering
            \includegraphics[width=\textwidth]{figs/moving-ue/experiment_7027-7.pdf}
            \caption{$\minuerate=90$ Mbps.}
        \end{subfigure}
        \hfill
        \begin{subfigure}[b]{0.49\textwidth}
            \centering
            \includegraphics[width=\textwidth]{figs/moving-ue/experiment_7027-13.pdf}
            \caption{$\minuerate=90$ Mbps.}
        \end{subfigure}
        \caption{Performance metrics vs. $\minuerate$ across 200 MC realizations for the proposed PRFI only.}
    \end{figure}
    \begin{figure}
        \centering
        \begin{subfigure}[b]{0.49\textwidth}
            \centering
            \includegraphics[width=\textwidth]{figs/moving-ue/experiment_7027-8.pdf}
            \caption{$\minuerate=110$ Mbps.}
        \end{subfigure}
        \hfill
        \begin{subfigure}[b]{0.49\textwidth}
            \centering
            \includegraphics[width=\textwidth]{figs/moving-ue/experiment_7027-14.pdf}
            \caption{$\minuerate=110$ Mbps.}
        \end{subfigure}
        \hfill
        \begin{subfigure}[b]{0.49\textwidth}
            \centering
            \includegraphics[width=\textwidth]{figs/moving-ue/experiment_7027-9.pdf}
            \caption{$\minuerate=130$ Mbps.}
        \end{subfigure}
        \hfill
        \begin{subfigure}[b]{0.49\textwidth}
            \centering
            \includegraphics[width=\textwidth]{figs/moving-ue/experiment_7027-15.pdf}
            \caption{$\minuerate=130$ Mbps.}
        \end{subfigure}
        \caption{Mean UE rate and total transferred data vs. time averaged across 200 MC realizations for the proposed PRFI only.}
    \end{figure}

\end{bullets}

\section{TO DO}

\begin{itemize}
    \item simulations

          \begin{itemize}
              \item Static user
                    \begin{itemize}
                        \item Repeat simulations with the same environment as for static user
                    \end{itemize}

              \item Moving user
                    \begin{itemize}
                        \item repeat simulation with building heights
                        \item merge main into relays\_move
                        \item submit PR
                        \item merge PR
                        \item extend code to allow two ways of lifting, one with shortest path recomputation and the other with plain lifting
                        \item submit PR
                    \end{itemize}

          \end{itemize}

    \item writing
          \begin{itemize}
              \item Move the figures that are not selected to an appendix \arev{\ra done}
              \item Unify figure appearance \arev{can be done by re-run the simulations with static ue}
              \item write the simulation section
              \item make sure that we use the terms trajectory and path appropriately
          \end{itemize}
\end{itemize}

\end{document}

%% file: include_v6.tex
\usepackage{savesym} 
\savesymbol{labelindent}
\usepackage{enumitem} 
\newcommand{\acom}[1]{\textcolor{red}{{[#1]}}} 
\newcommand{\arev}[1]{\textcolor{blue}{#1}} 

\if\editmode1  
\usepackage{cite}
\DeclareFieldFormat{labelalpha}{\thefield{entrykey}}
\DeclareFieldFormat{extraalpha}{}
\newcommand{\cmt}[1]{\noindent\textcolor{darkgreen}{\underline{[#1]}}} 
\newcommand{\hc}[1]{\textcolor{blue}{#1}} 

\setlistdepth{9}
\newlist{bulletlist}{enumerate}{9}
\setlist[bulletlist,1]{label=$\bullet$}
\setlist[bulletlist,2]{label=$\diamond$}
\setlist[bulletlist,3]{label=$\rightarrow$}
\setlist[bulletlist,4]{label=$\circ$}
\setlist[bulletlist,5]{label=$\square$}
\setlist[bulletlist,6]{label=$\star$}
\setlist[bulletlist,7]{label=$\checkmark$}
\setlist[bulletlist,8]{label=$\Delta$}
\setlist[bulletlist,9]{label=$-$}
\newenvironment{bullets}{\begin{bulletlist}}{\end{bulletlist}}

\newcommand{\blt}[1][noargpassed]{
  \item%
  \ifthenelse{\equal{#1}{noargpassed}}{}{\cmt{#1}}%
}

\else
\usepackage{cite}
\newcommand{\cmt}[1]{} 
\newcommand{\hc}[1]{\textcolor{black}{#1}} 
\newenvironment{bullets}{}{}
\newcommand{\blt}[1][noargpassed]{\ignorespaces}

\fi

\newcommand{\printmybibliography}{
\bibliography{\bibfilenames}
}

\usepackage{fixltx2e}

\usepackage{graphicx}

\usepackage{subcaption}              

\usepackage[utf8]{inputenc}

\usepackage{amsfonts}
\usepackage{amsmath}
\usepackage{mathtools}

\usepackage{amssymb}

\usepackage{bm}

\usepackage{color,verbatim}
\usepackage{multirow}
\usepackage{accents}

\usepackage{theoremref}

\usepackage{url}

\makeatletter
\newtoks\@soltoks
\newcommand\printquizsolutions{%
  \the\@soltoks}
\newcommand\quizsolution[1]{%
  \begin{flushright}
    See the solution \hyperref[{qxsol:\themyquiz}]{here}.
  \end{flushright}
  \edef\@soltemp{%
    \the\@soltoks 
    \noexpand\solutionpreamble{\themyquiz}
    \unexpanded{
      #1
    }%
  }
  \global\@soltoks=\expandafter{\@soltemp}
}
\newcommand{\solutionpreamble}[1]{%
  \par\medskip\noindent\section{Solution to Quiz~\ref{qx:#1}}\label{qxsol:#1}\enspace\ignorespaces
}
\makeatother

\newcounter{rulecounter}
\newcommand{\resetrule}{ \setcounter{rulecounter}{0}}
\resetrule

\newtheorem{myauxproblem}{Problem}

\newtheorem{myauxoptionalproblem}{Optional Problem}

\newenvironment{salign}[1][0]{\subequations \if#10  \else \label{#1} \fi \align}{\endalign\endsubequations}
\newsavebox{\selvestebox}

\newenvironment{colbox}[1]
  {\newcommand\colboxcolor{#1}%
   \begin{lrbox}{\selvestebox}%
   \begin{minipage}{\linewidth}}
  {\end{minipage}\end{lrbox}%
   \begin{center}
   \colorbox{\colboxcolor}{\usebox{\selvestebox}}
   \end{center}}

\definecolor{orange}{rgb}{1,0.8,0}
\definecolor{gray}{rgb}{.9,0.9,0.9}
\definecolor{darkgray}{rgb}{.3,0.3,0.3}
\definecolor{darkblue}{rgb}{.1,0.0,0.3}
\definecolor{lightblue}{rgb}{0.7,0.7,1}
\definecolor{lightred}{rgb}{1,0.7,.7}
\definecolor{purple}{RGB}{204,153,255}
\definecolor{lightgray}{rgb}{.95,0.95,0.95}
\definecolor{lightgreen}{rgb}{0.6,0.8,0.6}
\definecolor{darkgreen}{rgb}{0.05,0.3,0.05}
\definecolor{pistachio}{RGB}{204,255,153}
\definecolor{paleturquoise}{RGB}{175,238,238}
\definecolor{yellow}{RGB}{255,255,153}

\newcommand{\ra}{$\rightarrow$~}

\newcommand{\tbm}[1]{{\tilde{\bm #1}}}

\newcommand{\inv}{^{-1}}

\newcommand{\rfield}{\mathbb{R}}

\newcommand{\transpose}{^\top}
 \newcommand{\define}{:=}

\newcommand{\prob}{\mathop{\mathbb{P}} }

\newcommand{\expected}{\mathop{\textrm{E}}\nolimits}

\newcommand{\minimize}{\mathop{\text{minimize}}}

\newcommand{\st}{\mathop{\text{s.t.}}}

\newtheorem{myproposition}{Proposition}
\newtheorem{myquestion}{Question}
\newtheorem{myquiz}{Quiz}
\newtheorem{myremark}{Remark}
\newtheorem{myproblemstatement}{Problem}
\newtheorem{mylemma}{Lemma}
\newtheorem{mytheorem}{Theorem}
\newtheorem{mydefinition}{Definition}
\newtheorem{mycorollary}{Corollary}
\newtheorem{myexample}{Example}



%% file: notation.tex
\newcommand{\hcc}[1]{\textcolor{red}{#1}}

\renewcommand{\expected}{\mathbb{E}}
\renewcommand{\define}{\triangleq} 

\newcommand{\induav}{{\hc{k}}}
\newcommand{\numuav}{{\hc{K}}}
\newcommand{\dist}{{\hc{d}}} 
\newcommand{\distbsue}{{\dist_{\text{BS}}^{\text{UE}}}} 
\newcommand{\distbsuemin}{\check {\dist}_{\text{BS}}^{\text{UE}}}
\newcommand{\distbsuemax}{\hat {\dist}_{\text{BS}}^{\text{UE}}}
\newcommand{\distmax}{{\hc{d}_\text{max}}}
\newcommand{\destset}{{\hc{\mathcal{D}}}} 
\newcommand{\destsetpre}{{\hc{\tilde{\mathcal{D}}}}} 
\newcommand{\destsetuo}{{\hc{\mathcal{D}_1}}} 
\newcommand{\destsetut}{{\hc{\mathcal{D}_2}}} 

\newcommand{\funobj}{{\hc{J}}}

\newcommand{\indwp}{{\hc{n}}} 
\newcommand{\indfirstrepeatedwp}{\bar{\indwp}} 
\newcommand{\numwp}{{\hc{N}}} 
\newcommand{\numwpten}{{\numwp}} 
\newcommand{\indwpwithwaits}{{\hc{\tilde n}}} 
\newcommand{\numwpwithwaits}{{\hc{\tilde N}}} 
\newcommand{\numwptrim}{{\hc{\check N}}} 
\newcommand{\prnumwp}{{\hc{ N}^\text{PR}}} 
\newcommand{\numwpue}{\numwp_{\text{UE}}} 

\newcommand{\numwptakeoff}{\numwp_{\text{tko}}} 

\newcommand{\indaux}{{\hc{i}}} 
\newcommand{\indlift}{{\hc{u}}} 
\newcommand{\numlift}{{\hc{U}}} 
\newcommand{\numlifttotop}{\bar{\numlift}} 
\newcommand{\numliftup}{{\hc{U^{\uparrow}}}} 
\newcommand{\maxnumliftup}{\indlift^\uparrow_\text{max}} 
\newcommand{\numliftdown}{{\hc{U^{\downarrow}}}} 

\newcommand{\weightfun}{{\hc{w}}} %
\newcommand{\weightpenalty}{{\weightfun_\text{p}}} %
\newcommand{\liftfun}{{\hc{L}}} %
\newcommand{\trimfun}{{\hc{Z}}} %
\newcommand{\addfun}{{\hc{A}}} %
\newcommand{\lifttrimfun}{\bar{\liftfun}} %

\newcommand{\ttc}{{\hc{T}_\text{c}}} 
\newcommand{\trajectory}{{\hc{\bm Q(\cdot)}}} %
\newcommand{\opttrajectory}{{\hc{\bm Q^*(\cdot)}}} %
\newcommand{\trajectoryue}{\pathsingle_\text{UE}} %
\newcommand{\pathsingle}{\hc{p}} %
\newcommand{\trajuavsecup}{\pathsingle^{\uparrow}} %
\newcommand{\trajuavsectop}{\overset{\rightarrow}{\pathsingle}} %
\newcommand{\trajuavsecdown}{\pathsingle^{\downarrow}} %

\newcommand{\pathuavs}{\trajectorywps}
\newcommand{\trajdouble}{\trajectorywps} %

\newcommand{\trajectorywps}{\hc{{P}}} 
\newcommand{\trajectorywpsvalid}{{\trajectorywps^{\text{V}}}} 
\newcommand{\trajectorywpsfeas}{{\trajectorywps^{\text{F}}}} 
\newcommand{\trajectorywpspr}{{\trajectorywps^{\text{PR}}}} 
\newcommand{\trajminoutage}{{\trajectorywps_{\text{min}}}} 

\newcommand{\maxbheight}{{\hc{h}}} 
\newcommand{\heighttop}{{\hc{h}}_{\text{top}}} %
\newcommand{\higherbheight}{{\hc{h}}_{\text{max}}} 
\newcommand{\coorhorizon}{{\hc{x}}}
\newcommand{\coorvertical}{{\hc{y}}}
\newcommand{\zmax}{{\hc{z}_\text{max}}}

\newcommand{\indicatorlos}{\hc{\mathcal{L}}}
\newcommand{\indicatoradjacent}{\hc{\mathcal{A}}}
\newcommand{\indicatorconnect}{\hc{\mathcal{C}}}
\newcommand{\funindicator}{\hc{\mathcal{I}}}

\newcommand{\indsecond}{\hc{j}}
\newcommand{\indthird}{\hc{p}}
\newcommand{\indfourth}{\hc{q}}

\newcommand{\logicaland}{\hc{\wedge}}

\newcommand{\region}{\hc{\mathcal{S}}}
\newcommand{\flyregion}{\hc{\bar{\mathcal{F}}}}
\newcommand{\oobregion}{\hc{\mathcal{F}}}
\newcommand{\grid}{{\hc{\bar{\mathcal{F}}}_\text{G}}}
\newcommand{\adjset}{{\hc{\mathcal{E}_{\grid}}}} 
\newcommand{\confspacegrid}{{\hc{\mathcal{Q}}_\text{G}}}
\newcommand{\gridpt}{\hc{\bm x}^{\grid}}
\newcommand{\gridptlosBS}{\hc{\bm x}^{\setlosbs}}
\newcommand{\gridptloslosBS}{\hc{\bm x}^{\setloslosbs}}
\newcommand{\gridptind}{\hc{g}}
\newcommand{\gridptnum}{\hc{G}}
\newcommand{\gridnumlosbs}{\hc{G}_\text{BS}}
\newcommand{\gridnumloslosbs}{\tilde{\hc{G}}_\text{BS}}
\newcommand{\gridmaxliftlevel}{{\hc{\bar{\mathcal{F}}}_\text{G}^\text{max}}}

\newcommand{\numgridpt}{\hc{N}_{\text{grid}}}
\newcommand{\numgridptx}{\hc{N}^{\text{grid}}_{\text{x}}}
\newcommand{\numgridpty}{\hc{N}^{\text{grid}}_{\text{y}}}
\newcommand{\numgridptz}{\hc{N}^{\text{grid}}_{\text{z}}}

\newcommand{\numres}{\hc{N}_{\text{R}}} 

\newcommand{\trajectorypt}{\hc{\bm x}}

\newcommand{\capacity}{\hc{c}}
\newcommand{\capacitysnr}{\hc{c}_\text{SNR}}
\newcommand{\capacitydist}{{\bar{c}}}
\newcommand{\txpower}{\hc{P_\text{t}}}
\newcommand{\rxpowerfull}{\hc{A}}
\newcommand{\txgain}{\hc{G_\text{t}}}
\newcommand{\rxgain}{\hc{G_\text{r}}}

\newcommand{\patlosexp}{\hc{\eta}}
\newcommand{\horizon}{\hc{T}}

\newcommand{\rate}{\hc{r}}
\newcommand{\rateoverbw}{\hc{a}}
\newcommand{\rateueoveruav}{\hc{\ell}}
\newcommand{\rateset}{\hc{\mathcal{R}}}
\newcommand{\rateuavone}{\hc{r}_{\text{1}}}
\newcommand{\rateuavtwo}{\hc{r}_{\text{2}}}
\newcommand{\uerate}{\hc{r}_{\text{UE}}}
\newcommand{\ueratey}{\hc{r}_{\text{UE}}^{(\text{y})}}
\newcommand{\ueratex}{\hc{r}_{\text{UE}}^{(\text{x})}}

\newcommand{\uerateapprox}{\hc{\tilde{r}}_{\text{UE}}^{(\text{x})}}
\newcommand{\minuerate}{\hc{r}_{\text{UE}}^{\text{min}}}
\newcommand{\minuavrate}{\hc{r}_{\text{CC}}}

\newcommand{\loc}{\hc{\bm q}}
\newcommand{\locaux}{\hc{\bm \zeta}}
\newcommand{\locrnd}{\hc{\tbm q}}
\newcommand{\locx}{\hc{x}}
\newcommand{\locy}{\hc{y}}
\newcommand{\locz}{\hc{z}}
\newcommand{\locbs}{\hc{\bm q}_{\text{BS}}}
\newcommand{\locxbs}{\hc{x}_{\text{BS}}}
\newcommand{\locybs}{\hc{y}_{\text{BS}}}
\newcommand{\loczbs}{\hc{z}_{\text{BS}}}
\newcommand{\locue}{\hc{\bm q}_{\text{UE}}}
\newcommand{\locxue}{\hc{x}_{\text{UE}}}
\newcommand{\locyue}{\hc{y}_{\text{UE}}}
\newcommand{\loczue}{\hc{z}_{\text{UE}}}
\newcommand{\loctop}{\bar{\loc}}
\newcommand{\locdesuavo}{\hc{\loc}_1^*}
\newcommand{\locdesuavotop}{\hc{\bar{\loc}_1^*}}

\newcommand{\operatormin}{\hc{\text{min}}}

\newcommand{\setlosbs}{{\hc{\mathcal{F}}}_{\text{BS}}}
\newcommand{\setloslosbs}{\tilde{\hc{\mathcal{F}}}_{\text{BS}}}
\newcommand{\setdeslosbs}{{\hc{\mathcal{D}}}_{\text{BS}}}
\newcommand{\setdesloslosbs}{\tilde{\hc{\mathcal{D}}}_{\text{BS}}}
\newcommand{\setconfptstenta}{{\hc{\mathcal{Q}}}^{\text{ten}}}

\newcommand{\bandwidth}{\hc{B}}

\newcommand{\confpt}{\hc{\bm Q}}
\newcommand{\confptrnd}{\hc{\tbm Q}}

\newcommand{\exconfpt}{\hc{\acute{\bm Q}}}
\newcommand{\exloc}{\hc{\acute{\bm q}}}
\newcommand{\exnumwp}{\hc{\acute{\numwp}}}

\newcommand{\numconfpt}{\hc{ C}}
\newcommand{\numconfptperslot}{\hc{ \tilde C}}
\newcommand{\confptoptimal}{\hc{\bm Q}^*}
\newcommand{\confpttenta}{\tilde{\hc{\bm Q}}}
\newcommand{\confptsuavone}{\hc{\bm a}}
\newcommand{\confptsuavtwo}{\hc{\bm b}}

\newcommand{\maxnumneighbors}{\hc{N}_{\text{nb}}}

\newcommand{\maxuavspeed}{\hc{v}_{\text{max}}}

\newcommand{\sampint}{\hc{\tau}}

\newcommand{\costmatrixuav}{\hc{\tilde{\bm M}}_{\text{BS}}}
\newcommand{\costentryuav}{\hc{\tilde{m}}}
\newcommand{\costmatrixtenta}{\hc{\bm M}^{\text{ten}}}

\newcommand{\graph}{\hc{\mathcal{G}}}
\newcommand{\graphuavone}{\hc{\mathcal{G}_2}}
\newcommand{\graphuavtwo}{\hc{\mathcal{G}_1}}
\newcommand{\graphloslosbs}{\tilde{\hc{\mathcal{G}}}_{\text{BS}}}

\newcommand{\timetotal}{\hc{T}}
\newcommand{\timetotalue}{\timetotal_{\text{UE}}}
\newcommand{\indtimestep}{\hc{n}}
\newcommand{\numtimestepoptpresample}{\hc{N}_{\text{pre}}}
\newcommand{\numtimestepuavone}{\hc{T}_1}
\newcommand{\numtimestepuavtwo}{\hc{T}_2}
\newcommand{\numtimesteps}{\hc{N}_{\text{step}}}

\newcommand{\prnodeset}{\hc{\mathcal{N}}}
\newcommand{\predgeset}{\hc{\mathcal{E}}}
\newcommand{\nodeset}{\hc{\mathcal{N}}}
\newcommand{\nodesetpre}{\tilde{\nodeset}}
\newcommand{\nodesetut}{\hc{\mathcal{N}_2}} 

\newcommand{\enodeset}{\hc{\bar{\mathcal{N}}}}
\newcommand{\enodesetuo}{\enodeset_1}
\newcommand{\enodesetut}{\enodeset_2}
\newcommand{\edestsetut}{\hc{\bar{\mathcal{D}}_2}}
\newcommand{\nodesetuo}{\hc{\mathcal{N}_1}} 

\newcommand{\numptsuavone}{\hc{N}}
\newcommand{\numliftstart}{\hc{N}_\text{lift}^\text{start}}
\newcommand{\numliftend}{\hc{N}_\text{lift}^\text{end}}

\newcommand{\matrixallinfty}{\hc{\bm I}_{\infty}}

\newcommand{\spacingvec}{\hc{\bm\delta}}
\newcommand{\spacingscalar}{\hc{\delta}}
\newcommand{\spacingx}{\spacingscalar_\text{x}}
\newcommand{\spacingy}{\spacingscalar_\text{y}}
\newcommand{\spacingz}{\spacingscalar_\text{z}}

\newcommand{\indinteration}{\hc{k}}

\newcommand{\confspace}{\hc{\mathcal{Q}}}
\newcommand{\confspacefree}{\hc{\mathcal{Q}}^{\text{free}}}

\newcommand{\numlocstoreplan}{\numwp_\text{replan}}
\newcommand{\numknownuelocs}{\numwp_\text{known}}
\newcommand{\indknownuelocs}{\indwp_\text{known}}

\newcommand{\cylinder}{{\hc{\mathcal{C}}}} 
\newcommand{\cylinderradius}{\hc{{d}}_{\cylinder}} %
\newcommand{\cylinderradiusmin}{\hc{{d}_{\cylinder,\text{min}}}} %

\newcommand{\arealength}{\hc{L}}
\newcommand{\arealengthx}{\arealength_{x}}
\newcommand{\arealengthy}{\arealength_{y}} 
\newcommand{\arealengthz}{\arealength_{z}} 

\newcommand{\pathlossexp}{\hc{\beta}} 
\newcommand{\lequal}{\hc{\rateueoveruav}_{\text{eq}}} 

\newcommand{\pointA}{\hc{\pmb{a}}}
\newcommand{\pointB}{\hc{\pmb{b}}}
\newcommand{\pointC}{\hc{\pmb{c}}}
\newcommand{\pointD}{\hc{\pmb{d}}}
\newcommand{\pointE}{\hc{\pmb{e}}}
\newcommand{\pointF}{\hc{\pmb{f}}}
\newcommand{\pointG}{\hc{\pmb{g}}}
\newcommand{\pointH}{\hc{\pmb{h}}}
\newcommand{\pointM}{\hc{\pmb{m}}}
\newcommand{\pointN}{\hc{\pmb{n}}}
\newcommand{\pointX}{\hc{\pmb{x}}} 
\newcommand{\pointY}{\hc{\pmb{y}}(\timeInstant)}
\newcommand{\pointZ}{\hc{\pmb{z}}(\timeInstant)}
\newcommand{\timeInstant}{\hc{t}} 
\newcommand{\timemaxx}{\hc{t}_{\text{1}}}
\newcommand{\timemaxy}{\hc{t}_{\text{0}}}
\newcommand{\timeInstantAux}{\hc{\tau}} 
\newcommand{\locUavAux}{\pointX(\timeInstant)}
\newcommand{\buildingWidth}{\hc{W}}
\newcommand{\buildingHeight}{\hc{H}}
\newcommand{\absorption}{\hc{\alpha}}
\newcommand{\absorptionmax}{\hc{\bar{\alpha}}}
\newcommand{\horizonAngle}{\hc{\theta}(\timeInstant)}
\newcommand{\spaceToBuilding}{\hc{\epsilon}}
\newcommand{\distAuxOne}{\hc{d}_1(\timeInstant)}
\newcommand{\distAuxTwo}{\hc{d}_2(\timeInstant)}
\newcommand{\distAuxThree}{\hc{d}_3(\timeInstant)}
\newcommand{\gain}{\hc{\gamma}}
\newcommand{\gainpl}{\hc{\gamma}_\text{PL}}
\newcommand{\gainAux}{\hc{\bar{\gamma}}}
\newcommand{\gainAuxNat}{\hc{\gamma}}
\newcommand{\gainAuxNatY}{\hc{\gamma}^{(\text{y})}}
\newcommand{\gaindbabs}{\hc{\gamma}^\text{dB}_\text{abs}}
\newcommand{\slf}{\hc{s}}
\newcommand{\wavelength}{\hc{\lambda}}
\newcommand{\dataTransfer}{\hc{D}}
\newcommand{\dataTransferHor}{\hc{D}_{\text{UE}}^{(\text{x})}}
\newcommand{\dataTransferHorApprox}{\hc{\tilde{D}}_{\text{UE}}^{(\text{x})}}
\newcommand{\dataTransferHorAux}{\hat{\dataTransfer}^{(\text{x})}}
\newcommand{\dataTransferHorAuxAux}{\check{\dataTransfer}^{(\text{x})}}
\newcommand{\dataTransferVer}{\hc{D}_{\text{UE}}^{(\text{y})}}
\newcommand{\vecElemFirst}{\hc{\pmb{i}}_{1}}
\newcommand{\vecElemSecond}{\hc{\pmb{i}}_{2}}
\newcommand{\varAuxOne}{\hc{u}}
\newcommand{\varAuxTwo}{\hc{v}}
\newcommand{\intAuxOne}{\hc{\mathcal{I}}_1}
\newcommand{\intAuxTwo}{\hc{\mathcal{I}}_2}
\newcommand{\noisepower}{\hc{\sigma}^2}

\newcommand{\complexity}{\hc{\mathcal{O}}}
\newcommand{\complexityfun}{\hc{{f}}}

\newcommand{\snrmincc}{\hc{\text{SNR}}_\text{cc}^\text{min}}

\newcommand{\numues}{\hc{M}}
\newcommand{\setueinds}{\hc{\mathcal{M}}}
\newcommand{\indue}{\hc{m}}
\newcommand{\setgridptsmultiues}{\hc{\tilde{\mathcal{D}}}^{\text{UE}}}
\newcommand{\ptmid}{\hc{\pmb{p}}_{\text{mid}}}
\newcommand{\ptueabove}{\hc{\pmb{p}}_{\text{UE}}^{\text{top}}}
\newcommand{\ptbsabove}{\hc{\pmb{p}}_{\text{BS}}^{\text{top}}}

\newcommand{\meanconn}{\hc{\bar{T}}_{\text{c}}}
\newcommand{\probfail}{\hc{P}_{\text{f}}}

\newcommand{\trajectoryyaxis}{\hc{\mathcal{T}}_{\text{y}}}
\newcommand{\trajectoryttc}{\hc{\mathcal{T}}_{\text{ttc}}}
\newcommand{\trajectorydat}{\hc{\mathcal{T}}_{\text{dat}}}
\newcommand{\trajectoryxaxis}{\hc{\mathcal{T}}_{\text{x}}}

\newcommand{\absorptionx}{\absorption_{\text{x}}}
\newcommand{\absorptiony}{\absorption_{\text{y}}}

\newcommand{\absorptiondat}{\absorption_{\text{dat}}}
\newcommand{\absorptionttc}{\absorption_{\text{ttc}}}

%% file: intro.tex

\begin{bullets}
    \blt[motivation]
    \begin{bullets}
        \blt[comms with UAVs]
        \begin{bullets}
            \blt[application]Autonomous uncrewed aerial vehicles (UAVs) received great attention in wireless communications due to their ability to extend the coverage of
            cellular  networks~\cite{viet2022introduction,zhao2019emergency}.            
            \blt[example] This need arises e.g. when terrestrial
            infrastructure is absent (as in remote areas), damaged by a natural or man-made disaster (as in military attacks), or not operational (as in the recent Spanish blackout~\cite{nopower2025}).
            \blt[goal]In these situations, UAVs can be used to serve cellular users. 
        \end{bullets}
    \end{bullets}

    \begin{figure}
        \centering
        \captionsetup{width=\linewidth}
        \includegraphics[width=.81\linewidth, trim={0cm 3cm 0cm 3cm},clip]{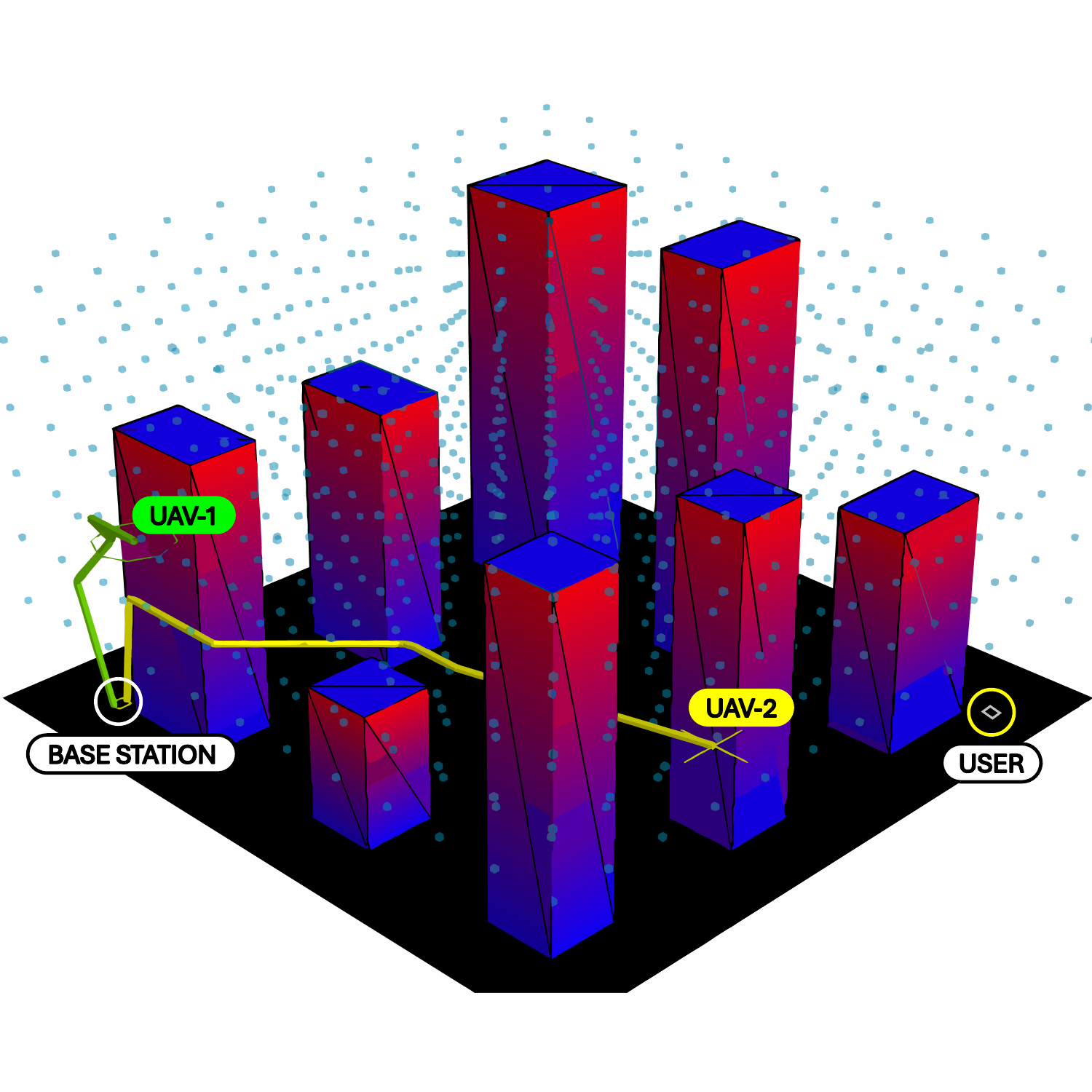}
        \captionof{figure}{Trajectories of two relay UAVs obtained with the proposed algorithm. Red/blue boxes represent buildings.  The flight grid points are represented as blue dots. The green and yellow lines denote the trajectories of the UAVs. 
        }
        \label{fig:environment}
    \end{figure}

    \blt[problems \ra time sensitive \ra relays]Several research problems emerge to address this need.
    \begin{bullets}%
        \blt[ABS placement]For example, in \emph{aerial base station (ABS) placement}, the goal is to find positions where the UAVs should hover to serve the users~\cite{viet2022introduction}. In this problem, the UAVs remain therefore static.
        \blt[Relay path planning]However, some applications involve time-sensitive requirements (e.g. an avalanche warning must be delivered to users in a mountain) or moving users (e.g. a police pursuit). For those scenarios, one needs to  plan the path of the UAVs to fly to suitable positions where they can serve the users. This is known as \emph{relay path planning}.%
    \end{bullets}%

    \blt[literature]
    \begin{bullets}%
        \blt[papers]This problem has spurred  extensive recent research interest; see e.g. \cite{wang2018power,zeng2016relaying,jiang2019trajectory,sun2021trajectory,chen2020trajectory,zhang2017trajectory,zhang2022cooperative,zhang2018multi,liu2021relaying,ghazzai2018dual,lee2022trajectory,yanmaz2022positioning,yanmaz2023joint,yanmaz2024dynamic}.
        \blt[limitations]However, as detailed in Sec.~\ref{sec:related_work},
        \begin{bullets}%
            \blt[all]all existing schemes suffer from two main limitations:
            \begin{itemize}%
                \item[(L1)]\cmt{flight constraints} None of them can accommodate arbitrary flight constraints. Although they can often deal with convex constraints, flight regions in practice are non-convex due to the presence of obstacles (e.g.  buildings and mountains) or  no-fly zones (e.g. airports and military bases).
                Height limits in the EU and US are also specified with respect to the ground level, which results in non-convex regions%
                \begin{extendedonly}%
                ~whenever the terrain is not flat    
                \end{extendedonly}.

                \item[(L2)]\cmt{channel models} No scheme can accommodate arbitrary channel
                models. In fact, the vast majority of them assume free-space propagation, which does not generally hold.
            \end{itemize}

            Furthermore, all existing schemes suffer from at least two of the following limitations:
            \begin{bullets}
                \blt[2D] (L3) the paths of the relays are 2D, i.e. confined to a horizontal plane,
                \blt[static users] (L4) they require that the users remain static,
                \blt[one relay] (L5) they can only handle one relay,
                \blt[no theoretical guarantees] (L6) their optimality cannot be guaranteed, and
                \blt[limited objectives] (L7) they cannot accommodate a general  family of objective functions.                 
            \end{bullets}
            See Table~\ref{tab:literature} and Sec.~\ref{sec:related_work} for details.
            \begin{table*}[!t]
                \begin{changes}                    
                \begin{center}
                \caption{Comparison of the proposed approach with existing works. * means that the sequence of iterates converges to a non-necessarily optimal solution.}
                \label{tab:literature}
                \begin{tabular}{ |p{1.2cm}|p{8cm}|>{\centering}p{1.3cm}|>{\centering}p{.3cm}|>{\centering}p{1cm}|>{\centering}p{.8cm}|>{\centering}p{.9cm}|>{\centering\arraybackslash}p{1.3cm}| }
                    \hline
                     & \vspace{.4em}{Objective}                                                       & Arbitrary Flight Constraints & \vspace{.4em}{3D} & Arbitrary Channel Model & 
                     \vspace{.0em} Moving User & \vspace{.0em}Multiple Relays & \vspace{.0em}\arraybackslash Theoretical Guarantees \\
                    \hline
                    \cite{wang2018power, chen2020trajectory}              & Maximize the minimum (across users) average (over time slots) rate.               & &                                                               &                                                        &                                             &                                                 &   * \\
                    \hline
                    \cite{zeng2016relaying}           & Maximize the summation of the user's throughput in all time steps. &                    &                                                               &                                                        & \centering{$\checkmark$}               &                                                 & {*}\\
                    \hline
                    {\cite{jiang2019trajectory}}        & Maximize the minimum (across terminal pairs) average (over time slots) rate.  &     &                                                               &                                                        &                                             &                                                 & {*}                                      \\
                    \hline
                    {\cite{sun2021trajectory}}          & Minimize the energy to relay given amounts of data in all pairs of users.    &          &                                                               &                                                        &                                             &                                                 & {*}                                      \\[.3cm]
                    \hline
                    {\cite{zhang2017trajectory}}        & Minimize the sum (over all time slots) of the probability of outage.   &          &                                                               &                                                        &                                             &                                                 & {*}                                                       \\
                    \hline
                    {\cite{zhang2022cooperative}}       & Maximize the average (over all time slots) rate.     &                                  &                                                               &                                                        &                                             & {$\checkmark$}                   & {*}                                      \\
                    \hline
                    {\cite{zhang2018multi}}             & Maximize the total amount of data received by the user.       &                         &                                                               &                                                        &                                             & \centering{$\checkmark$}                   & {*}                                      \\
                    \hline
                    {\cite{liu2021relaying}}            & Maximize the minimum (across terminal pairs) average (over time slots) rate.  &     &   \centering{$\checkmark$}                                                            &                                                        &                                             & {$\checkmark$}                   & {*}                                      \\
                    \hline
                    {\cite{ghazzai2018dual}}            & Minimize the (weighted) cumulative time needed to relay data of all terminal pairs. & &                                                               &                                                        &                                             & {$\checkmark$}                   & {*}                                      \\[.3cm]
                    \hline
                    {\cite{lee2022trajectory}}          & Minimize the flight time to relay data to all clusters of users.                &       &                                                               &                                                        &                                             & {$\checkmark$}                   & {*}                                      \\
                    \hline
                    {\cite{yanmaz2022positioning, yanmaz2023joint}}      & Minimize the number of relay UAVs.      &                                               &                                                               &                                                        & {$\checkmark$}               & {$\checkmark$}                   &                                                         \\
                    %
                    \hline
                    {\cite{yanmaz2024dynamic}}          & Minimize the summation of flight distances. &                                           &                                                               &                                                        & {$\checkmark$}               & {$\checkmark$}                   &                                                         \\
                    \hline
                    {\centering{\textbf{Proposed}}} & Arbitrary additive  objective. &                 {\centering{$\checkmark$}}                                        & {\centering{$\checkmark$}}                     & {\centering{$\checkmark$}}              & {\centering{$\checkmark$}}   & {\centering{$\checkmark$}}       & Optimality guarantees.  \\
                    \hline
                \end{tabular}
                \end{center}
                \end{changes}
            \end{table*}
        \end{bullets}


    \end{bullets}%


    \blt[contributions]The main contribution\footnote{\label{footnote:relconf} \arevtwo{Parts of this work have been presented at the IEEE Global Communications Conference 2023~\cite{viet2023roadmaps}. Relative to that conference version, the present paper includes a generalization of the considered channel model to accommodate an arbitrary path loss exponent, an algorithm for serving moving users, additional objective functions, Theorem 2, an extension for more than 2 relays, an extension for more than 1 user, a complexity analysis, a large number of numerical experiments with additional benchmarks, and {the analysis of the influence of the objective function on the UAV trajectory.}}} of this paper
    \begin{bullets}
        \blt[framework]is a \emph{general framework for relay path planning}  that addresses all the above limitations. To understand the main idea,
        \begin{bullets}%
            \blt[existing]%
            \begin{bullets}%
                \blt[aerial relay literature \ra no discretization]note that the reason why the aerial relay schemes in the \emph{communications literature}  cannot accommodate arbitrary flight constraints and channel models is that they do not discretize space. Treating the spatial coordinates  as continuous variables facilitates planning the paths of multiple relays, but is only applicable in free space.
                \blt[UAV literature \ra discretization]On the other hand, the paths of UAVs in the \emph{robotics literature} are typically planned in a discretized space; see e.g. \cite{hong2021path}. This enables obstacle avoidance but hinders the usage of more than one relay due to the exponential growth of the number of possible states with respect to the number of UAVs and waypoints. Besides, these schemes cannot be applied to the relay path planning problem, which involves communication metrics and constraints, such as the need to maintain connectivity among the relays and with a terrestrial base station (BS) throughout the trajectories. %
            \end{bullets}%

            \blt[proposed]
            \begin{bullets}%
                \blt[general]To combine the merits  of both bodies of literature, this paper proposes \emph{the first  framework for relay path planning that is  based on spatial discretization}.  To cope with the exponential growth of the state space while accommodating communication constraints,
                the idea here is to adapt the \emph{probabilistic roadmap} (PR) algorithm~\cite{kavraki1996roadmaps},   which finds  shortest-paths on graphs of randomly generated states.
                \blt[extension]The adaptation is based on  a custom probability distribution for  generating these states, which in turn relies on a heuristic   that produces 
                feasible waypoint sequences. While the complexity of plain PR would grow exponentially with the number of UAVs, the complexity of the proposed algorithm grows just linearly. 

            \end{bullets}

            \blt[strengths]The resulting framework, referred to as \emph{PR with feasible initialization} (PRFI), features several strengths.
            \begin{bullets}%
                \blt[objectives] First, (S1) PRFI can optimize arbitrary  objectives so long as they are additive over time.
                \blt[arbitrary channel] (S2) These objectives, which will typically capture communication performance metrics, can involve arbitrary channel models, not necessarily the free-space model.
                \blt[flight constraints]The adopted objective is optimized subject to (S3) arbitrary constraints on the flight region and (S4) the constraint that  the relays remain connected throughout their trajectories.
                \blt[multiple relays] In addition, (S5) multiple relays and
                \blt[moving users](S6) moving users can be accommodated.
                \blt[theoretical guarantees] Finally, (S7) the optimality of PRFI can be established theoretically  in certain  cases.%
            \end{bullets}%
        \end{bullets}%

        \blt[algorithms]The proposed PRFI framework is applied to three specific objectives: (i) the time it takes to establish connectivity with a  user, (ii) the outage time, and  (iii) the amount of transferred data. The case of a static user is treated separately since it allows greater optimality by avoiding uniform time discretization. An example in this scenario  is shown in Fig.~\ref{fig:environment}. The heuristic waypoint sequences used by PRFI are shown to be optimal in certain scenarios. 
        Videos of trajectories can be found along with the developed simulator and  simulation code   
        at \url{https://github.com/uiano/pr_for_relay_path_planning}.
    \end{bullets}%


    %

    \begin{changes}
        \blt[paper structure]
        Sec.~\ref{sec:related_work} summarizes the related work.        
        Afterwards,
        Sec.~\ref{sec:problem}  introduces the model and formulates the problem. The PR approach is then reviewed in
        Sec.~\ref{sec:introroadmap}.
        Specific PRFI algorithms are derived   in Secs.~\ref{sec:static} and~\ref{sec:moving} and extensions are presented in Sec.~\ref{sec:extensions}. 
        Finally, numerical experiments and conclusions are respectively presented in Secs.~\ref{sec:experiments} and~\ref{sec:conclusions}.
         Complete proofs of the theorems and further details can be found in \cite{viet2026planningarxiv}.
    \end{changes}%

    \blt[notation]\emph{Notation: }%
        \begin{bullets}%
            \blt[set] Sets are notated by uppercase calligraphic letters.
            \blt[cardinality] $|\mathcal{S}|$ is the cardinality of set $\mathcal{S}$.
            \blt[cartesian product] $\mathcal{A}\times\mathcal{B}\define\{(a,b):a\in\mathcal{A}, b\in\mathcal{B}\}$ is the Cartesian product of sets $\mathcal{A}$ and $\mathcal{B}$, where $(a,b)$ denotes a tuple.
            %
            %
            \blt[r field positive] $\rfield_+$ is the set of non-negative real numbers.
            \blt[matrix vector] Boldface uppercase (lowercase) letters denote matrices (column vectors).
            \blt[Euclidean distance] $\|\loc\|$ stands for the $\ell_2$-norm of vector $\loc$.
            \blt[indicator]$\funindicator[.]$ is a function that returns 1 if the condition inside is true and 0 otherwise.
            \blt[min operator] $\min(a,b)$ and $\max(a,b)$ respectively denote the minimum and maximum between $a$ and $b$.
            \blt[1st derivatives] $\dot\loc(t)$ stands for the entrywise first derivative of $\loc(t)$ with respect to $t$.
            \blt[floor] $\lfloor a \rfloor$ is the largest integer that is less than or equal to $a$.
        \end{bullets}%

\end{bullets}%

%% file: algos_static.tex

\begin{algorithm}[t!]
    \caption{Tentative Path UAV-2, Static UE}
    \label{algo:trajtwo}
    \begin{algorithmic}[1]
        \item[] \textbf{input:} $\grid,\locbs,\locue,\loc_2[0],\minuavrate,\minuerate, \capacity$

        \STATE Find the candidate locations of UAV-2 \ra  $\nodesetut=\rateset(\locbs,2\minuavrate,\minuavrate)$
        \label{step:static-ue:uav2-candidates}
        
        \STATE Find the destinations of UAV-2 \ra $\destsetut=\rateset(\locbs, 2\minuavrate + \minuerate,\minuavrate + \minuerate)\cap \rateset(\locue,\minuerate)$
        \label{step:static-ue:uav2-destinations}

        \STATE Construct graph $\graphuavtwo$  with weights  $\weightfun(\loc,\loc') = \|\loc-\loc'\|$

        \STATE \textbf{return} $\pathsingle_2\define$ shortest\_path($\loc_2[0]$, $\destsetut$)
    \end{algorithmic}
\end{algorithm}

\begin{algorithm}[t!]
    \caption{Tentative Path UAV-1, Static UE}
    \label{algo:trajone}
    \begin{algorithmic}[1]
        \item[] \textbf{input:} $\grid,\locbs,\locue$, $\loc_1[0]$, $\pathsingle_2$, $\minuavrate,\minuerate$, $\capacity$

        \FOR{$\indlift=0,1,\ldots$}

        \STATE $\{\exloc_2[0],\exloc_2[1],\ldots,\exloc_2[\exnumwp_\indlift-1]\}
            \define\liftfun^{(\indlift)}(\pathsingle_2)
        $
        \label{step:static-ue:lift}

        \STATE Find candidate locations of UAV-1 \ra $\nodesetuo[\indwp]=\rateset(\locbs, 2\minuavrate)\cap \rateset(\exloc_2[\indwp],\minuavrate)$
        \label{step:static-ue:uav1-candidates}

        \STATE Find destinations of UAV-1 \ra $\destset_1=\rateset(\locbs, 2\minuavrate + \minuerate)\cap \rateset(\exloc_2[\exnumwp_\indlift-1],\minuavrate + \minuerate)$
        \label{step:static-ue:uav1-destinations}


        \STATE Construct extended graph $\graphuavone$ with weights as in Sec.~\ref{sec:static:pathuav1}.
        \label{step:static-ue:graph}

        \IF{path\_exists($(0,\loc_1[0]),\mathbb{N}\times\destsetuo$)}

        \STATE
        $\{(\indwp_0,\loc_1[0]),(\indwp_1,\loc_1[1]),\ldots, (\indwp_{\numwpwithwaits-1},\loc_1[\numwpwithwaits-1])\}$ \\~\hspace{.4cm}$\define$ shortest\_path($(0,\loc_1[0]),\mathbb{N}\times\destsetuo$)

        \STATE Obtain $\confpt[\indwpwithwaits] = [\loc_1[\indwpwithwaits], \exloc_2[\indwp_{\indwpwithwaits}]]$, $\indwpwithwaits=0,1,\ldots,\numwpwithwaits-1$

        \STATE  \textbf{return} $\trajectorywpsvalid=\{\confpt[0],\confpt[1],\ldots, \confpt[\numwpwithwaits-1]\}$
        \ENDIF

        \ENDFOR
    \end{algorithmic}
\end{algorithm}

\begin{algorithm}[t!]
    \caption{PRFI for Static UE}
    \label{algo:prfi-static}
    \begin{algorithmic}[1]
        \item[] \textbf{input:} $\grid,\locbs,\locue,\confpt_0,\numconfpt, \maxuavspeed,\minuavrate,\minuerate,\capacity$

        \STATE \label{step:tentative} $\pathsingle_2\define$  tentative path for UAV-2 via Algorithm \ref{algo:trajtwo}

        \STATE \label{step:combine} $\trajectorywpsvalid\define$ combined tentative path via Algorithm \ref{algo:trajone}

        \STATE  For each $\confpt \in
            \trajectorywpsvalid $, draw
        $\lfloor \numconfpt/\numwpwithwaits\rfloor$ CPs \ra $\prnodeset$

        \STATE Construct a nearest-neighbor graph from $\prnodeset$

        \STATE $\trajectorywpspr\define$ shortest\_path($\confpt_0$, $\{\confpt \in \prnodeset:\uerate(\confpt)\geq\minuerate\}$)
        \label{step:static-shortest-path}

        \STATE Compute waypoints  $\{(t_\indwp,\confpt(t_\indwp))\}_\indwp$ as  in Sec.~\ref{sec:fromwptotrajstaticue}

        \STATE \textbf{return}  $\{(t_\indwp,\confpt(t_\indwp))\}_\indwp$

    \end{algorithmic}
\end{algorithm}

%% file: algos_moving.tex

\begin{algorithm}[t!]
    \caption{Tentative Path UAV-2, Moving UE}
    \label{algo:pathtwo-movingue}
    \begin{algorithmic}[1]
        \item[] \textbf{input:} $\grid,\locbs,\loc_2[0],$
        \begin{nonextendedonly}
            $\trajectoryue, $
        \end{nonextendedonly}
        \begin{extendedonly}
            \\~\hspace{.8cm} $\trajectoryue =\{\locue[0],\locue[1],\ldots,\locue[\numwpue-1]\}, $
        \end{extendedonly}
        $\minuavrate,\minuerate,\capacity$

        \STATE Find the candidate locations of UAV-2 \ra $\nodesetut[\indwp]=\rateset(\locbs,2\minuavrate,\minuavrate)$
        \label{step:moving-ue:uav2-candidates}

        \STATE Find the destinations of UAV-2 \ra $\destsetut[\indwp]=\rateset(\locbs,2\minuavrate+\minuerate,\minuavrate+\minuerate)\cap\rateset(\locue[\indwp],\minuerate)$
        \label{step:moving-ue:uav2-destinations}

        \STATE Construct extended graph $\graph_3$ with weights  \eqref{eq:costmovingueuavt}

        \STATE $\pathsingle_2\define$ shortest\_path($(0,\loc_2[0])$, $\enodesetut[\numwpue-1]$)

        \STATE \textbf{return} $\pathsingle_2=\{\loc_2[0],\loc_2[1],\ldots,\loc_2[\numwpue-1]\}$
    \end{algorithmic}
\end{algorithm}

\begin{algorithm}[t!]
    \caption{Tentative Path UAV-1, Moving UE}
    \label{algo:pathone-movingue}
    \begin{algorithmic}[1]
        \item[] \textbf{input:} $\grid,\locbs,\loc_1[0], \pathsingle_2,\minuavrate,\minuerate,\capacity$

        \FOR{$\indlift=0,1,\ldots$}

        \STATE $\{\exloc_2[0],\exloc_2[1],\ldots,\exloc_2[\numwpue-1]\}\define\lifttrimfun^{(\indlift)}(\pathsingle_2)$
        \label{step:moving-ue:lift}

        \STATE Find the candidate locations of UAV-1 \ra $\nodesetuo[\indwp]=\rateset(\locbs,2\minuavrate)\cap\rateset(\exloc_2[\indwp],\minuavrate)$
        \label{step:moving-ue:uav1-candidates}

        \STATE Find the destinations of UAV-1 \ra $\destsetuo[\indwp]=\rateset(\locbs,2\minuavrate+\minuerate)\cap\rateset(\exloc_2[\indwp],\minuavrate+\minuerate)$\\[.05cm]
        \label{step:moving-ue:uav1-destinations}
        
        \STATE Form extended nodes \ra $\enodesetuo[\indwp]\define\{(\indwp,\loc)|\loc\in\nodesetuo[\indwp]\}$

        \STATE Construct extended graph $\graph_4$ with weights given in Sec.~\ref{sec:moving:tentative-uav1}.

        \IF{path\_exists($(0,\loc_1[0])$, $\enodesetuo[\numwpue-1]$)}
        \STATE  $\{(0,\loc_1[0]),\ldots, (\numwpue-1,\loc_1[\numwpue-1])\}$ \\\hspace{.5cm}$=$~shortest\_path($(0,\loc_1[0])$,~$\enodesetuo[\numwpue-1]$)

        \STATE Obtain $\confpt[\indwp] = [\loc_1[\indwp], \exloc_2[\indwp]]$, $\indwp=0,1,\ldots,\numwpue-1$

        \STATE \textbf{Return} $\trajectorywpsfeas=\{\confpt[0],\confpt[1],\ldots, \confpt[\numwpue-1]\}$
        \ENDIF
        \ENDFOR
    \end{algorithmic}
\end{algorithm}

\begin{algorithm}[t!]
    \caption{PRFI for Moving UE}
    \label{algo:prfi-moving}
    \begin{algorithmic}[1]
        \item[] \textbf{input:} $\grid,\locbs,\confpt_0, \trajectoryue,\numconfpt,\sampint,\minuavrate,\minuerate,\capacity$

        \STATE $\pathsingle_2\define$ tentative path for UAV-2 via Algorithm \ref{algo:pathtwo-movingue}

        \STATE $\trajectorywpsfeas\define$ combined tentative path via Algorithm \ref{algo:pathone-movingue}

        \STATE  For each $\confpt[\indwp] \in
            \trajectorywpsfeas$, draw
        $\lfloor \numconfpt/\numwpue\rfloor$ CPs \ra $\enodeset[\indwp]$

        \STATE Construct extended graph with weights as in Sec.~\ref{sec:moving-qspace}

        \STATE  $\trajectorywpspr\define$   shortest\_path($(0,\confpt_0)$, $\enodeset[\numwpue-1]$)

        \STATE \textbf{return} times and waypoints $\{(\indwp\sampint,\confpt[\indwp])\}_\indwp$.

    \end{algorithmic}
\end{algorithm}

%% file: complexity.tex
    
        \begin{bullets}%
            \blt[params] Assume for simplicity that
            $\numues=1$,             
            $\oobregion$ is a cube of side $\arealength$, and  $\spacingx=\spacingy=\spacingz=\spacingscalar>0$.
            \begin{bullets}%
                \blt The number of grid points is then roughly $\numgridpt=\arealength^3/\spacingscalar^3$, which implies that $\spacingscalar=\arealength/\sqrt[3]{\numgridpt}$.
            \end{bullets}%
            %
            \blt[rate sets]Overall, the proposed algorithms utilize  $\nodeset_{\induav}[\indwp]$ and $\destset_{\induav}[\indwp]$; cf. Sec.~\ref{sec:multi-uavs}. This requires   $\nodesetpre_\induav\define\rateset(\locbs,\numuav\minuavrate,(\numuav-1)\minuavrate,\ldots,(\numuav-\induav+1)\minuavrate)$ and $\destsetpre_\induav\define\rateset(\locbs,\numuav\minuavrate+\minuerate,(\numuav-1)\minuavrate+\minuerate,\ldots,(\numuav-\induav+1)\minuavrate+\minuerate),\forall\induav=1,\ldots,\numuav$.
            \begin{bullets}%
                \blt[capacity map] Obtaining these sets does not count towards complexity as it only requires comparisons -- the capacity map $\capacity$  already provides $\capacity(\loc,\loc')$ for all $\loc,\loc'\in\grid$.
            \end{bullets}%
            \blt[dijkstra complexity] Therefore, the complexity of the proposed algorithms mainly stems from the adopted shortest path algorithm. In the case of Dijkstra's, its standard implementation has  complexity  $\complexity\left((V+E)\log V\right)$ \cite{mehlhorn2008algorithm}, where $V$ and $E$ are the numbers of vertices and edges of the considered graph. 
            \blt[complexity of prop. algos]Based on this, the complexity of the proposed algorithms is  obtained next.
            
            \blt[static] Start by considering the case of a static UE.
            \begin{bullets}%
                \blt[prfi]
                \begin{bullets}%
                    \blt[tentative paths]The first step is to obtain the  complexity of computing the tentative path. 
                        Suppose initially that  $\numuav=2$.
                    \begin{bullets}%
                        \blt[uav 2]  Planning the path of UAV-2 involves  $\numgridpt$ nodes and $26\numgridpt/2$ edges, which requires $\complexity((\numgridpt+13\numgridpt)\log(\numgridpt))=\complexity(\numgridpt\log(\numgridpt))$ operations.
                        \blt[uav 1] If this path has $\numwpten$ nodes, planning the path of UAV-1 involves  an extended graph with $\numwpten\numgridpt$ nodes and $27\numwpten\numgridpt$ edges. This requires $\complexity((\numwpten\numgridpt+27\numwpten\numgridpt)\log(\numwpten\numgridpt))=\complexity(\numwpten\numgridpt\log(\numwpten\numgridpt))$ operations.
                        \blt[more uavs]Proceeding along these lines for $\numuav>2$ yields
                        \begin{align}
                            \label{eq:complexitystaticue}
                            \complexity(\numuav\numwpten\numgridpt\log(\numwpten\numgridpt)).
                        \end{align}
                    \end{bullets}%
                    \blt[combined path] Planning the combined path involves $\numconfpt$ nodes and $\numconfpt\maxnumneighbors/2$ edges, where $\maxnumneighbors$ is the number of neighbors of a node. Since these are constants, the contribution of this step to the overall complexity can be neglected. 
                    \blt[summary]Thus, in short, the complexity of PRFI for static UEs is given by \eqref{eq:complexitystaticue}. 

                \end{bullets}%
                \blt[comparison]
                \begin{bullets}%
                    %
                    \blt[other works] This  complexity  can be related to existing schemes. To this end, note that:
                    \begin{itemize}
                        \item the complexity of \cite{zhang2022cooperative} and \cite{liu2021relaying} is 
                              $\complexity(\numuav\numwp^{3.5}\log(1/\epsilon))$, where $\epsilon=10^{-3}$ is the solution accuracy.
                        \item{}\cite{zhang2018multi} has  complexity  $\complexity(\numlift(\numuav\numwp)^{3.5})$, where $\numlift$ is the number of iterations.
                        \item  the complexity of \cite{zeng2016relaying} is only said to be polynomial w.r.t $\numwp$, but it is likely larger than \eqref{eq:complexitystaticue}.
                        \item obtaining an exact  solution in \cite{ghazzai2018dual, lee2022trajectory} is NP-hard (although polynomial time approximations can be used). 
                    \end{itemize}
                    \blt[conclusion] Thus, the complexity of existing schemes in terms  of $\numwp$ typically grows as $\complexity(\numwp^{3.5})$, much faster than with the proposed algorithm, for which complexity grows as
                    $\complexity(\numwp\log(\numwp))$.%
                \end{bullets}%
            \end{bullets}%
            
            \blt[moving]Next, consider the case of a moving UE. 
            \begin{bullets}%
                \blt[prfi]
                \begin{bullets}%
                    \blt[tentative paths]First, the computation of the tentative path will be considered. 
                    \begin{bullets}%
                        \blt[last uav] For planning the path of UAV-$\numuav$,
                        \begin{bullets}%
                            \blt[node set]one needs $\nodeset_{\numuav}[\indwp]=\nodesetpre_{\numuav}$, which is already known.
                            \blt[dest set]
                            \begin{bullets}%
                                \blt[one time step]One also needs  $\destset_{\numuav}[\indwp]=\destsetpre_\numuav\cap\rateset(\locue[\indwp],\minuerate)$, 
                                which in turn requires computing $\rateset(\locue[\indwp],\minuerate)$.            
                                Since this involves  $\numgridpt$ operations,
                                %
                                \blt[all time steps] finding $\destset_{\numuav}[\indwp]~\forall\indwp$ has  complexity  $\numwpue\numgridpt$.
                            \end{bullets}%
                            \blt[planning a path]Having obtained  $\nodeset_\numuav[\indwp]$ and $\destset_\numuav[\indwp]$, the next step is to plan the path of UAV-K through an extended graph with $\numwpue\numgridpt$ nodes and $13\numwpue\numgridpt$ edges, which has complexity 
                            $\complexity\left(\left(\numwpue\numgridpt + 13\numwpue\numgridpt\right)\log\left(\numwpue\numgridpt\right)\right)$
                            $=\complexity\left(\numwpue\numgridpt\log\left(\numwpue\numgridpt\right)\right)$.
                        \end{bullets}%
                        \blt[uav-k] When planning the  paths of the remaining UAVs,
                        \begin{bullets}%
                            \blt[all time steps] one must compute $\nodeset_{\induav}[\indwp]$ and $\destset_{\induav}[\indwp]$, $\forall\indwp$, which only requires comparisons.
                            %
                            \blt[planning a path]After that, planning the  path of UAV-$\induav$ has the same complexity as planning the path of UAV-$\numuav$, i.e. $\complexity\left(\numwpue\numgridpt\log\left(\numwpue\numgridpt\right)\right)$.
                        \end{bullets}%
                        \blt[sum up]To sum up, the complexity of planning the tentative path is $\complexity(\numwpue\numgridpt+\numuav\numwpue\numgridpt\log\left(\numwpue\numgridpt\right))=\complexity(\numuav\numwpue\numgridpt\log\left(\numwpue\numgridpt\right))$.
                    \end{bullets}%
                    \blt[combined path] Planning the combined path involves $\numconfptperslot\define\lfloor \numconfpt /\numwpue \rfloor+1$ nodes and $\numconfptperslot\numconfptperslot=\numconfptperslot^2$ edges per time step, which results in a complexity of
                    $\complexity((\numconfptperslot\numwpue+\numconfptperslot^2(\numwpue-1))\log(\numconfptperslot\numwpue))$ $=\complexity(\numwpue\log(\numwpue))$.
                    \blt[total complexity] Thus, the total complexity is
                    \begin{align}
                        \label{eq:compmovinguenuengrid}
                        \complexity(\numuav\numwpue\numgridpt\log(\numwpue\numgridpt)).
                    \end{align}%
                \end{bullets}%
                \blt[other works]Unlike the static UE case, the complexity in
            \eqref{eq:compmovinguenuengrid} cannot be compared with the
            literature since there is only one scheme for serving moving
            users~\cite{yanmaz2022positioning} and it has no  complexity
            analysis. \end{bullets}%
        \end{bullets}%
        

%% file: diff_obj.tex
\begin{bullets}%
    %
    %
    \blt[assumption]The  UE is assumed static; the case of a moving UE is left for future work.
    \blt[scenario]
    \begin{bullets}%
        \blt[area] Consider a simple scenario where a  single UAV is used. The BS and the UE are at opposite corners of a building; see Fig.~\ref{fig:diff_directions}.
        \begin{figure}
            \centering
            \caption{Scenario for the analysis of the influence of the objective function on the optimal trajectory.}
            \includegraphics[trim={2.1cm 1.5cm 2.2cm 3.2cm},clip,width=0.8\linewidth]{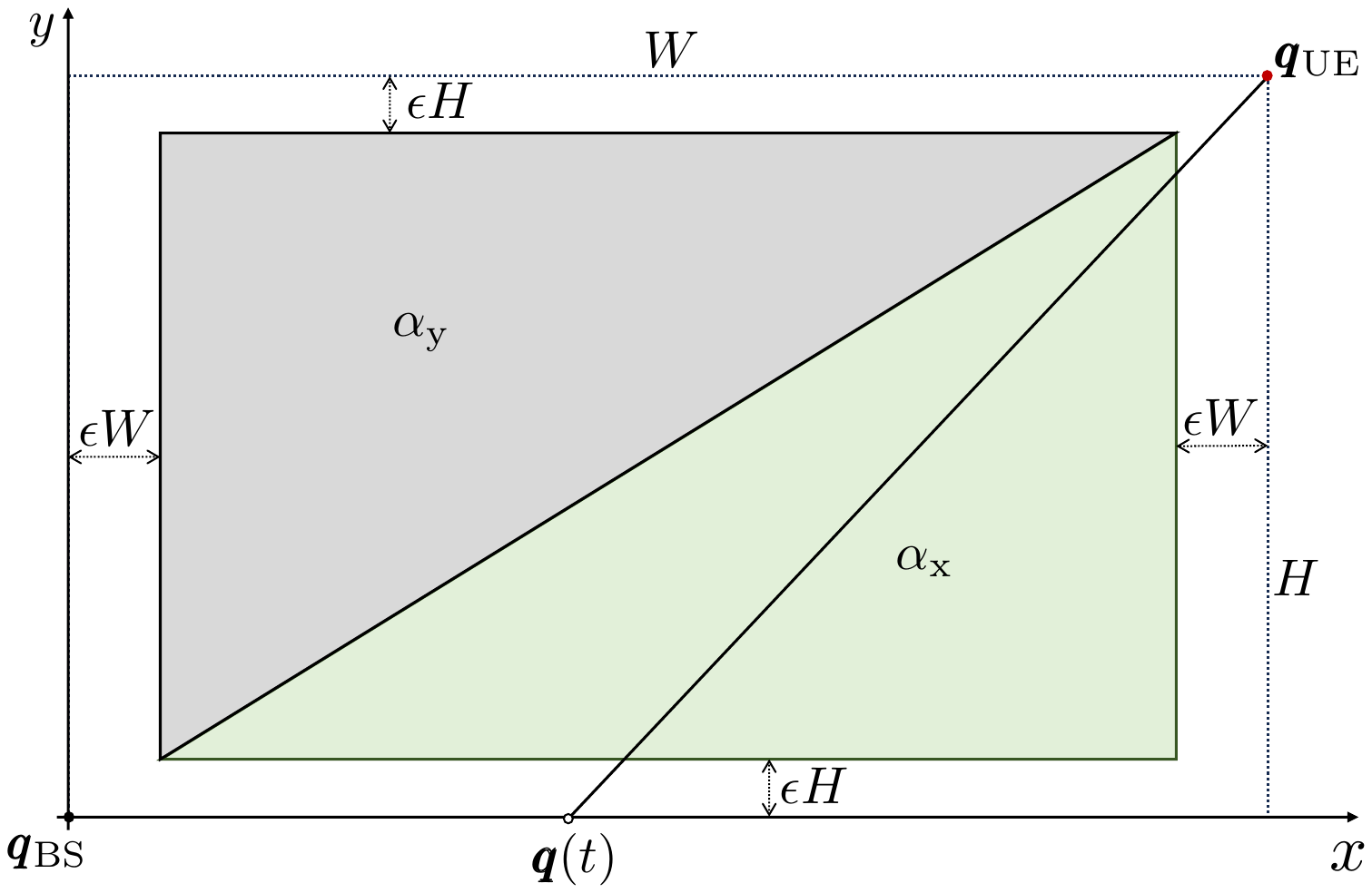}
            \label{fig:diff_directions}
        \end{figure}
        \blt[locs] Specifically, the BS and the UE are respectively at the ground locations $\locbs=[0,0]\transpose$ and $\locue=[\buildingWidth,\buildingHeight]\transpose$, where  $\buildingHeight$ and $\buildingWidth$ are two constants satisfying $\buildingHeight<\buildingWidth$.
        \blt[building]
        The floor plan of the building spans the rectangle $[\spaceToBuilding\buildingWidth, \buildingWidth-\spaceToBuilding\buildingWidth] \times [\spaceToBuilding\buildingHeight, \buildingHeight-\spaceToBuilding\buildingHeight]$, where $\spaceToBuilding>0$ is arbitrarily small.
        \blt[uav]
        \begin{bullets}%
            \blt[init loc] The UAV flies with speed $\maxuavspeed$ from the initial location $\locbs$.
            \blt[loc]For simplicity, the height of the UAV is 0. Thus, its position  at time instant $\timeInstant$ can be represented by the 2D vector $\loc(\timeInstant)\in\rfield^2$.
        \end{bullets}%

        \blt[communication]
        \begin{bullets}
            \blt[tx power] The UAV transmits with fixed power $\txpower$, wavelength
            $\wavelength$, and bandwidth $\bandwidth$. For simplicity, the BS and UAV transmit
            on orthogonal channel resources and, thus, they do not  interfere each
            other. The transmit power and bandwidth of the BS are sufficiently large
            so that the UE rate coincides with the rate of the UAV-UE link; cf.
            Sec.~\ref{sec:communication_model}.
            \blt[channel]%
            \begin{bullets}%
                \blt[model]The tomographic model (see Sec.~\ref{sec:capmap}) is adopted.
                \blt[absorption] The absorption of the building is  $\absorptiony=+\infty$~dB/m when $\buildingWidth y\geq \buildingHeight x$  and $\absorptionx>0$ when $\buildingWidth y< \buildingHeight x$.
            \end{bullets}
            \blt[noise] The noise power at the UE is $\noisepower$.
        \end{bullets}

        \blt[objective]The paper considers three objectives: \eqref{eq:timeconnectobj}, \eqref{eq:outageobj}, and \eqref{eq:cumrateobj}.
        Since the UE does not move, \eqref{eq:timeconnectobj} and \eqref{eq:outageobj} are equivalent. Thus, the following comparison focuses on the \emph{connection time} \eqref{eq:timeconnectobj} and the \emph{transferred data} \eqref{eq:cumrateobj} objectives.
        \begin{extendedonly}
            \begin{bullets}
                \blt[connection]Regarding \eqref{eq:timeconnectobj}, recall that the UE is said to be connected once $\uerate(\timeInstant)\geq\minuerate$.
                \blt[data]When it comes to \eqref{eq:cumrateobj}, there is no minimum required rate.
            \end{bullets}
        \end{extendedonly}

        \blt[Possible trajectories]
        \begin{bullets}
            \blt[time horizon]  Suppose that the time horizon is $\timemaxx\define\buildingWidth/\maxuavspeed$.
            \blt[trajectories] In this case, it can be easily seen that the optimal trajectory according to \eqref{eq:timeconnectobj} and \eqref{eq:cumrateobj} is necessarily one of the following:
            \begin{itemize}
                \item $\trajectoryxaxis$: In this trajectory, the UAV flies along the x axis from $\locbs$ to $[\buildingWidth,0]\transpose$, i.e., $
                            \loc(\timeInstant)=[\maxuavspeed\timeInstant,0]\transpose,~\timeInstant\in [0,\timemaxx].$
                \item $\trajectoryyaxis$: In this trajectory, the UAV flies along the y axis and then stops: $\loc(\timeInstant)=[0,\maxuavspeed\timeInstant]\transpose,\forall\timeInstant\in[0,\timemaxy)$ and $\loc(\timeInstant)=[0,\buildingHeight]\transpose,\forall\timeInstant\in[\timemaxy,\timemaxx],$
                        where $\timemaxy\define\buildingHeight/\maxuavspeed$.
                        \begin{extendedonly}
                            \begin{bullets}%
                                \blt Note that the reason why the UAV stays at
                                $\loc(\timemaxy)=[0,\buildingHeight]\transpose$ instead of continuing right is  to avoid losing connectivity with the BS due to the infinite absorption $\absorptiony$. It does not  continue to fly along the y axis either because in that case the distance to the UE would increase.
                            \end{bullets}%
                        \end{extendedonly}
            \end{itemize}
        \end{bullets}

        \blt[optimal trajectories]
        \begin{bullets}
            \blt[notation] Let $\trajectoryttc$ and $\trajectorydat$ respectively denote the optimal trajectories according to \eqref{eq:timeconnectobj} and \eqref{eq:cumrateobj}.
            \blt[min rate]If $\minuerate$ is too large, then neither
            $\trajectoryxaxis$ nor $\trajectoryyaxis$ will result in a finite \eqref{eq:timeconnectobj}. Thus,  suppose that $\minuerate$ is such that
            the UAV can meet the requirement $\uerate(\timeInstant)>\minuerate$ at the end of both trajectories, i.e.,
            \begin{align}
                \label{eq:cond-minuerate}
                \minuerate \leq\bandwidth\log_2\left(1+\frac{\txpower\wavelength^2}{16\pi^2\noisepower\buildingWidth^2}\right).
            \end{align}
            Under this assumption, the optimal trajectory will depend on $\absorptionx$.
            \blt[extreme cases]
            If $\absorptionx$ is very large, then clearly
            $\trajectoryttc=\trajectorydat=\trajectoryyaxis$. Conversely, if $\absorptionx$ is very small, then clearly
            $\trajectoryttc=\trajectorydat=\trajectoryxaxis$.
            \blt[intermediate case]The case where $\absorptionx$ is neither too large nor too small is addressed by the following result:
            \begin{bullets}
                \blt[theorem]
                \begin{mytheorem}%
                    \label{thm:diff_directions}
                    Let
                    \begin{align}
                        \nonumber
                        \absorptionttc \define& \frac{10}{\sqrt{(\buildingHeight-\buildingWidth)^2+\buildingHeight^2}}\\&\cdot\log_{10}\left(\frac{\txpower\wavelength^2\left(2^{\minuerate/\bandwidth}-1\right)^{-1}}{16\pi^2\noisepower\left((\buildingHeight-\buildingWidth)^2+\buildingHeight^2\right)}\right)
                        \label{eq:absorption_lower}
                    \end{align}
                    and let
                    \begin{align}
                        \label{eq:alphadatdef}
                        \absorptiondat \define & \frac{1}{\dataTransferHorAuxAux}\Bigg(\bandwidth\log_2\left(\frac{\txpower\wavelength^2}{16\pi^2\noisepower}\right)\frac{\buildingWidth}{\maxuavspeed}                            \\
                        \nonumber
                                                & - \bandwidth\log_2\left(1+\frac{\txpower\wavelength^2}{16\pi^2\noisepower\buildingWidth^2}\right)\frac{\buildingWidth-\buildingHeight}{\maxuavspeed} - \dataTransferHorAux\Bigg),
                    \end{align}
                    where
                    \begin{align}
                        \nonumber
                        \dataTransferHorAuxAux
                        \define & \bandwidth\frac{\log_210}{20\maxuavspeed}\bigg(\buildingWidth\sqrt{\buildingWidth^2+\buildingHeight^2} \\&+\buildingHeight^2\tanh^{-1}\left(\frac{\buildingWidth}{\sqrt{\buildingWidth^2+\buildingHeight^2}}\right)\bigg)
                    \end{align}
                    and
                    \begin{align}
                        \nonumber
                        \dataTransferHorAux \define & \bandwidth\timemaxx\log_2\left(\buildingWidth^2+\buildingHeight^2\right) \\&-2\bandwidth\timemaxx\log_2e + 2\bandwidth\timemaxy\log_2e\tan^{-1}\left(\frac{\buildingWidth}{\buildingHeight}\right).
                    \end{align}
                    If \eqref{eq:cond-minuerate} holds
                    and
                    \begin{nonextendedonly}
                        $\absorptionttc<\absorptionx<\absorptiondat,~$
                    \end{nonextendedonly}
                    \begin{extendedonly}
                        \begin{align}
                            \label{eq:cond-ttc-dat}
                            \absorptionttc<\absorptionx<\absorptiondat,
                        \end{align}
                    \end{extendedonly}
                    then $\trajectoryttc = \trajectoryyaxis$ and $\trajectorydat = \trajectoryxaxis$.
                \end{mytheorem}%
                \blt[proof theorem]
                \begin{nonextendedonly}
                    \begin{proof}
                        See the extended version of this paper \cite{viet2026planningarxiv}.
                    \end{proof}
                \end{nonextendedonly}
                \begin{extendedonly}
                    \begin{IEEEproof}
                        \begin{bullets}
                            \blt[sketch]
                            \begin{bullets}
                                \blt[rate] The proof will first derive the UE rate over time for both $\trajectoryxaxis$ and $\trajectoryyaxis$.
                                \blt[show] Then, it will show that
                                \begin{bullets}
                                    \blt[cond1](i) if $\absorptionttc\left(\minuerate\right)<\absorptionx$, then $\trajectoryttc=\trajectoryyaxis$;
                                    \blt[cond2] and (ii), if $\absorptionx<\absorptiondat$, then $\trajectorydat=\trajectoryxaxis$.
                                \end{bullets}
                                %
                            \end{bullets}

                            \blt[common]
                            \begin{bullets}
                                \blt[gain]Due to the tomographic model, the channel gain at time instant $\timeInstant$  between the UAV and the UE
                                for sufficiently small $\spaceToBuilding$ is
                                \begin{bullets}
                                    \blt[dB]
                                    \begin{align}
                                        \label{eq:gainforabsorptiondB}
                                        \gainAux\left(\absorption,\timeInstant\right) & =20\log_{10}\left(\frac{\wavelength}{4\pi\left\|\loc(\timeInstant)-\locue\right\|}\right) -\absorption\left\|\loc(\timeInstant)-\locue\right\|,
                                    \end{align}
                                    where $\absorption$ is the absorption.
                                    In the case of $\trajectoryxaxis$, one has $\absorption=\absorptionx$. In the case of $\trajectoryyaxis$, one has $\absorption=\absorptiony=+\infty$ for $\timeInstant\in[0,\timemaxy)$ and $\absorption=0$ for $\timeInstant\in[\timemaxy,\timemaxx]$.
                                    \blt[natural]In natural units, \eqref{eq:gainforabsorptiondB} reads as
                                    \begin{align}
                                        \alignchar
                                        \begin{intermed}
                                            \gainAuxNat(\absorption,\timeInstant)  = 10^{\gainAux\left(\absorption,\timeInstant\right)/10}
                                        \end{intermed}
                                        \jumpline
                                            & \gainAuxNat(\absorption,\timeInstant) = \left(\frac{\wavelength}{4\pi\left\|\loc(\timeInstant)-\locue\right\|}\right)^2\frac{1}{10^{\absorption\left\|\loc(\timeInstant)-\locue\right\|/10}}.
                                    \end{align}
                                \end{bullets}
                                \blt[rate]The resulting UE rate is, therefore,
                                \begin{salign}
                                    ~&\uerate(\absorption,\timeInstant)\\
                                    =& \bandwidth\log_{2}\left(1+\frac{\txpower}{\noisepower}\gainAuxNat(\absorption,\timeInstant)\right) \\
                                    \alignchar
                                    \begin{intermed}
                                        =\bandwidth\log_2\left(1+\frac{\txpower}{\noisepower}\left(\frac{\wavelength}{4\pi\left\|\loc(\timeInstant)-\locue\right\|}\right)^2\frac{1}{10^{\absorption\|\loc(\timeInstant)-\locue\|/10}}\right)
                                    \end{intermed}
                                    \jumpline
                                    =&\bandwidth\log_2\left(1+\frac{\txpower\wavelength^2}{16\pi^2\noisepower\left\|\loc(\timeInstant)-\locue\right\|^210^{\absorption\|\loc(\timeInstant)-\locue\|/10}}\right).
                                    \label{eq:rate_general}
                                \end{salign}

                                \begin{bullets}
                                    \blt[along the y]Let $\ueratey(\timeInstant)$ be the UE rate at time instant $\timeInstant$ when the UAV follows $\trajectoryyaxis$.
                                    \begin{bullets}
                                        \blt[0 to t0]For $ \timeInstant\in[0,\timemaxy)$, one has that $\absorption  =\absorptiony=+\infty$ and, therefore,
                                        \begin{align}
                                            \ueratey(\timeInstant) & \approx \uerate(\infty,\timeInstant) =0.
                                        \end{align}
                                        \blt[t0 to t1]On the other hand,  for $ \timeInstant\in[\timemaxy,\timemaxx]$, one has that $\absorption = 0$ and, as a result,
                                        \begin{salign}
                                            \ueratey(\timeInstant) & =\bandwidth\log_2\left(1+\frac{\txpower\wavelength^2}{16\pi^2\noisepower\left\|\loc(\timeInstant)-\locue\right\|^2}\right)\\
                                            & =\bandwidth\log_2\left(1+\frac{\txpower\wavelength^2}{16\pi^2\noisepower\buildingWidth^2}\right)\overset{\eqref{eq:cond-minuerate}}{\geq}\minuerate.
                                            \label{eq:rate_y}
                                        \end{salign}
                                    \end{bullets}
                                    \blt[along the x]Similarly, let  $\ueratex(\timeInstant)$  be the UE rate at time instant $\timeInstant$  when the UAV follows $\trajectoryxaxis$. Since in this case $\absorption  =\absorptionx~\forall \timeInstant\in[0,\timemaxx]$, it holds that
                                    \begin{salign}
                                        ~&\ueratex(\timeInstant)\\
                                        =& \uerate(\absorptionx,\timeInstant)\\
                                        =& \bandwidth\log_2\left(1+\frac{\txpower\lambda^2}{16\pi^2\noisepower\left\|\loc(\timeInstant)-\locue\right\|^210^{\absorptionx\|\loc(\timeInstant)-\locue\|/10}}\right).
                                        \label{eq:rate_x}
                                    \end{salign}
                                \end{bullets}
                            \end{bullets}

                            \blt[connection time] Next, it will be shown that $\trajectoryttc=\trajectoryyaxis$ whenever  $\absorptionttc<\absorptionx$. To this end, note from \eqref{eq:absorption_lower} that $\absorptionttc < \absorptionx$ if and only if
                            \begin{align}
                                    & \frac{10}{\sqrt{(\buildingHeight-\buildingWidth)^2+\buildingHeight^2}}\log_{10}\left(\frac{\txpower\wavelength^2\left(2^{\minuerate/\bandwidth}-1\right)^{-1}}{16\pi^2\noisepower\left((\buildingHeight-\buildingWidth)^2+\buildingHeight^2\right)}\right) < \absorptionx.
                            \end{align}
                            Solving for $\minuerate$ yields
                            \begin{intermed}
                                \begin{align}
                                    \alignchar
                                    \begin{intermed}
                                        \Leftrightarrow \frac{\txpower\wavelength^2\left(2^{\minuerate/\bandwidth}-1\right)^{-1}}{16\pi^2\noisepower\left((\buildingHeight-\buildingWidth)^2+\buildingHeight^2\right)} < 10^{\absorptionx\sqrt{(\buildingHeight-\buildingWidth)^2+\buildingHeight^2}/10}
                                    \end{intermed}
                                    \jumpline
                                    \alignchar
                                    \begin{intermed}
                                        \Leftrightarrow \frac{\txpower\wavelength^2}{16\pi^2\noisepower\left((\buildingHeight-\buildingWidth)^2+\buildingHeight^2\right) 10^{\absorptionx\sqrt{(\buildingHeight-\buildingWidth)^2+\buildingHeight^2}/10}} < 2^{\minuerate/\bandwidth}-1
                                    \end{intermed}
                                    \jumpline
                                    \alignchar
                                    \begin{intermed}
                                        \Leftrightarrow \log_2\left(1+\frac{\txpower\wavelength^2}{16\pi^2\noisepower\left((\buildingHeight-\buildingWidth)^2+\buildingHeight^2\right) 10^{\absorptionx\sqrt{(\buildingHeight-\buildingWidth)^2+\buildingHeight^2}/10}}\right) < \frac{\minuerate}{\bandwidth}
                                    \end{intermed}
                                \end{align}
                            \end{intermed}
                            \begin{salign}
                                ~&\minuerate \\
                                >& \bandwidth\log_2\left(1+\frac{\txpower\wavelength^2}{16\pi^2\noisepower\left((\buildingHeight-\buildingWidth)^2+\buildingHeight^2\right) 10^{\absorptionx\sqrt{(\buildingHeight-\buildingWidth)^2+\buildingHeight^2}/10}}\right)
                                \\
                                =&\bandwidth\log_2\left(1+\frac{\txpower\wavelength^2}{16\pi^2\noisepower\left\|\loc(\timeInstant_0)-\locue\right\|^210^{\absorptionx\|\loc(\timeInstant_0)-\locue\|/10}}\right)
                                \\
                                =&\ueratex(\timemaxy).
                            \end{salign}
                            Combining this inequality with \eqref{eq:rate_y} results in
                            \begin{align}
                                \ueratex(\timemaxy) < \minuerate\overset{\eqref{eq:rate_y}}{\leq}\ueratey(\timemaxy)
                            \end{align}
                            and, therefore, $\trajectoryttc=\trajectoryyaxis$.

                            \blt[transferred data] Finally, it will be shown that $\trajectorydat=\trajectoryxaxis$ if  $\absorptionx<\absorptiondat$.
                            \begin{bullets}
                                \blt[along the y] The amount of transferred data when the UAV follows $\trajectoryyaxis$ is given by
                                \begin{salign}
                                    \dataTransferVer(\timemaxx) & \define\int_{0}^{\timemaxx}\ueratey(\timeInstantAux)d\timeInstantAux\\
                                    &=\int_{0}^{\timemaxy}\ueratey(\timeInstantAux)d\timeInstantAux + \int_{\timemaxy}^{\timemaxx}\ueratey(\timeInstantAux)d\timeInstantAux\\
                                    &= 0 + \bandwidth\  \log_2\left(1+\frac{\txpower\wavelength^2}{16\pi^2\noisepower\buildingWidth^2}\right)\left(\timemaxx-\timemaxy\right)\\
                                    & =\bandwidth\  \log_2\left(1+\frac{\txpower\wavelength^2}{16\pi^2\noisepower\buildingWidth^2}\right)\frac{\buildingWidth-\buildingHeight}{\maxuavspeed}.
                                    \label{eq:total_data_y}
                                \end{salign}
                                \blt[along the x] The amount of transferred data when the UAV follows $\trajectoryxaxis$ is given by
                                \begin{salign}[eq:dtildexupperbound]
                                    ~&\dataTransferHor(\timemaxx)\\
                                    \define & \int_{0}^{\timemaxx}\ueratex(\timeInstantAux)d\timeInstantAux\\
                                    > & \int_{0}^{\timemaxx}\bandwidth\log_2\left(\frac{\txpower\wavelength^2}{16\pi^2\noisepower\left\|\loc(\timeInstantAux)-\locue\right\|^210^{\absorptionx\|\loc(\timeInstantAux)-\locue\|/10}}\right)d\timeInstantAux\\
                                    \alignchar
                                    \begin{intermed}
                                        = \int_{0}^{\timemaxx}\bandwidth\bigg(\log_2\left(\frac{\txpower\wavelength^2}{16\pi^2\noisepower}\right) - \log_2\left(\left\|\loc(\timeInstantAux)-\locue\right\|^2\right)
                                    \end{intermed}
                                    \jumpline
                                    \alignchar
                                    \begin{intermed}
                                        \nonumber
                                        \hspace{1.5cm}-\log_2\left(10^{\absorptionx\|\loc(\timeInstantAux)-\locue\|/10}\right)\bigg)d\timeInstantAux
                                    \end{intermed}
                                    \jumpline
                                    \alignchar
                                    \begin{intermed}
                                        = \bandwidth\log_2\left(\frac{\txpower\wavelength^2}{16\pi^2\noisepower}\right)\timemaxx - \bandwidth\int_{0}^{\timemaxx}\log_2\left(\left\|\loc(\timeInstantAux)-\locue\right\|^2\right)d\timeInstantAux
                                    \end{intermed}
                                    \jumpline
                                    \alignchar
                                    \begin{intermed}
                                        \nonumber
                                        \hspace{1.5cm}-\bandwidth\frac{\log_210}{10}\absorptionx\int_{0}^{\timemaxx}\|\loc(\timeInstantAux)-\locue\| d\timeInstantAux
                                    \end{intermed}
                                    \jumpline
                                    = & \bandwidth\log_2\left(\frac{\txpower\wavelength^2}{16\pi^2\noisepower}\right)\timemaxx - \dataTransferHorAux(\timemaxx) - \absorptionx\dataTransferHorAuxAux(\timemaxx)\\
                                    \define & \dataTransferHorApprox(\timemaxx),
                                    \label{eq:total_data_x}
                                \end{salign}
                                where
                                \begin{salign}
                                    \dataTransferHorAux(\timemaxx) & \define\bandwidth\int_{0}^{\timemaxx}\log_2\left(\left\|\loc(\timeInstantAux)-\locue\right\|^2\right)d\timeInstantAux
                                    \\
                                    \dataTransferHorAuxAux(\timemaxx) & \define\bandwidth\frac{\log_210}{10}\int_{0}^{\timemaxx}\|\loc(\timeInstantAux)-\locue\| d\timeInstantAux.
                                \end{salign}
                                \blt[D's computation]Since $\loc(\timeInstant)=[\maxuavspeed\timeInstant,0]\transpose$, it can be seen that
                                \begin{salign}
                                    ~&\dataTransferHorAux(\timemaxx)\\
                                    =& \frac{\bandwidth\log_2e}{\maxuavspeed}\bigg(\left(\maxuavspeed\timemaxx-\buildingWidth\right)\log\left(\left(\maxuavspeed\timemaxx-\buildingWidth\right)^2+\buildingHeight^2\right)\\
                                    \nonumber
                                    &\hspace{2cm} +2\buildingHeight\tan^{-1}\left(\frac{\maxuavspeed\timemaxx-\buildingWidth}{\buildingHeight}\right)-2\maxuavspeed\timemaxx\bigg)\\
                                    \nonumber
                                    & -\frac{\bandwidth\log_2e}{\maxuavspeed}\bigg(\left(-\buildingWidth\right)\log\left(\left(-\buildingWidth\right)^2+\buildingHeight^2\right)
                                    \\
                                    \nonumber
                                    &\hspace{5cm} +2\buildingHeight\tan^{-1}\left(\frac{-\buildingWidth}{\buildingHeight}\right)\bigg)
                                    \jumpline
                                    \alignchar
                                    \begin{intermed}
                                        = -2\bandwidth\timemaxx\log_2e
                                    \end{intermed}
                                    \jumpline
                                    \alignchar
                                    \begin{intermed}
                                        \nonumber
                                        \hspace{1cm} +\frac{\bandwidth\log_2e}{\maxuavspeed}\bigg(\buildingWidth\log\left(\buildingWidth^2+\buildingHeight^2\right)-2\buildingHeight\tan^{-1}\left(\frac{-\buildingWidth}{\buildingHeight}\right)\bigg)
                                    \end{intermed}
                                    \jumpline
                                    \alignchar
                                    \begin{intermed}
                                        = \frac{\bandwidth\log_2e}{\maxuavspeed}\buildingWidth\log\left(\buildingWidth^2+\buildingHeight^2\right) -2\bandwidth\timemaxx\log_2e
                                    \end{intermed}
                                    \jumpline
                                    \alignchar
                                    \begin{intermed}
                                        \nonumber
                                        \hspace{1cm} -2\frac{\bandwidth\log_2e}{\maxuavspeed}\buildingHeight\tan^{-1}\left(\frac{-\buildingWidth}{\buildingHeight}\right)
                                    \end{intermed}
                                    \jumpline
                                    \alignchar
                                    \begin{intermed}
                                        = \bandwidth\timemaxx\log_2e\log\left(\buildingWidth^2+\buildingHeight^2\right) -2\bandwidth\timemaxx\log_2e + 2\bandwidth\timemaxy\log_2e\tan^{-1}\left(\frac{\buildingWidth}{\buildingHeight}\right),
                                    \end{intermed}\\
                                    =& \bandwidth\timemaxx\log_2\left(\buildingWidth^2+\buildingHeight^2\right) -2\bandwidth\timemaxx\log_2e
                                    \\
                                    \nonumber
                                    & \hspace{4cm}+ 2\bandwidth\timemaxy\log_2e\tan^{-1}\left(\frac{\buildingWidth}{\buildingHeight}\right)
                                    \\
                                    =&\dataTransferHorAux
                                \end{salign}
                                and
                                \begin{salign}
                                    ~&\dataTransferHorAuxAux(\timemaxx)\\
                                    =&\bandwidth\frac{\log_210}{20\maxuavspeed}\bigg(\left(\maxuavspeed\timemaxx-\buildingWidth\right)\sqrt{\left(\maxuavspeed\timemaxx-\buildingWidth\right)^2+\buildingHeight^2}\\
                                    \nonumber
                                    &\hspace{1.5cm}-\buildingHeight^2\tanh^{-1}\left(\frac{\buildingWidth-\maxuavspeed\timemaxx}{\sqrt{\left(\maxuavspeed\timemaxx-\buildingWidth\right)^2+\buildingHeight^2}}\right)\bigg)\\
                                    \nonumber
                                    &-\bandwidth\frac{\log_210}{20\maxuavspeed}\bigg(\left(-\buildingWidth\right)\sqrt{\left(-\buildingWidth\right)^2+\buildingHeight^2} \\
                                    \nonumber
                                    &\hspace{3.5cm}-\buildingHeight^2\tanh^{-1}\left(\frac{\buildingWidth}{\sqrt{\buildingWidth^2+\buildingHeight^2}}\right)\bigg)
                                    \jumpline
                                    \alignchar
                                    \begin{intermed}
                                        = \bandwidth\frac{\log_210}{20\maxuavspeed}\bigg(\buildingWidth\sqrt{\buildingWidth^2+\buildingHeight^2} +\buildingHeight^2\tanh^{-1}\left(\frac{\buildingWidth}{\sqrt{\buildingWidth^2+\buildingHeight^2}}\right)\bigg)
                                    \end{intermed}\\
                                    %
                                    =& \bandwidth\frac{\log_210}{20\maxuavspeed}\bigg(\buildingWidth\sqrt{\buildingWidth^2+\buildingHeight^2} +\buildingHeight^2\tanh^{-1}\left(\frac{\buildingWidth}{\sqrt{\buildingWidth^2+\buildingHeight^2}}\right)\bigg)
                                    \\
                                    =&\dataTransferHorAuxAux.
                                \end{salign}

                                \blt[inequality]From \eqref{eq:dtildexupperbound},  $\dataTransferHorApprox(\timemaxx)$ is a lower bound of $\dataTransferHor(\timemaxx)$. It can also  be seen that $\dataTransferHorApprox(\timemaxx)$ is an upper bound for  $\dataTransferVer(\timemaxx)$.
                                To this end, note from \eqref{eq:alphadatdef} that
                                \begin{align}
                                        & \absorptionx < \absorptiondat=\frac{1}{\dataTransferHorAuxAux}\Bigg(\bandwidth\log_2\left(\frac{\txpower\wavelength^2}{16\pi^2\noisepower}\right)\frac{\buildingWidth}{\maxuavspeed}          \\
                                    \nonumber
                                        & \hspace{1cm}- \bandwidth\log_2\left(1+\frac{\txpower\wavelength^2}{16\pi^2\noisepower\buildingWidth^2}\right)\frac{\buildingWidth-\buildingHeight}{\maxuavspeed} - \dataTransferHorAux\Bigg).
                                \end{align}
                                Rearranging terms, one finds that
                                \begin{salign}
                                    \dataTransferVer(\timemaxx)=
                                    &
                                    \bandwidth\log_2\left(1+\frac{\txpower\wavelength^2}{16\pi^2\noisepower\buildingWidth^2}\right)\frac{\buildingWidth-\buildingHeight}{\maxuavspeed} \\
                                    <& \bandwidth\log_2\left(\frac{\txpower\wavelength^2}{16\pi^2\noisepower}\right)\frac{\buildingWidth}{\maxuavspeed} - \dataTransferHorAux - \absorptionx\dataTransferHorAuxAux
                                    \\
                                    =&                       \dataTransferHorApprox(\timemaxx)                        \end{salign}
                                Thus, to sum up,
                                \begin{align}
                                    \dataTransferVer(\timemaxx) \overset{}{<} \dataTransferHorApprox(\timemaxx) \overset{}{<} \dataTransferHor(\timemaxx)
                                \end{align}
                                and, as a result,
                                $\trajectorydat=\trajectoryxaxis$.
                            \end{bullets}
                        \end{bullets}
                    \end{IEEEproof}
                \end{extendedonly}
            \end{bullets}
        \end{bullets}
    \end{bullets}
    In words, when $\absorptionx$ takes intermediate values, the optimal UAV trajectory depends on whether one adopts \eqref{eq:timeconnectobj} or \eqref{eq:cumrateobj}.

    \begin{extendedonly}
        \blt[simulation]
        \begin{bullets}%
            \blt[parameters] A numerical example will be presented to illustrate Theorem~\ref{thm:diff_directions}. The simulation parameters are listed on Table~\ref{simparams}.
            \begin{table}[!t]
                \begin{center}
                    \caption{}
                    \label{simparams}
                    \begin{tabular}{ |c|c|c| }
                        \hline
                        \textbf{Notation}                        & \textbf{Physical meaning}       & \textbf{Simulation value} \\
                        \hline
                        $\buildingWidth \times \buildingHeight $ & Building dimensions             & 40 $\times$ 27 m          \\
                        \hline
                        $\spaceToBuilding$                       & Distance from the building edge & $10^{-3}$ m               \\
                        \hline
                        $\maxuavspeed$                           & Maximum UAV speed               & 4 m/s                     \\
                        \hline
                        $\txpower$                               & UAV transmit power              & 0.1 W                     \\
                        \hline
                        $f$                                      & Carrier frequency               & 6 GHz                     \\
                        \hline
                        $\bandwidth$                             & Bandwidth                       & 20 MHz                    \\
                        \hline
                        $\noisepower$                            & Noise power                     & -97 dBm                   \\
                        \hline
                        $\minuerate$                             & Minimum required UE rate        & 145 Mbps                  \\
                        \hline
                    \end{tabular}
                \end{center}
            \end{table}
            ~They were  chosen so that \eqref{eq:cond-minuerate} holds. As a result, one has  $\absorptionttc=0.59$ and $\absorption_{\text{dat}}=0.75$.

            \blt[figure] Fig.~\ref{fig:ttc_vs_data} shows the UE rate and the amount of transferred data over time when the UAV follows $\trajectoryxaxis$ and $\trajectoryyaxis$.
            \begin{figure}[t]
                \centering
                \begin{subfigure}[t]{.45\textwidth}
                    \centering
                    \includegraphics[width=\textwidth]{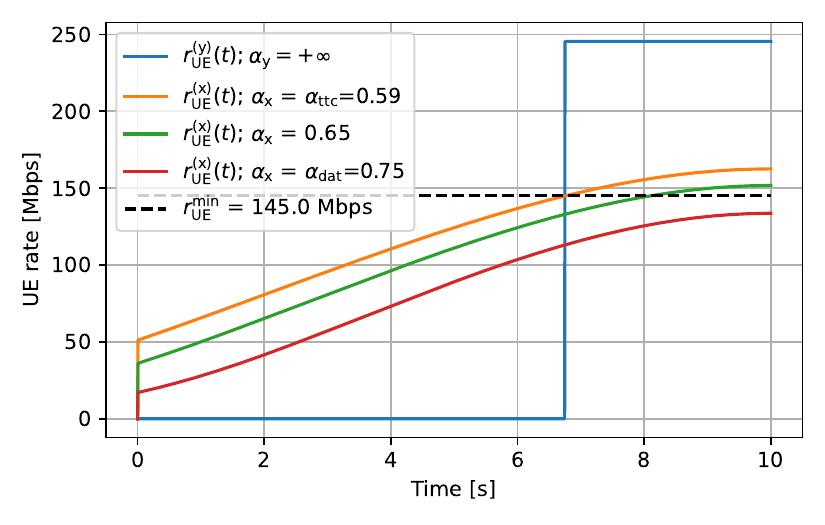}
                    \caption{}
                    \label{fig:sim_rate}
                \end{subfigure}
                \hfill
                \begin{subfigure}[t]{.45\textwidth}
                    \centering
                    \includegraphics[width=\textwidth]{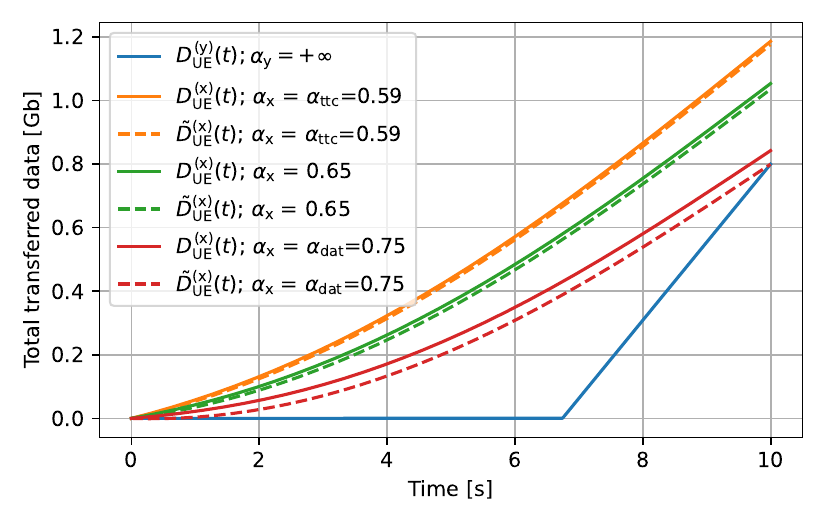}
                    \caption{}
                    \label{fig:sim_data}
                \end{subfigure}
                \caption{UE rate and the amount of transferred data over time; $\minuerate=145$ Mbps.}
                \label{fig:ttc_vs_data}
            \end{figure}
            \blt[statement] It is observed that:

            \begin{itemize}
                \item If $\absorptionx>\absorptionttc$, then $\trajectoryyaxis$ minimizes the connection time ($\trajectoryttc = \trajectoryyaxis$); cf. Fig.~\ref{fig:sim_rate}.
                \item If $\absorptionx<\absorptiondat$, then $\trajectoryxaxis$ maximizes the amount of transferred data ($\trajectorydat = \trajectoryxaxis$); cf. Fig.~\ref{fig:sim_data}.
                \item When $\absorptionttc<\absorptionx<\absorptiondat$, one has that $\trajectoryttc = \trajectoryyaxis$ and $\trajectorydat = \trajectoryxaxis$. This is precisely as predicted by
                        Theorem~\ref{thm:diff_directions}.
            \end{itemize}
        \end{bullets}%
    \end{extendedonly}

\end{bullets}%